\documentclass[11pt, reqno]{amsart}
\usepackage{a4wide}
\usepackage{wrapfig}
\usepackage{amsmath}
\usepackage{amsthm}
\usepackage{mathrsfs}
\usepackage{color}
\usepackage{xcolor}
\usepackage{graphicx}
\usepackage{amssymb}
\usepackage{url}
\usepackage{multicol}
\usepackage{tikz}
\usepackage{amsmath}
\usepackage{fancyhdr}
\usepackage{eufrak}
\usetikzlibrary{matrix}
\usetikzlibrary{arrows}
\usepackage{amssymb}
\usepackage{subfig}
\usepackage{hyperref}
\usepackage{listings}
\usepackage{enumitem}
\usepackage{varwidth}
\allowdisplaybreaks
\def\id{{\fontsize{.85em}{1.1em}\selectfont1}\normalfont\kern-.8ex1}

\title[Asymptotic stability]{On asymptotic stability of the sine-Gordon kink in the energy space}

\author[M.A. Alejo]{Miguel A. Alejo$^*$}
\address{Departamento de Matem\'atica. Universidade Federal de Santa Catarina-UFSC.\\ 
Campus Trindade,  88040-900. Florian\'opolis-SC, Brazil.}
\email{miguel.alejo@ufsc.br}
\thanks{$^{*}$ M.A.  Alejo was partially supported by  CNPq grant no. 305205/2016-1.}

\author[C. Mu\~noz]{Claudio Mu\~noz$^{**}$}
\address{CNRS and Departamento de Ingenier\'{\i}a Matem\'atica and Centro
de Modelamiento Matem\'atico (UMI 2807 CNRS), Universidad de Chile, Casilla
170 Correo 3, Santiago, Chile.}
\email{cmunoz@dim.uchile.cl}
\thanks{$^{**}$ C. M. work was funded in part by Chilean research grants FONDECYT 1191412, project France-Chile
ECOS-Sud C18E06 and CMM Conicyt PIA AFB170001.}

\author[J.M. Palacios]{Jos\'e M. Palacios$^{***}$}
\thanks{$^{***}$ J. M. P.  was partially supported by Chilean research grants FONDECYT 1191412 and project France-Chile
ECOS-Sud C18E06.}
\address{Institut Denis Poisson, Universit\'e de Tours, Universit\'e d'Orleans, CNRS, Parc Grandmont 37200, Tours, France}
\email{jose.palacios@etu.univ-tours.fr}

\date{\today}

\newcommand{\be}{\begin{equation}}
\newcommand{\ee}{\end{equation}}
\newcommand{\bp}{\begin{proof}}
\newcommand{\ep}{\end{proof}}
\newcommand{\bel}{\begin{equation}\label}
\newcommand{\eeq}{\end{equation}}
\newcommand{\bea}{\begin{eqnarray}}
\newcommand{\eea}{\end{eqnarray}}
\newcommand{\bee}{\begin{eqnarray*}}
\newcommand{\eee}{\end{eqnarray*}}
\newcommand{\ben}{\begin{enumerate}}
\newcommand{\een}{\end{enumerate}}

\providecommand{\abs}[1]{\lvert#1 \rvert}
\newcommand{\ve}{\varepsilon}

\newcommand{\R}{\mathbb{R}}

\newcommand{\Com}{\mathbb{C}}
\newcommand{\la}{\lambda}

\newcommand{\al}{\alpha}
\newcommand{\bt}{\beta}
\newcommand{\ga}{\gamma}

\newcommand{\sech}{\operatorname{sech}}

\newtheorem{thm}{Theorem}[section]
\newtheorem{cor}[thm]{Corollary}
\newtheorem{lem}[thm]{Lemma}
\newtheorem{prop}[thm]{Proposition}

\newtheorem{defn}[thm]{Definition}

\theoremstyle{remark}
\newtheorem{rem}{Remark}[section]

\definecolor{codegreen}{rgb}{0,0.6,0}
\definecolor{codegray}{rgb}{0.5,0.5,0.5}
\definecolor{codepurple}{rgb}{0.58,0,0.82}
\definecolor{backcolour}{rgb}{0.95,0.95,0.92}

\lstdefinestyle{mystyle}{
	backgroundcolor=\color{backcolour},   
	commentstyle=\color{codegreen},
	keywordstyle=\color{magenta},
	numberstyle=\tiny\color{codegray},
	stringstyle=\color{codepurple},
	basicstyle=\footnotesize,
	breakatwhitespace=false,         
	breaklines=true,                 
	captionpos=b,                    
	keepspaces=true,                 
	numbers=left,                    
	numbersep=5pt,                  
	showspaces=false,                
	showstringspaces=false,
	showtabs=false,                  
	tabsize=2
}
\numberwithin{equation}{section}

\usepackage{pgfplots}
\pgfplotsset{compat=newest}



\theoremstyle{definition}

\numberwithin{ej}{section}

\begin{document}





\renewcommand{\sectionmark}[1]{\markright{\thesection.\ #1}}
\renewcommand{\headrulewidth}{0.5pt}
\renewcommand{\footrulewidth}{0.5pt}


\begin{abstract}
We consider the sine-Gordon (SG) equation in 1+1 dimensions. The kink is a static, \emph{non symmetric} exact solution to SG, stable in the energy space $H^1\times L^2$. It is well-known that the linearized operator around the kink has a simple kernel and no internal modes. However, it possesses an \emph{odd resonance} at the bottom of the continuum spectrum, deeply related to the existence of the (in)famous \emph{wobbling kink}, an explicit periodic-in-time solution of SG around the kink that contradicts the asymptotic stability of the kink in the energy space. 

%

\smallskip

In this paper we further investigate the influence of resonances in the asymptotic stability question. We also discuss the relationship between breathers, wobbling kinks and resonances in the SG setting. By gathering B\"acklund transformations (BT) as in \cite{HW,MP} and Virial estimates around odd perturbations of the vacuum solution, in the spirit of \cite{KMM2}, we first identify the manifold of initial data around zero under which BTs are related to the wobbling kink solution. It turns out that (even) small breathers are deeply related to odd perturbations around the kink, including the wobbling kink itself. As a consequence of this result and \cite{KMM2}, using BTs we can construct a smooth manifold of initial data close to the kink, for which there is asymptotic stability in the energy space. The initial data has spatial symmetry of the form (kink + odd, even), non resonant in principle, and not preserved by the flow. This asymptotic stability property holds despite the existence of wobbling kinks in SG. We also show that wobbling kinks are orbitally stable under odd data, and clarify some interesting connections between SG and $\phi^4$ at the level of linear B\"acklund transformations.
\end{abstract}
\maketitle \markboth{Asymptotic stability of SG kinks} {M.A. Alejo, C. Mu\~noz and J.M. Palacios}
\tableofcontents

\section{Introduction and Main results}

 Consider the 1+1 dimensional sine-Gordon (SG) equation, in physical coordinates $(t,x)$, for a scalar field $\phi$:
\be\label{sg1}
\phi_{tt}-\phi_{xx}+\sin \phi=0. 
\ee 
Here, $\phi =\phi(t,x)$ is a real-valued function, and $(t,x)\in \R^2$. A natural energy space for \eqref{sg1} is given by
\[
(H_{sin}^1 \times L^2)(\R) := \left\{ \vec\phi := (\phi,\phi_t) \in (\dot H^1 \times L^2) (\R)  ~  : ~ \sin \left(\frac{\phi}{2}\right)  \in L^2 (\R)\right\},
\]
where we use the standard notation $\vec\phi :=(\phi,\phi_t)$, corresponding to a wave-like dynamics. This fact essentially follows form the lower order conservation laws called energy and momentum, respectively:
\be\label{Energy}
E[\vec\phi](t)=\dfrac{1}{2}\int_\mathbb{R}(\phi_x^2+\phi_t^2)(t,x)dx+\int_{\mathbb{R}}(1-\cos \phi(t,x))dx=E[\vec\phi](0), 
\ee
and
\be\label{Momentum}
P[\vec\phi](t)=\dfrac{1}{2}\int_\mathbb{R}\phi_t(t,x)\phi_x(t,x)dx=P[\vec\phi](0).
\ee
Real-valued solutions of \eqref{sg1} that initially are in $H^1_{sin}\times L^2$ are preserved for all time, see e.g \cite{DG} and \cite{MP}. Additionally, they are globally well-defined {\color{black}thanks to the fact that $\sin(\cdot)$ is a smooth bounded function}. In what follows, we will assume that we have a real-valued solution of \eqref{sg1} (in vector form) $\vec\phi \in C(\R; H_{sin}^1\times L^2)$. Additionally, small perturbations of a given solution in $H_{sin}^1\times L^2$ are essentially in $H^1\times L^2$, and vice-versa. 

\medskip

Solutions of \eqref{sg1} {\color{black}are known to satisfy} several symmetry properties: shifts in space and time $(t_0,x_0)$, i.e. the mapping $\vec\phi (t,x) \mapsto \vec\phi(t+t_0,x+x_0) $ among SG solutions is preserved, as well as \emph{Lorentz boosts}: for each $\bt\in (-1,1)$, given $\vec\phi(t,x)=(\phi,\phi_t)(t,x)$ solution, then 
\be\label{Lorentz}
(\phi,\phi_t)_\bt(t,x):= (\phi,\phi_t)\big(\ga(t-\bt x),\ga(x-\bt t) \big), \quad \ga:=(1-\bt^2)^{-1/2},
\ee
is another solution of \eqref{sg1}. The parameter $\ga$ is called Lorentz scaling factor, having an important role in the Physics of SG, and in what follows.

\medskip

As for the motivation for studying SG, this equation has been extensively used in differential geometry  (constant negative curvature surfaces), as well as relativistic field theory and soliton integrable systems. The interested reader may consult the monograph by Lamb \cite[Section 5.2]{Lamb}, 
and for more details about the physics of SG, see e.g. Dauxois and Peyrard \cite{DP}, and the recent monographs \cite{JK,KC}.

\medskip

SG has particular (topological) stationary solutions, known as \emph{kinks} \cite{Lamb}:
\be\label{Kink}
Q(x) := 4\arctan e^{x}.
\ee
This exact solution connects the final states 0 and $2\pi$. Thanks to Lorentz \eqref{Lorentz} and translation invariances, it is possible to define a kink of arbitrary speed $\bt\in (-1,1)$ and shift $x_0\in\R$, given by 
\be\label{Kink_lorentz}
Q(t,x;\bt,x_0) := 4\arctan (e^{\ga(x-\beta t+x_0)}), \qquad \ga=(1-\bt^2)^{-1/2}.
\ee
From the integrability of SG \cite{AKNS,ZS}, interactions between kinks are elastic, i.e. they are ``solitons'' in the strict sense of the word \cite{Lamb}. 
Also, $-Q(x)$ is another stationary solution of SG, usually called \emph{anti-kink}. 

\medskip

It is well-known that $(Q,0)$ is \emph{orbitally stable} under small perturbations in the energy space $(H^1\times L^2)(\R)$, see Henry-Perez-Wreszinski \cite{HPW}. More precisely, there exists $C_0>0$ such that, for all sufficiently small $\eta>0$,
\be\label{Orbital}
\begin{aligned}
&\| (\phi,\phi_t)(t=0) - (Q,0)\|_{H^1\times L^2} <\eta \\
& \qquad \qquad  \implies \sup_{t\in\R} \| (\phi,\phi_t)(t) - (Q,0)(\cdot-y(t))\|_{H^1\times L^2} <C_0\eta,
\end{aligned}
\ee
for some $y(t)\in \R$. Using the B\"acklund transformation present for SG, and extensively mentioned below, Hoffman and Wayne \cite{HW} extended this stability result to the case of the kink and sketched the case of several kink structures. Inspired by this work, and using the same technique, in a recent work \cite{MP} the three main 2-soliton solutions of SG were proved to be \emph{orbitally stable} for small perturbations in the energy space. 
In that paper, 2-kinks solutions \eqref{R_0} were considered, but also breathers (see \eqref{breather} below) and kink-antikinks, two additional 2-soliton solutions which are \emph{even} in space. All of them were shown to be orbitally stable for small perturbations in $H^1\times L^2$. 

\medskip

In this paper we consider the \emph{asymptotic stability}  (AS) problem for the SG kink in the energy space. More precisely, we would like to understand the possible final states allowed by \eqref{Orbital}. As we will explain below, this is not a simple problem, because of several intriguing ingredients. Our main result, stated in few words, claims the following.

\begin{thm}\label{MT00} There exists a smooth infinite codimensional manifold $\mathcal M_{\eta,0}$ of initial data $(\phi_0,\phi_1)$ of the form
\be\label{initial_data0}
 (\phi_0,\phi_1)= (Q + u_0, s_0), \quad u_0 \hbox{ odd}, \; s_0 \hbox{ even}, \quad \Vert (u_0, s_0) \Vert_{H^1\times L^2} <\eta\ll1,
\ee
of zero momentum \eqref{Momentum}, under which the SG kink $Q$ in \eqref{Kink} is asymptotically stable in the energy space.
\end{thm}

What do we mean by \emph{asymptotically stable} in this setting, and what kind of manifold are we talking about, is something that we have to explain in detail, but it requires the introduction of several additional ingredients. These ingredients are the so-called wobbling kinks, breathers, (spectral) resonances and B\"acklund transformations,  and we deeply think that they are certainly necessary to fully understand Theorem \ref{MT00}. A key element for the proof of Theorem \ref{MT00} is to understand how spatial parity properties relate under B\"acklund transformations, a subject left out in our previous paper \cite{MP}, and schematically explained in Figs. \ref{Fig:1} and \ref{Fig:2}. The impatient reader can directly go to Theorem \ref{MT0} to read a detailed description of our main result. 

\subsection{Wobbling kinks}\label{1p1}
Proving Theorem \ref{MT00} is not direct, essentially because of the existence, near the static kink, of arbitrarily close \emph{wobbling kinks} in SG \cite{Segur1,Segur2}, \cite[Thm. 2.6]{CQS} (see also references therein and \cite[Remark 1.3]{KMM}).

\medskip

Recall the kink \eqref{Kink}. Wobbling kinks are explicit solutions $W_\beta=W_\beta(t,x)$, $\beta\in (-1,1)$, to the SG equation \eqref{sg1}, which behave as periodic in time, localized perturbations of the static kink solution:\footnote{Note also that $W_\beta$ is defined using the multi-valued, complex-valued function $\hbox{Arg}$, in order to avoid undesirable jumps obtained by using $\arctan$.}
\be\label{wobbling}
\begin{aligned}
W_\beta(t,x):= &~{} 4~\hbox{Arg} \left( U_\beta + i V_\beta\right),\\
 U_\beta:=&~{}  \cosh(\beta x) + \beta \sinh (\beta x)  -\beta e^{ x} \cos(\alpha t) \\
 V_\beta:= &~{} e^x \left(\cosh(\beta x) -\beta \sinh (\beta x)  -\beta e^{-x} \cos(\alpha t)  \right) , \quad \al:=\sqrt{1-\beta^2}.
 \end{aligned}
\ee
See Fig. \ref{Fig:3} for a graphic depiction of this solution. Formally, wobbling kinks are solutions of the form \emph{kink $+$ breather}, where a breather is a periodic in time solution of SG, for reasons to be explained below. Note also that $W_\beta$ reduces to the SG kink \eqref{Kink} as $\beta\to 0$. By construction, when $\beta\neq 0$, these modes never converge to a final state, no matter how close they are to the kink $Q$. Therefore, as already stated in \cite[Remark 1.3]{KMM}, \emph{SG kinks are not asymptotically stable in the energy space}. 

\begin{figure}[h!]
   \centering
   \includegraphics[scale=0.5]{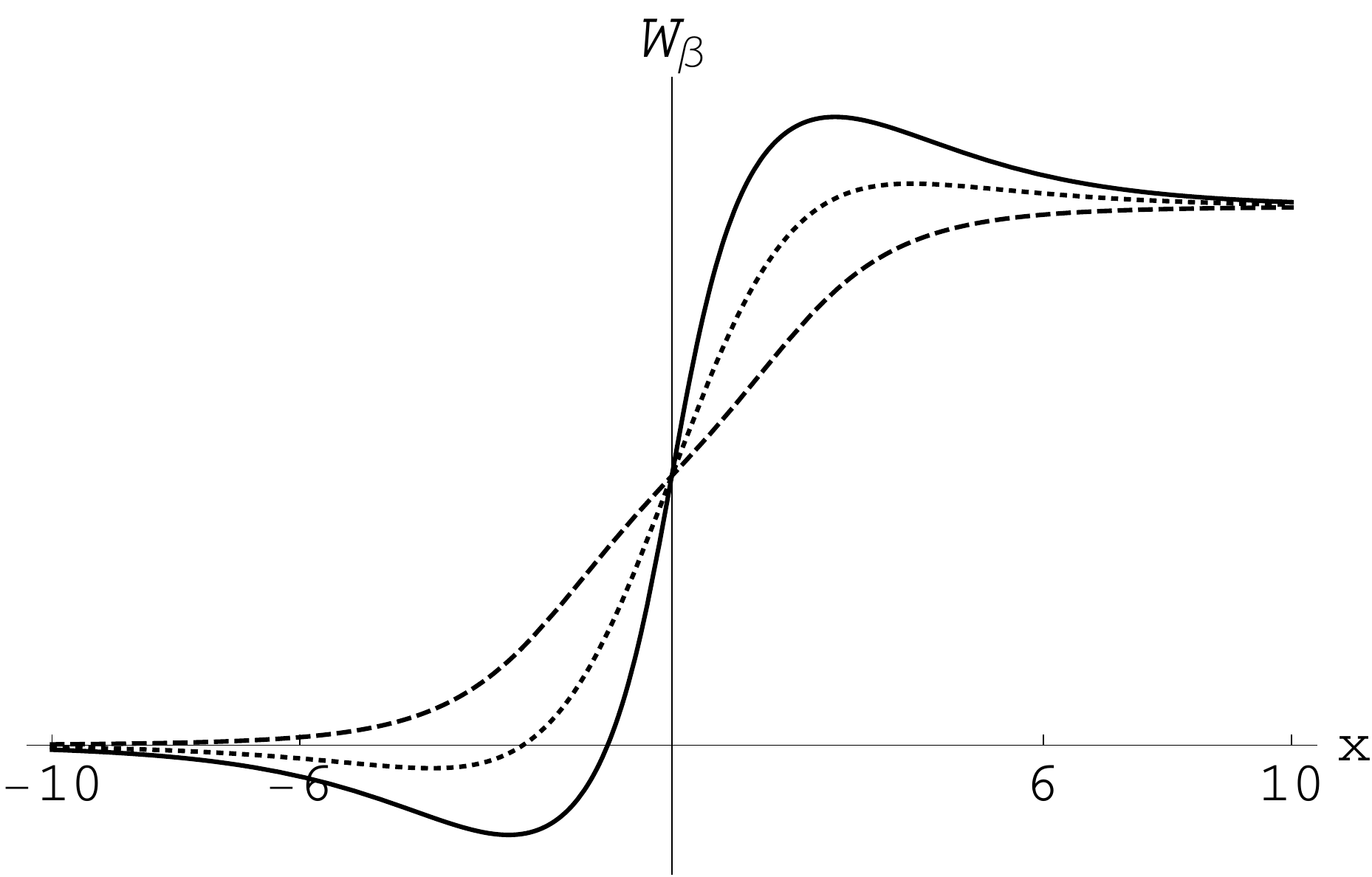}
   \caption{The wobbling kink \eqref{wobbling} with $\beta=0.5$ at times $t=0$ (continuous curve), $t=2$ (dashed curve), and $t=6$ (dotted curve).}
   \label{Fig:3}
\end{figure}

\medskip

Consequently, any result concerning the long time behavior of SG kinks (see Theorem \ref{MT00}) will require to take into account these counterexamples (and probably others) to the existence of final states.

\medskip

For further purposes, we will need the following standard notation: for $m=0,1,2,\ldots,$ denote
\be\label{Espacios_paridad}
\begin{aligned}
H^m_e(\R) := &~{} \{ f \in H^m(\R) ~  : ~  f \hbox{ is even in space}  \},\\
H^m_o(\R) := &~{} \{ f \in H^m(\R) ~  : ~  f \hbox{ is odd in space}  \}.
\end{aligned}
\ee
As usual, we denote $L^2_e(\R)=H^0_e(\R)$ and $L^2_o(\R)=H^0_o(\R)$. Our second result of this paper is related to the orbital stability of the wobbling kink, under odd perturbations. 
 
\begin{thm}\label{WK_orbital}
The SG wobbling kink is orbitally stable under small $H^1_o\times L^2_o$ perturbations. 
\end{thm}

A more quantitative version of this result is given in Theorem \ref{WK_orbital0}. Whether or not the wobbling kink is orbitally stable under general perturbations depends on the definition of \emph{wobbling kink solution}. Precisely, for some particular initial data one can see (see Lemma \ref{3soliton}) that the wobbling kink structure as itself (periodic in time, odd perturbations of a kink) is destroyed; however, this is because the wobbling kink \eqref{wobbling} is part of a more general family of topological 3-soliton solutions consisting of a kink and an attached static/moving breather. This phenomenon is similar to the case of NLS breathers/2-solitons, which are part of a whole family, see e.g. \cite{AFM} for details. The stability of the whole 3-soliton family remains an interesting open problem.

\medskip

SG can be also described using Inverse Scattering Techniques (IST) (recall that SG is an integrable model \cite{AKNS,ZS,Lamb}). Some spatial decay hypotheses are needed to define the associated scattering data (or Riemann-Hilbert problem), and data only the energy space are not well-suited for those methods. 
Also, the dynamics around kinks is usually not treated because of its unusual limit at infinity. Therefore, a description of the (wobbling) kink dynamics as in Theorems \ref{MT00} and \ref{WK_orbital} for data only in the energy space is far from obvious, and not known as far as we understand. However, the IST description, when made rigorous, is far more accurate than ours. The interested reader can consult the recent monograph by Klein and Saut \cite{KS_ISTPDE} for a complete description of this fascinating topic on IST vs. PDE techniques. The integrable character of SG was proved in \cite{AKNS,ZS}. Some early descriptions of the dynamics can be found in Ercolani, Forest and McLaughlin, \cite{EFM}. Birnir, McKean and Weinstein \cite{BMW} studied nonexistence of breathers for perturbations of SG formally using B\"acklund transformations. Denzler \cite{D} improved this result by considering more nonlinearities. See also Vuillermot \cite{Vuillermot} and Kichenassamy \cite{Kich}, and the monograph by Schuur \cite{Sch} for more details on the methods. 
See also \cite{Mas} for a recent construction of invariant soliton manifold for perturbed SG equations. A completely rigorous result on nonexistence of odd breathers can be found in \cite{KMM2}. 

\medskip

The wobbling kink in SG \eqref{wobbling} was first discovered by Segur \cite{Segur1} 
(see also \cite{Segur2}), while searching for wobbling kink solutions for $\phi^4$ (see Section \ref{2}). Using Inverse Scattering 
techniques and a permutability theorem \cite{Lamb}, the wobble \eqref{wobbling} is easily found as a solution consisting of a static 
kink plus an attached breather, exactly as expressed in \eqref{wobblingK}. The same procedure for the $\phi^4$ model \eqref{phi4} seems not to work (i.e. there is no wobbling kink), as the authors pointed out in \cite{Segur2}. A more rigorous proof was given in \cite{KMM}, in the case of (odd, odd) data, but for general data the question remains largely open. In this paper we answer parallel questions for the SG case, which enjoys far more algebraic properties than $\phi^4$, although they meet nicely at the linear level, see Section \ref{2}. This close connection between SG and $\phi^4$ has fascinated to plenty of authors in the mathematical physics community since past forty years; see e.g. the monographs \cite{KC,JK} for further details. In this paper, we also explore this connection in terms of the components needed for the proof of Theorems \ref{MT00} and \ref{WK_orbital}, see in particular the bridge between Theorems \ref{WK_orbital} and Theorem \ref{MT00}, which is Section \ref{5}.

\medskip

Is the wobbling kink asymptotically stable for odd data? Clearly not. Fix $\beta\in (0,1)$. Then the initial perturbation of the wobbling kink $W_\beta(t,x)$ given by $W_{\beta'}(0,x)$, with $\beta'\sim \beta$ does not converge to the wobbling kink $W_\beta$. This means that wobbling kinks are not AS. The problem of asymptotic stability of the wobbling kink for a manifold of initial data just as in Theorem \ref{MT00} remains an interesting open question.

\medskip

Another point of view under which Theorem \ref{MT00} can be put in context, is the one associated to generalized Korteweg-de Vries (gKdV), nonlinear Schr\"odinger (NLS) and Klein-Gordon (NLKG) equations and their associated soliton dynamics. We first focus on the NLKG case, closely related to SG. Soffer and Weinstein \cite{SW} successfully solved the intriguing interaction between solitons and radiation in 3D  KLKG. A complete description of the invariant manifolds around the 1D NLKG soliton for supercritical powers was also described in \cite{KNS}, recently extended in \cite{KMM3}, based in previous results by Bizo\'n et. al. \cite{BCS}. See the monograph \cite{NS1} for a complete account of the methods developed by Krieger, Nakanishi and Schlag in the case of Klein-Gordon theories in several dimensions, and motivated by earlier fundamental results in this area by Bates and collaborators \cite{BJ,BLZ}. For generalized KdV equations, see the works by Pego and Weinstein \cite{PW} and Martel and Merle \cite{MMarma,MMnon,MaMe}. Martel, Merle and Tsai \cite{MMT1} showed the stability of the sum of $N$ solitons in general gKdV equations. The recently written review paper \cite{KoMaMu2} contains a more complete description of the remaining NLS case, and of the literature around this important subject.

\medskip

If the background is not soliton like, there are also important results to mention. Delort \cite{Delort,DeF} considered the global existence and scattering of small solutions to quasilinear NLS and NLKG equations. Bambusi and Cuccagna \cite{BC} considered the NLKG dynamics around the zero state. Other recent results concerning the scattering of small solutions in NLKG equations can be found in \cite{LS,Ste}. 

\medskip

As for kink structures is referred, and their asymptotic stability, there are several works on this subject. Merle and Vega \cite{MV} showed asymptotic stability of the modified KdV kink (see also \cite{AMV,Mu}). Kopylova and Komech \cite{KK1,KK2} considered the case of kink structures in scalar field models with higher nonlinearities. The kink in the $\phi^4$ model was treated in \cite{KMM}, as previously explained. Finally, see \cite{Snelson} for the final state of a variable coefficients $\phi^4$ kink.

\medskip

The results previously proved in \cite{KMM,KMM2}, and the ones in this paper, make strong use of the parity of the initial data. Here we also consider particular parity for initial data even if it is not preserved in time. The use of parity in wave like equations is not a new subject, but it has had some increasing use in the previous years.
Kenig et al. \cite{KLLS} considered energy channels for wave equations in odd dimensions, where initial 
data of the form $(f,0)$ and $(0,g)$ were considered, much in the spirit of the generator of the manifold $\mathcal M_{\eta,0}$ considered in this paper. However, it seems here that our results are the first ones where this symmetry is not respected by the flow.

\medskip

We believe that some of the results here proved can be extended to more general solutions of SG, for instance, to the case of 2-kinks, wobbling kinks, or the so-called modified KdV kinks \cite{MV,Mu}. Concerning the first case, a 2-kink is a solution of SG that behaves as the elastic interaction between two kinks. 
In the SG case, this solution is explicit, 
and given by (see Lamb \cite[pp. 145--149]{Lamb}\footnote{Note that in our previous paper \cite[eqns. (1.6) and (1.7)]{
MP}, there is a missing $\beta$ in the definition of the 2-kink ${\color{black}\mathcal{R}}(t,x)$ and kink-antikink $A(t,x)$.}):
\be\label{R_0}
\begin{aligned}
{\color{black}\mathcal{R}}(t,x;\bt)= &~{} 4\arctan \left( \bt \frac{\sinh(\ga x)}{\cosh(\ga \beta t)} \right), \quad \bt\in (-1,1),\quad \bt \neq 0.
\end{aligned}
\ee
Here $\beta$ is the scaling factor (or speed), and $\ga=(1-\bt^2)^{-1/2}$ is the usual Lorentz factor. 
The 2-kink represents the interaction of two SG kinks with speeds $\pm\bt$, with limits as $x\to \pm \infty$ equal to $-2\pi$ and $2\pi$ 
respectively (i.e., ${\color{black}\mathcal{R}}$ does not decay to zero). Note that ${\color{black}\mathcal{R}}$ is \emph{odd} in $x$ and \emph{even} in $t$. 
Also recall that this solution was proved to be stable \cite{MP}. In another direction, the extension of Theorem \ref{MT00} 
to the case of breathers \eqref{breather} is a challenging problem, first of all, because it will be necessary to identify the correct 
perturbative manifold for decay. See also \cite{AMP1,AMP2,Mu4} for other early stability results in the case of breathers and \cite{Mu3} 
for a simple account of stability results in integrable and nonintegrable equations.

\subsection*{Organization of this article}
This article is organized as follows. Section \ref{2} presents preliminaries that we will need along this paper, in particular, resonances in $\phi^4$ and SG around kink solutions. Section \ref{3} deals with the B\"acklund transformations in the SG case. Section \ref{4} refers to the action of the BT on certain parity manifolds, and contains the proof of Theorem \ref{WK_orbital} (see Theorem \ref{WK_orbital0}). Section \ref{5} is devoted to the study of the linearized BT around the SG and $\phi^4$ kinks.  Section \ref{7} contains the construction of the initial data and the zero-momentum manifold $\mathcal M_{\eta,0}$. Section \ref{8} deals with the modulation of the evolution. Section \ref{9} concerns with the lifting of the data around zero towards the kink solution, Section \ref{10} focus on estimates on the shift parameters on the kink, and finally Section \ref{11} is devoted to the end of proof of Theorem \ref{MT0}. 

\subsection*{Acknowledgements} We thank Y. Martel for several interesting discussion along the preparation of this work. Part of this work was done while the authors were visiting the Departamento de Matem\'atica Aplicada de Granada, UGR, Spain, whose hospitality is greatly and warmly acknowledged. C. M. also acknowledges the hospitality of the Laboratoire de Math\'ematiques d'{}Orsay, and Ecole Polytechnique (France), where part of this work was done. 

\medskip

\section{Breathers and resonances in $\phi^4$ and SG}\label{2}

This section is devoted to introduce some notation and key elements for forthcoming sections. Of particular interest will be the following three ingredients: $(i)$ the introduction of the $\phi^4$ model and its spectral properties (internal modes, resonances, etc.), useful in Section \ref{5}; $(ii)$ the SG spectral problem and its connection to the wobbling kink, also useful for Section \ref{5}, and finally, $(iii)$ the SG breather and its relationship via parity manifolds with the asymptotic stability problem around the vacuum, a result from \cite{KMM2} shall play a key role on the proof of Theorem \ref{MT00} (see Theorem \ref{thmkmm2}). We start out by recalling the definition of breather.

\subsection{Breathers}

A breather is a periodic in time, localized solution of SG around zero. The most famous example of breather is given by the formula \cite{Lamb}
\be\label{breather}
B_\beta(t,x)= 4\arctan \left(\frac{\bt}{\al} \frac{\sin (\al t)}{\cosh(\bt x)} \right), \quad \al=\sqrt{1-\bt^2}, \quad \beta\neq 0, \quad \beta\in (-1,1).
\ee
See Fig. \ref{Fig:4} for a picture of the breather at different times. This solution is stable \cite{AMP1,MP}, and for $\beta$ small contradicts the asymptotic stability of the vacuum in the energy space. The reader may consult \cite{AM,AM1} for more details on breather solutions and their stability.

\begin{figure}[h!]
\begin{center}
   \includegraphics[scale=0.5]{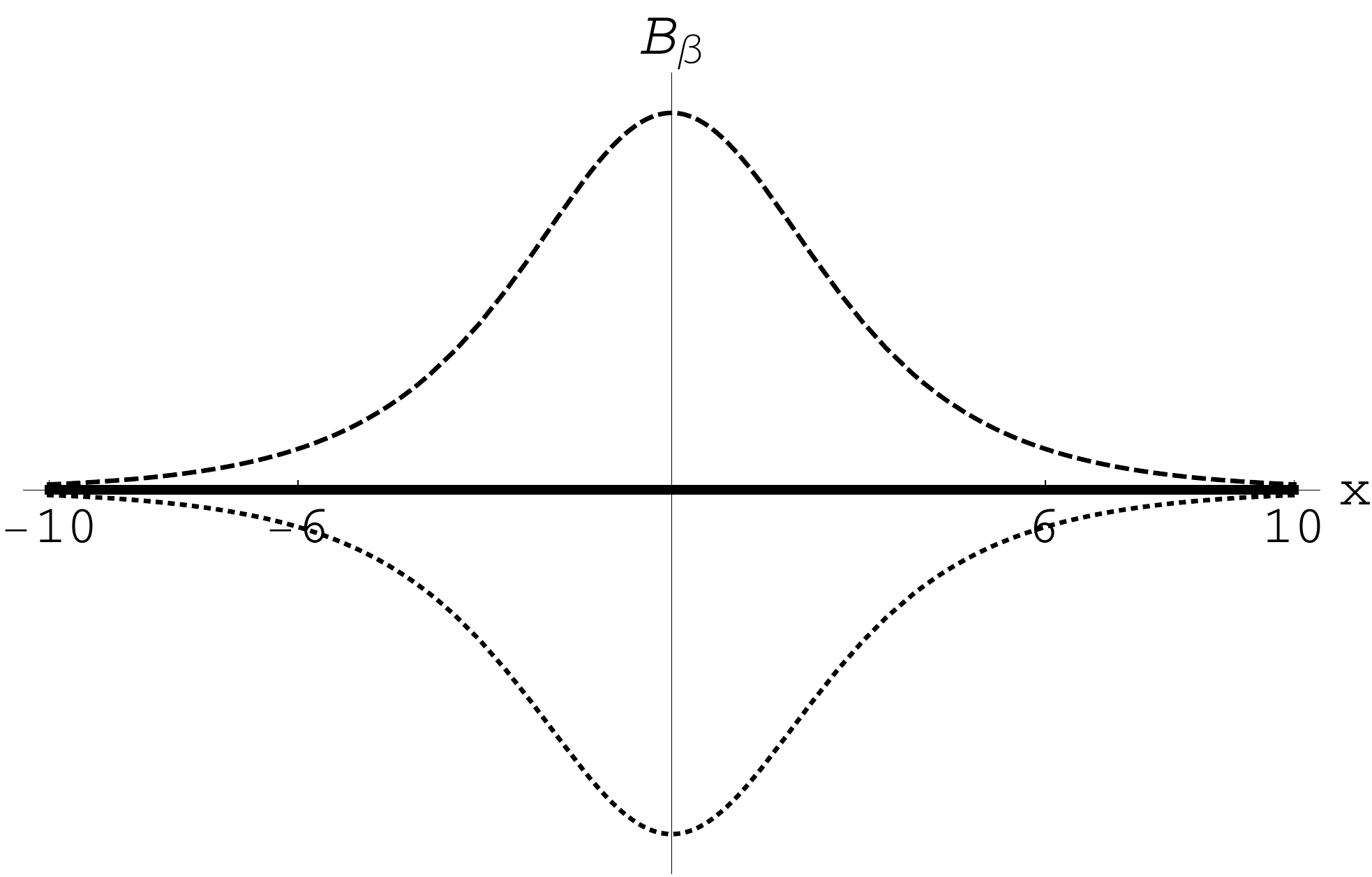}
\end{center}
\caption{Breather solution \eqref{breather} with $\beta=0.5$ at times $t=0$ (continuous curve), $t=2$ (dashed curve) and $t=6$ (dotted curve).}\label{Fig:4}
\end{figure}

\subsection{The $\phi^4$ kink  and the even resonance} A step forward towards the understanding of the long time dynamics around kink solutions in 1+1 dimensions was given in \cite{KMM}, where the authors considered \emph{odd perturbations} of the (odd) kink 
\be\label{H}
H(x)=\tanh \left(\frac{x}{\sqrt{2}} \right),
\ee
in the 1+1 dimensional $\phi^4$-model of Quantum Field Physics \cite{DP,KMM}
\begin{align}\label{phi4}
\phi_{tt}-\phi_{xx}- \phi +\phi^3=0. 
\end{align} 
This model, in its 3D version, is deeply related to the Higgs boson description \cite{Higgs}, via symmetry breaking around the global minima $|\phi|=1$. Although non integrable, $\phi^4$ is closely related to SG \eqref{sg1}. More precisely, after subtraction of $\pi$, SG solutions $\phi$ solve
\be\label{unstable}
\phi_{tt}-\phi_{xx}- \sin\phi =0,
\ee
for which $\phi^4$ \eqref{phi4} is a third order approximation, up to a suitable scaling factor. A beautiful description of the duality $\phi^4$-SG can be found in the monograph by Dauxois and Peyrard \cite{DP}, previously mentioned. In particular, many properties related to SG are also studied in $\phi^4$ and viceversa \cite{Segur1,Segur2,DP}. However, SG is integrable and $\phi^4$ is not.

\medskip

In \cite{KMM} it was proved that (under the oddness assumption on the initial data $(\phi,\phi_t)(t=0)$, which is preserved by the flow),
\be\label{AS_phi4}
\begin{aligned}
&\| (\phi,\phi_t)(t=0) - (H,0)\|_{H^1\times L^2} <\eta \ll 1 \\
& \qquad \qquad  \implies \lim_{t\to \pm\infty} \| (\phi,\phi_t)(t) - (H,0)\|_{(H^1\times L^2)(I)} =0,
\end{aligned}
\ee
for any compact interval of space $I$. This result was showed using fine \emph{virial estimates} allowing to control the existence of an \emph{internal mode} associated to the linear operator $\mathcal L_H$ around $H$:
\be\label{L_H}
\mathcal L_H:= -\partial_x^2 -1 +3H^2 =-\partial_x^2 +2 - 3\sech^2\left( \frac{x}{\sqrt{2}}\right).
\ee
Recall that an internal mode here is a positive eigenvalue below the continuum spectrum. Here the internal mode and its eigenvalue are \cite{KMM}
\be\label{Y_1}
Y_1:= \sech\left( \frac{x}{\sqrt{2}}\right)\tanh\left( \frac{x}{\sqrt{2}}\right),  \quad \la=\frac32.
\ee
The extension of the result \eqref{AS_phi4} to the case of general data is far from being simple, and remains a challenging question, mainly because of the existence of an {\bf spectral resonance} (a generalized eigenfunction of $\mathcal L_H$ in $L^\infty\backslash L^2$) at $\lambda=2$, given by 
\be\label{Reso_phi4}
\mathcal L_H \left(1-\frac32\sech^2\left(\frac{x}{\sqrt{2}}\right) \right) = 2 \left(1-\frac32\sech^2\left(\frac{x}{\sqrt{2}}\right) \right).
\ee
Moreover, this resonance is even, and that is really important for the proof in \cite{KMM}. See that work for more details.

\subsection{Wobbling kinks and the odd resonance} Coming back to SG \eqref{sg1}, and making a quick comparison with $\phi^4$, we can notice that the kink $Q$ (connecting 0 and $2\pi$) \emph{has no parity property}, and the subtraction of $\pi$ above mentioned leads to an equation (see \eqref{unstable}) which is not stable around the zero state. 

\medskip

Recall that we have said that a wobbling kink can be recast as kink + breather, and we know that breathers are even. However, this conception is a somehow misleading because of the following really surprising fact. 

\medskip

Indeed, contrary to $\phi^4$, one can notice from \eqref{wobbling} that wobbling kinks $W_\beta$ can be recast as \emph{(odd, odd)} perturbations of the SG kink $(Q,0)$. Indeed, from  \eqref{wobbling} one has (see also Fig. \ref{Fig:7})
\be\label{wobblingK}
\begin{aligned}
W_\beta(t,x) -Q(x) =&~{} 4\,\hbox{Arg}\Big( \left(  \cosh x \cosh (\beta x) -\beta \sinh x \sinh (\beta x) -\beta \cos(\alpha t) \right) \\
&~{} \qquad \quad+ i \beta \left(\sinh x \cos(\al t) -\sinh(\beta x) \right)\Big)\\
= &~{} 4\arctan \left(\frac{\beta \left(\sinh x \cos(\al t) -\sinh(\beta x) \right)}{\cosh x \cosh (\beta x) -\beta \sinh x \sinh (\beta x) -\beta \cos(\alpha t) }\right).
\end{aligned}
\ee
Even more surprising, is the following fact: if the initial data $\vec \phi (t=0)=(Q,0) +(\widetilde u_0, \widetilde s_0)$ are such that $ (\widetilde u_0, \widetilde s_0)$ are odd, then the equation \eqref{sg1} formally preserves this property: one has $\vec \phi (t)=(Q,0) + (\widetilde u, \widetilde s)(t)$, with $ (\widetilde u, \widetilde s)(t)$ odd for all time\footnote{This is a consequence of the fact that $\sin Q$ is odd and $\cos Q$ is even; the equation is now odd parity invariant: $\partial_t^2 u -\partial_x^2 u +\sin Q(\cos u -1) +\cos Q \sin u=0.$}. 
The wobbling case is a direct example of this property, and it seems the unique parity property around the kink preserved by SG.

\begin{figure}[h!]
   \centering
   \includegraphics[scale=0.35]{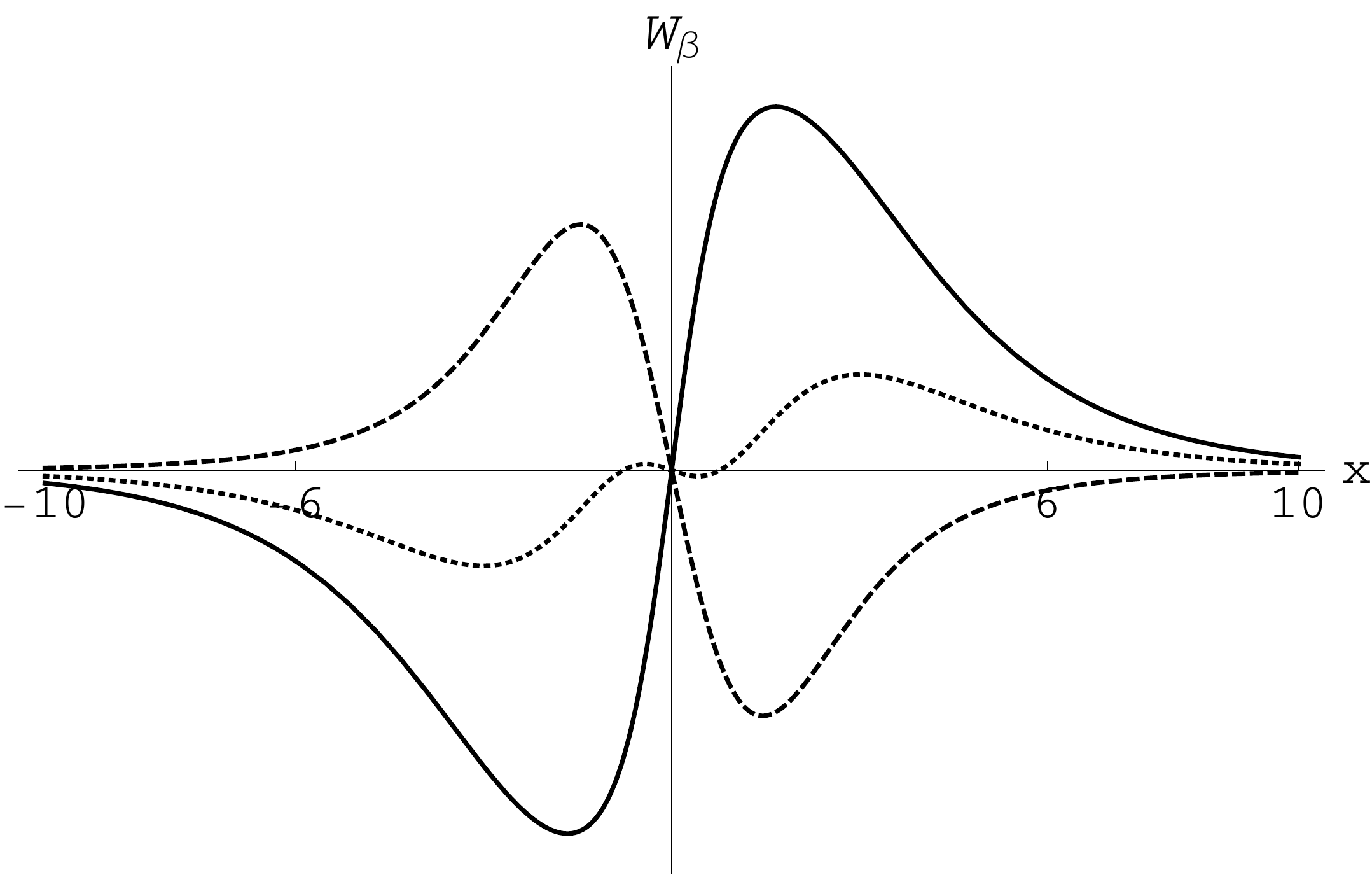}
   \includegraphics[scale=0.35]{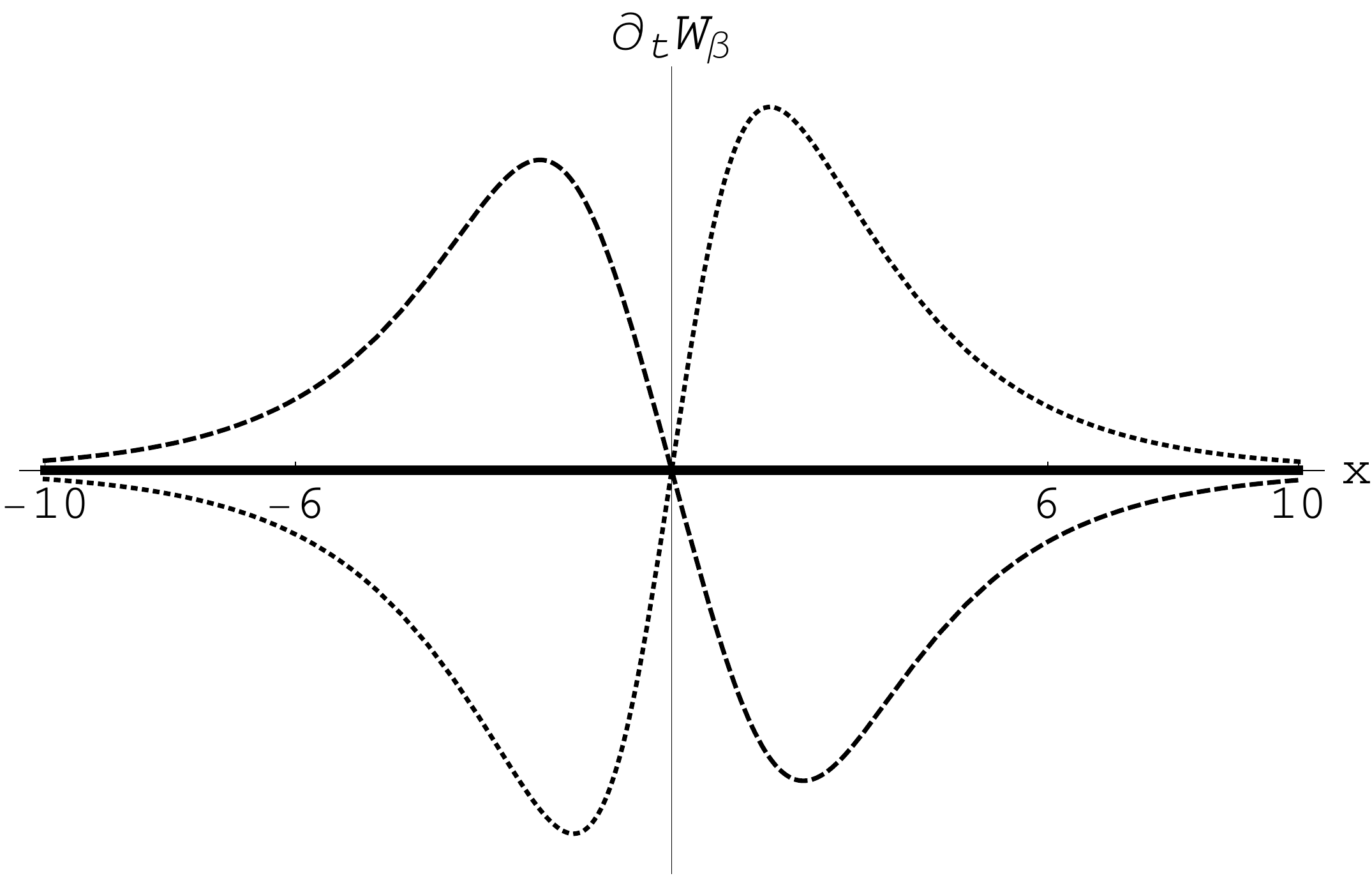}
   \caption{Evolution of the wobbling kink \eqref{wobblingK} minus $Q$ in the case $\beta=0.5$. 
   On the left, $W_\beta(t,x) -Q(x)$ for times $t=0$ (continuous curve), $t=2$ (dashed curve) and $t=6$ (dotted curve). On the right panel, $\partial_tW_\beta(t,x)$ for 
   times $t=0$ (continuous curve), $t=2$ (dashed curve) and $t=6$ (dotted curve). Note the oddness character of both graphs. Contrary to the common belief that wobbling 
   kinks are ``kink+breather'' structures, the even character of the breather \eqref{breather} is not preserved by the wobbling kink.}
   \label{Fig:7}
\end{figure}

\medskip

Consequently, and in view of \eqref{wobblingK}, no result like \eqref{AS_phi4} can be proved in the SG case in the ($Q+$ odd, odd) data case.

\medskip

Another key point to have in mind, related to the odd parity in SG, is that the linear operator around $Q$ given by 
\be\label{L_Q}
\mathcal L_Q:= -\partial_x^2 +\cos Q =-\partial_x^2 +1-2\sech^2 x,
\ee
has no internal modes  (unlike $\phi^4$), but a resonance at $\la=1$ with odd generalized eigenfunction ($=\tanh x$) at the bottom of its continuum spectrum.  This property is in concordance with the existence of an odd perturbation of the kink which does not decay to the kink (the wobbling kink), since the resonance function is also odd. To add more substance to this analogy, in the $\phi^4$ case the resonance above mentioned is associated to an even generalized eigenfunction. Moreover, the resonance $\tanh x$ of period $\la=1$ can be formally found as the spatial part of the $\beta\to 0$ limit of the derivative of $W_\beta$:
\[
L:=\frac14\lim_{\beta\to 0} \partial_\beta W_\beta(t,x) =  \tanh x \cos t.
\]
Note that $L$ does not decay in time, and solves $L_{tt} +\mathcal L_Q (L)=0.$ See also \eqref{L_M} for more properties about $L$. 
{\color{black} A natural question that arises from 
this observation is the following: }
Is there any corresponding 
connection between the $\phi^4$ resonance \eqref{Reso_phi4} and a hypothetical $\phi^4$ \emph{wobbling kink}? In the odd-data case, such a 
connection does not exist in the case of small perturbations \cite{KMM}, but here we talk about even data.

\subsection{Breathers and the AS manifold structure around zero}\label{contra}
A key feature of the breather solution $(B_\beta,\partial_t B_\beta)$ in \eqref{breather} is its parity character in space. Indeed, note that SG preserves (even, even) and (odd, odd) parities around zero, and breathers are even solutions of SG. Also consider the following \emph{parity manifolds}:
\be\label{EO}
\begin{aligned}
\mathcal E_0 := &~{}  H^1_e\times L^2_e, \\
\mathcal O_0 := &~{} H^1_o\times L^2_o .
\end{aligned}
\ee
Both manifolds are preserved by the SG flow. Also, $\mathcal E_0$ is related to the manifold of initial data under which the \emph{zero solution} is not asymptotically stable, since $(B_\beta,\partial_t B_\beta)(t=0)\in \mathcal E_0.$

\medskip

In \cite{KMM2}, it was proved that $\mathcal O_0$ is indeed related to the manifold where asymptotic stability holds:

\begin{thm}[See also Fig. \ref{Fig:1}]\label{thmkmm2}
There exists $\ve_0>0$ such that, if $(y,v)\in C(\R;H^1_o\times L^2_o)$ is a globally defined odd solution to SG such that $\sup_{t\in\R} \|(y,v)(t)\|_{H^1\times L^2} <\ve_0$, then for any compact interval $I \subset\R$ one has
\be\label{AS_zero}
\lim_{t\to \pm\infty} \|(y,v)(t)\|_{(H^1\times L^2)(I)} =0.
\ee
Moreover, there is integration in time of local norms: for any small $c_1>0$ fixed,
\be\label{Integration0}
\int \int e^{-c_1|x|} (y_x^2 +y^2+ v^2)(t,x)dxdt \lesssim \ve_0^2.
\ee
\end{thm}
\begin{rem}
Estimate \eqref{Integration0} will be useful to prove Theorem \ref{MT00}, more precisely, the convergence result in Theorem \ref{MT0}, eqn. \ref{final_data0_1} (see Section \ref{11}, Step 4).
\end{rem}

\begin{figure}[h!]
\begin{center}
\begin{tikzpicture}[scale=1.1]
\filldraw[thick, color=black!50] (-2,0.05)--(2,0.05) -- (2,-0.05) --(-2,-0.05) -- (-2,0.05);
\filldraw[thick, color=lightgray!90] (0.05,2)--(-0.05,2) -- (-0.05,-2) --(0.05,-2) -- (0.05,2);
\draw[thick,dashed] (-2,0) -- (2,0);
\draw[->] (-2,0) -- (3.5,0) node[below] {\small $(e,e)\in\mathcal E_0$};
\draw[->] (0,-2) -- (0,3) node[right] {\small $(o,o) \in\mathcal O_0$};
\node at (0,0){$\bullet$};
\node at (4/3,0){$\bullet$};
\node at (4/3,-0.4){\small $(B_\beta,B_{\beta,t})$};
\node at (0.83,2.3){\small $\longrightarrow P=0$};
\node at (0,0){$\bullet$};
\node at (0,1){$\bullet$};
\node at (-0.8,1){$(y_0,v_0)$};
\node at (-0.6,-0.4){$(0,0)$};
\end{tikzpicture}
\end{center}
\caption{A schematic representation of the initial-data manifolds $\mathcal E_0$ and $\mathcal O_0$ in \eqref{EO}. Here $(o,o)$ and $(e,e)$ mean odd-odd and even-even data in $H^1\times L^2$. The horizontal dark region represents a submanifold of small initial data for which no asymptotic stability (AS) around zero is present. The breather family $(B_\beta,B_{\beta,t})$ (for any small $\beta$) is part of this manifold (but it is not known yet if it is the \emph{only} counterexample to AS). On the other hand, the vertical submanifold is related to AS thanks to Theorem \ref{thmkmm2}, it has zero momentum $P$ (see \eqref{Momentum}) and it is part of the region in the energy space where AS is present. Finally, note that both manifolds $\mathcal E_0$ and $\mathcal O_0$ are preserved by the SG flow, intersect themselves only at the origin, and they are $H^1\times L^2$ orthogonal.}\label{Fig:1}
\end{figure}
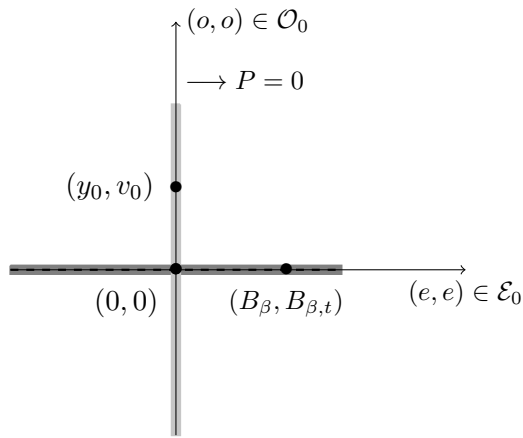

\medskip

\section{B\"acklund transformations}\label{3}

The previous discussion, mesmerizing in terms of allowing us to extend the techniques used in \cite{KMM} to the SG case, opens a new window of possibilities where asymptotic stability could hold, this time in the complement of the \emph{odd parity manifold} (note how different are parities between $\phi^4$ and SG). But first we need to introduce the SG B\"acklund Transformations (BT). For more details, see e.g. \cite{MP}.

\subsection{Definitions} Let us write \eqref{sg1} in matrix form, that is $\vec\phi =(\phi,\phi_t)=(\phi_1,\phi_2)$, in such a form that \eqref{sg1} reads now
\be\label{SG}
\begin{cases}
\partial_t \phi_1 =\phi_2,\\
\partial_t \phi_2 =\partial_x^2 \phi_1 -\sin\phi_1.
\end{cases}
\ee
Now we introduce the B\"acklund transformation (BT) that we will use in this article. Recall that $\dot H^1$ represents the closure of $C_0^\infty$ under the norm $\|\partial_x \cdot  \|_{L^2}$.
\begin{defn}[B\"acklund Transformation]\label{defi1}
Let $a\in \Com$ be fixed. Let {\color{black}$\vec\phi=(\phi_0,\phi_1)(x)$} be a function defined in $\dot H^1(\Com)\times L^2(\Com)$. We will say that $\vec\varphi$ in $\dot H^1(\Com)\times L^2(\Com)$ is a {\bf B\"acklund transformation} (BT) of $\vec \phi$ by the parameter $a$, denoted
\be\label{Flecha}
 \, \mathbb{B}(\vec\phi) \xrightarrow{\ a \ }{\vec\varphi},
\ee
if the triple $(\vec\phi,\vec\varphi,a)$ satisfies the following equations, for all $x\in \R$:
\begin{align} 
    \varphi_{0,x}-\phi_1 \ & = \ \dfrac{1}{a}\sin\left(\dfrac{\varphi_0+\phi_0}{2}\right)+a\sin\left(\dfrac{\varphi_0-\phi_0}{2}\right), \label{b1}
    \\ 
    \varphi_1-\phi_{0,x} \ & = \ \dfrac{1}{a}\sin\left(\dfrac{\varphi_0+\phi_0}{2}\right)-a\sin\left(\dfrac{\varphi_0-\phi_0}{2}\right).\label{b2}
    \end{align}
\end{defn}

The following result is standard in the literature, justifying the introduction of the BT \eqref{b1}-\eqref{b2}.

\begin{lem}[\cite{Lamb,MP}]\label{Solutions_Backlund}
If $(\vec\phi,\vec\varphi)$ are {\color{black} $(C^2\times C^1)(\R_t \times \R_x)$} functions related via a BT \eqref{b1}-\eqref{b2}, then both solve \eqref{SG}. 
\end{lem}

By using a density argument, the previous property holds for energy solutions of SG \cite{MP}, and \eqref{b1} and \eqref{b2} are satisfied in the $L^2$ sense. 

\begin{rem}
The use of the B\"acklund transformation is not new in the field of integrable stability theory. The reader can consult the monograph \cite{MS} for a detailed introduction to the subject. In recent years, several works dealing with stability of solitonic structures via B\"acklund transformations have appeared: Hoffman and Wayne \cite{HW}, Mizumachi and Pelinovsky \cite{MizPel}, \cite{AM1,AMP2} and \cite{MP}, 
but its use as a method for proving asymptotic stability results in this paper seems to be new in nonlinear wave like equations. 
\end{rem}

We finish this subsection by considering B\"acklund functionals, in the sense considered in \cite{AM1,MP} .
 
\begin{defn}[B\"acklund functionals]\label{fperturbacion}
Let $(\varphi_0,\,\varphi_1,\,\phi_0,\,\phi_1,\,a)$ be data in a space $X(\R)$ to be chosen later. Let us define the functional with vector values $\mathcal F:=(\mathcal F_1,\mathcal F_2)$, where $\mathcal F=\mathcal F(\varphi_0,\,\varphi_1,\,\phi_0,\,\phi_1,\,a) \in L^2(\R) \times L^2(\R)$, given by the system: 
\begin{align}
    \mathcal F_1 \big(\varphi_0,\,\varphi_1,\,\phi_0,\,\phi_1,\,a\big) &:=  \varphi_{0,x}-\phi_1 - \dfrac{1}{a}\sin\left(\dfrac{\varphi_0+\phi_0}{2}\right)-a\sin\left(\dfrac{\varphi_0-\phi_0}{2}\right), \label{f1}
    \\ \mathcal F_2 \big(\varphi_0,\,\varphi_1,\,\phi_0,\,\phi_1,\,a\big) &:=
    \varphi_1-\phi_{0,x}- \dfrac{1}{a}\sin\left(\dfrac{\varphi_0+\phi_0}{2}\right)+a\sin\left(\dfrac{\varphi_0-\phi_0}{2}\right).\label{f2}
\end{align}  
\end{defn}
The choice of the space $X(\R)$ heavily depends on the considered background solution. In our case, since $Q$ does not belong to $L^2$, we will consider a different space for $\mathcal F$.

%
%
%
%


\subsection{Kink profiles} Here we introduce the notion of kink profiles. See \cite{AMP1,MP} for more details.

\begin{defn}[Kink profiles]\label{Kink0}
Let $\beta\in (-1,1)$, $\beta\neq0$, and $x_0\in \R$ be fixed parameters. We define the real-valued kink profile $\vec Q:=(Q,Q_t)$ with speed $\beta$ as
\begin{align}\label{Q} 
Q(x):= Q(x; \beta, x_0)=4\arctan\big(e^{\ga(x - x_0)}\big), \qquad \ga:= (1-\beta^2)^{-1/2},
\end{align}
\be\label{Qx}
Q_x(x):= Q_x(x; \beta,x_0)=    \frac{4 \ga e^{\ga(x - x_0)}}{1+ e^{2\ga(x - x_0)}} = \frac{2\ga }{\cosh(\ga(x-x_0))},
\ee
and
\be\label{Qt}
Q_t(x):= Q_t(x; \beta,x_0)=    \frac{-4\beta \ga e^{\ga(x - x_0)}}{1+ e^{2\ga(x - x_0)}} = \frac{-2\beta\ga }{\cosh(\ga(x-x_0))}.
\ee
\end{defn}

\begin{rem}[See also \cite{MP}]\label{Solucion_Exacta}
The profile $(Q,Q_t)$ is the standard profile associated to the kink solution \eqref{Kink}. Although $(Q,Q_t)$ is not an exact solution of \eqref{SG}, 
it can be understood as follows: for each $(t,x)\in \R^2$, $(t,x) \mapsto (Q,Q_t)(x;\beta,x_0-\beta t)$ is an exact solution of \eqref{SG}, moving with speed $\beta$. 
\end{rem}

In what follows, we prove connections between kink profiles and the zero solution in SG. 
Although some of these results are standard, recall that 
we prove them not only for exact solutions, but also for profiles which are not exact solutions of SG.

\begin{lem}[Kink as BT of zero, \cite{MP}]\label{prkk1} Let $(Q, Q_t)$ be a  SG kink profile with scaling parameter $\beta \in (-1,1)$, $\beta\neq 0$, and shift $x_0$, see Definition \ref{Kink0}. Then, for each $x\in \R$, $(Q, Q_t)$ is a BT of the origin $(0,0)$ with parameter
\be\label{a(beta)}
a= a(\beta):= \left(\frac{1+\beta}{1-\beta}\right)^{1/2}.
\ee
That is,  
\begin{align}\label{eqn:BT_Q}
    Q_x  = \ \dfrac{1}{a}\sin\Big(\dfrac{ Q}{2}\Big)+a\,\sin\Big(\dfrac{ Q}{2}\Big), \qquad &   Q_t  = \ \dfrac{1}{a}\sin\Big(\dfrac{Q}{2}\Big)-a\,\sin\Big(\dfrac{Q}{2}\Big).
\end{align}
\end{lem}
Note that \eqref{eqn:BT_Q} can be read as $\mathcal F(Q, Q_t,0,0,a)=0$, in terms of \eqref{f1}-\eqref{f2}.

\begin{rem}
In \eqref{Flecha}, the parameter $a\in\R$ links $\vec \phi$ to $\vec\varphi$. In this paper, thanks to Lemma \ref{prkk1} we will connect 
the zero state with the kink state using this transformation, and then we will perturb both states (Theorems \ref{MT00} and \ref{MT0}). 
However, since we know that asymptotic stability does not hold in the (odd, odd) regime (see Subsection \ref{1p1}), what we need is 
to ensure that our initial perturbation of the kink $\vec\varphi$, which will be of type (odd, even), could lead to a perturbation of the 
zero state $\vec\phi$ of type (odd, odd). And we need (odd, odd) data because decay of small data for SG occurs in this particular setting \cite{KMM2}. 
For instance, if the data is even, the breather \eqref{breather} is a counterexample to decay, see Subsection \ref{contra} for more details.
\end{rem}

Before proving its existence, we need some results about the parity properties satisfied by the B\"acklund functionals.

\subsection{Parity properties associated to B\"acklund functionals} Recall the kink profile $Q$ introduced in \eqref{Q}. Because of parity reasons, we will need a slight modification of $Q$, 
denoted $\widetilde Q$, and simply defined as 
\be\label{tildeQ}
\widetilde Q(x;\beta,x_0):= Q(x;\beta,x_0) -\pi ,\qquad  \widetilde Q_t(x;\beta,x_0):= Q_t(x;\beta,x_0).
\ee
Note that $\widetilde Q$ is now odd in $x+x_0$, while $\widetilde Q_t$ is even, except when $\beta=0$. In such a case, it is also odd. 

\begin{lem}[Parity properties for $\mathcal F$]\label{Parity}
Let $(Q, Q_t)=(Q,Q_t)(x;\beta,x_0)$ be a kink profile of parameters $\beta\in (-1,1)$ and $x_0\in\R$. Consider 
the associated modified kink profile $(\widetilde Q,\widetilde Q_t)=(\widetilde Q,\widetilde Q_t)(x;\beta,x_0)$ introduced in \eqref{tildeQ}. Let also $(\widetilde u_{0},\widetilde s_0)\in H^1\times L^2$ 
and $(y_0,v_0) \in H^1\times L^2$ be given functions. Finally, consider $a=a(\beta)$ as defined in \eqref{a(beta)}, and $\delta\in\R$ sufficiently small. Then the following are satisfied:
\smallskip
\ben
\item[(a)] One has from \eqref{f1} and \eqref{f2}
\begin{align}
    & \qquad \mathcal F_1 \big(Q+\widetilde u_0, Q_t +\widetilde s_0 , y_0,v_0, a + \delta \big)  = \widetilde{\mathcal F}_1(\widetilde u_0,\widetilde s_0,y_0,v_0,\delta) \nonumber \\
     & \qquad\quad  :=  \widetilde Q_x + \widetilde u_{0,x} -v_0 - \dfrac{1}{a+\delta}\cos\left(\dfrac{\widetilde Q +\widetilde u_0 + y_0}{2}\right)-(a+\delta)\cos\left(\dfrac{\widetilde Q +\widetilde u_0-y_0}{2}\right), \label{f1_new}
    \\ 
    &\qquad \mathcal F_2 \big(Q+\widetilde u_0, Q_t +\widetilde s_0 , y_0,v_0, a + \delta \big)  = \widetilde{\mathcal F}_2(\widetilde u_0,\widetilde s_0,y_0,v_0,\delta) \nonumber \\
   &\qquad\quad  :=  \widetilde Q_t + \widetilde s_{0} -y_{0,x} - \dfrac{1}{a+\delta}\cos\left(\dfrac{\widetilde Q +\widetilde u_0 + y_0}{2}\right) + (a+\delta)\cos\left(\dfrac{\widetilde Q +\widetilde u_0-y_0}{2}\right).\label{f2_new}
\end{align}  
\item[(b)] If now $\widetilde u_{0},y_0\in H^1_o$ (see definitions in \eqref{Espacios_paridad}), and $x_0=0$, then 
\[
\cos\left(\dfrac{\widetilde Q +\widetilde u_0 \pm y_0}{2}\right)
\] 
are even functions, with $\lim_{x\to\pm\infty} \cos\left(\frac{\widetilde Q +\widetilde u_0 \pm y_0}{2}\right) = \lim_{x\to\pm\infty}\cos\left(\frac{\widetilde Q}{2}\right) =0$. Moreover, both functions belong to $H^1_e$.
\smallskip
\item[(c)] If now $x_0=0$, $(\widetilde u_{0},\widetilde s_0)\in H^1_o\times L^2_e$ and $(y_0,v_0) \in H^1_o\times L^2_e$, then 
\[
 \widetilde{\mathcal F}_1(\widetilde u_0,\widetilde s_0,y_0,v_0,\delta) \in L^2_e,\quad  \widetilde{\mathcal F}_2(\widetilde u_0,\widetilde s_0,y_0,v_0,\delta) \in L^2_e. 
\]
Consequently, $ \widetilde{\mathcal F}:= (\widetilde{\mathcal F}_1, \widetilde{\mathcal F}_2)$ is a well-defined functional from $X(\R):= H^1_o\times L^2_e \times H^1_o\times L^2_e\times \R$ into $L^2_e\times L^2_e$, provided $\delta$ is chosen such that $a+\delta \neq 0$. It is also a $C^1$ functional among the considered spaces. 
\smallskip
\item[(d)] Assume $x_0=0$, $\beta=0$, $a=1$ and $\delta=0$. Then, if $(\widetilde u,\widetilde s)\in H^1_o\times L^2_o$ and $(y,v) \in H^1_e\times L^2_e$, then 
\[
 \widetilde{\mathcal F}_1(\widetilde u,\widetilde s,y,v, 0) \in L^2_e,\quad  \widetilde{\mathcal F}_2(\widetilde u,\widetilde s,y,v,0) \in L^2_o. 
\]
Consequently, abusing of notation, $ \widetilde{\mathcal F}:= (\widetilde{\mathcal F}_1, \widetilde{\mathcal F}_2)$ is a well-defined functional from $X_0(\R):= H^1_o\times L^2_o \times H^1_e\times L^2_e$ into $L^2_e\times L^2_o$. It is also a $C^1$ functional among the considered spaces. 
\een
\end{lem}

\begin{proof}
Equations \eqref{f1_new}-\eqref{f2_new} are just a rewrite of \eqref{f1}-\eqref{f2}. The equality of the limit in statement $(b)$ is a consequence of the Sobolev embedding for $H^1(\mathbb{R})$ and formula \eqref{coseno} below. On the other hand, the fact that $\cos(\tfrac{
\widetilde{Q}+\widetilde{u}_0\pm y_0}{2})$ belongs to $H^1_e$ follows from basic trigonometric identities, the hypothesis $\widetilde u_0,y_0\in H^1_o$ and the following identity: \begin{align}\label{coseno}
\cos\left(\dfrac{\widetilde Q}{2}\right)=\sech(\gamma (x+x_0)).
\end{align}
Statement $(c)$ is a direct consequence of the definitions \eqref{f1_new}-\eqref{f2_new} and part $(b)$ of this Lemma (for the parity property).

\medskip

Finally, $(d)$ is consequence of $\widetilde Q_t=0$ under $\beta=0$ ($a=1$ in \eqref{a(beta)}), $\widetilde Q_x$ even if $x_0=0$, and the formulae
\[
\begin{aligned}
     &~{} \widetilde{\mathcal F}_1(\widetilde u,\widetilde s,y,v,0) =  \widetilde Q_x + \widetilde u_{x} -v - \cos\left(\dfrac{\widetilde Q +\widetilde u + y}{2}\right)-\cos\left(\dfrac{\widetilde Q +\widetilde u-y}{2}\right) \in L^2_e, 
    \\ 
    &~{} \widetilde{\mathcal F}_2(\widetilde u,\widetilde s,y,v,0) = \widetilde s -y_{x} - \cos\left(\dfrac{\widetilde Q +\widetilde u + y}{2}\right) +\cos\left(\dfrac{\widetilde Q +\widetilde u-y}{2}\right)  \in L^2_o,
\end{aligned}
\] 
valid for $(\widetilde u,\widetilde s,y,v) \in H^1_o\times L^2_o \times H^1_e\times L^2_e.$ Indeed, note that $\widetilde Q$ is odd, and
\[
\cos\left(\dfrac{\widetilde Q +\widetilde u + y}{2}\right) + \cos\left(\dfrac{\widetilde Q +\widetilde u-y}{2}\right) =2\cos\left(\dfrac{\widetilde Q +\widetilde u}{2}\right)\cos\left(\dfrac{y}{2}\right) \in L_e^2,
\]
and
\[
\cos\left(\dfrac{\widetilde Q +\widetilde u + y}{2}\right) - \cos\left(\dfrac{\widetilde Q +\widetilde u-y}{2}\right) =-2\sin\left(\dfrac{\widetilde Q +\widetilde u}{2}\right)\sin\left(\dfrac{y}{2}\right) \in L_o^2,
\]
thanks to \eqref{coseno}.
\end{proof}

\medskip

\section{The action of BT on parity manifolds. Proof of Theorem \ref{WK_orbital}}\label{4}

\subsection{BT and parity manifolds around 0 and $Q$} In what follows, we inted to get a better understanding of the image of the manifolds $\mathcal E_0$ and $\mathcal O_0$ under the B\"acklund Transformation \eqref{b1}-\eqref{b2}, at least in the case of small data. Along this section, we will rigorously justify Fig. \ref{Fig:2}, continuation of Fig. \ref{Fig:1}. Our first result is the following:

\begin{prop}\label{THMQzeroA}
Every sufficiently small  $H^1_e\times L_e^2$  perturbation of the vacuum state leads to a unique sufficiently small $H^1_o\times L_o^2$ perturbation of the SG static kink via a B\"acklund transformation.
\end{prop}
\begin{proof}
See the appendix, Section \ref{BT_kink_parity_appendix}
\end{proof}

\begin{rem}
In terms of the terminology introduced in \cite{MP}, the previous result is a \emph{lifting} lemma. We lift data from a neighborhood of zero towards data near the static kink. In \cite{HW,MP}, such a property could not hold without the addition of an extra orthogonality condition around the kink solution (see Section \ref{9} for another example). This extra condition was ensured via modulation techniques. Here we do not need such an additional condition because of the parities assumptions involved in the proof. It turns out that under (even, even) data around zero (e.g. small breathers), \emph{it is always possible to uniquely solve} the BT leading to (odd, odd) data around the kink (e.g. the wobbling kink), \emph{no matter the time $t\in\R$ at which the lifting is performed}.  
\end{rem}

What is probably more impressing is that the reciprocal of the previous result is also true. 

\begin{prop}\label{THMQzeroB}
Every sufficiently small  $H^1_o\times L_o^2$  perturbation of the SG static kink leads to a unique sufficiently small $H^1_e\times L_e^2$ perturbation of the vacuum state.
\end{prop}

\begin{rem}
In the terminology of \cite{MP}, Proposition \ref{THMQzeroB} corresponds to a \emph{descent} from a vicinity of the kink towards a corresponding vicinity of the zero solution. This property was previously established in \cite{MP} in the case of breathers, 2-kinks and kink-antikinks. However, it was always necessary to adjust the parameter $\delta$ in the BT \eqref{f1_new}-\eqref{f2_new} to ensure this property. Here, from the proof it will be clear that, under the correct parity conditions, such an additional adjustment is not necessary.  
\end{rem}

\begin{rem}
Note that if perturbations of the SG kink are uniformly bounded in time, the proof of Proposition \ref{THMQzeroB} will ensure uniform bounds in time for perturbations of the zero solution as well. 
\end{rem}

\begin{proof}[Proof of Proposition \ref{THMQzeroB}] 
See the appendix, Section \ref{BT_kink_parity_2_appendix}.
\end{proof}

\begin{rem}
As a consequence of Propositions \ref{THMQzeroA} and \ref{THMQzeroB}, the section along the horizontal axis on the left panel in Fig. \ref{Fig:2} is uniquely related to a  section along the vertical axis on the right of the same figure. 
\end{rem}

Propositions \ref{THMQzeroA} and \ref{THMQzeroB} motivate the introduction of the first manifolds of initial data around the SG kink considered in this paper. Let
\be\label{O_Q}
\begin{aligned}
\mathcal O_Q := &~{} \{ (Q+\widetilde u_0, \widetilde s_0) ~ : ~  (\widetilde u_0, \widetilde s_0) \in \mathcal O_0= H^1_o\times L^2_o  \}, \\
\mathcal E_Q := &~{}\{ (Q+\widetilde u_0, \widetilde s_0) ~ : ~  (\widetilde u_0, \widetilde s_0) \in \mathcal E_0= H^1_e\times L^2_e  \}, \\
\mathcal {OE}_Q :=&~{} \{ (Q+\widetilde u_0, \widetilde s_0) ~ : ~  (\widetilde u_0, \widetilde s_0)\in H^1_o\times L^2_e \}.
\end{aligned}
\ee
Recall that only the first manifold is preserved by the flow in time, and that the wobbling kink $(W_\beta, \partial_t W_\beta)(t)$ in \eqref{wobbling} belongs to $\mathcal O_Q $. For some reasons to be explained below, the manifold $\mathcal {OE}_Q$ is well-suited for our problem, unlike an (even, even) manifold.

\medskip

Where is the wobbling kink in Propositions \ref{THMQzeroA} and \ref{THMQzeroB}? That is the purpose of the following paragraph.

\subsection{Breather-Wobbling kink's connexion}

Now we need the following classic connection between breathers and wobbling kinks, see e.g \cite{CQS}.

\begin{lem}\label{conection_WB}
Let $\beta\in (-1,1).$ Then breathers $B_\beta$ \eqref{breather} and wobbling kinks $W_\beta$ \eqref{wobbling} are connected via a BT of parameter $a=1$. More precisely, for all $t\in\R$,
\begin{align} 
    \partial_x W_\beta - \partial_t B_\beta \ & = \ \sin\left(\dfrac{W_\beta+B_\beta}{2}\right)+ \sin\left(\dfrac{W_\beta-B_\beta}{2}\right), \label{b1_new}
    \\ 
    \partial_t W_\beta  -\partial_x B_\beta   \ & = \ \sin\left(\dfrac{W_\beta+B_\beta}{2}\right)- \sin\left(\dfrac{W_\beta- B_\beta}{2}\right).\label{b2_new}
    \end{align}
\end{lem}

\begin{proof}
The proof is somehow standard, but we include it in Appendix \ref{A}.
\end{proof}

\begin{rem}\label{shift_remark}
Lemma \ref{conection_WB} also works if breathers and kinks are perturbed, at the same time, by parameters $t\mapsto t+ x_1$ and $x\mapsto x+ x_2$, with $x_1,x_2$ free real parameters. This is just a consequence of the invariance of the equation \eqref{sg1} under space and time translations. 
\end{rem}

\begin{rem}
Lemma \ref{conection_WB} is simple but it reveals a deep property of wobbling kinks. They are not immediately related to the zero solution (which is asymptotically stable under odd perturbations) as the breather was in \cite{MP}. Instead, even if they are odd solutions, wobbling kinks are related via BT to SG breathers, which are even nondecaying functions. The change in parity is a key element present in BT. 
\end{rem}

The following deep connections between the manifolds $\mathcal O_Q$ and $\mathcal E_0$ in \eqref{O_Q}, stated in Propositions \ref{THMA} and \ref{THMB}, will be key ingredients for the proof of Theorem \ref{WK_orbital}.

\begin{prop}\label{THMA}
Every sufficiently small  $H^1_e\times L_e^2$  perturbation of the SG breather leads to a unique sufficiently small $H^1_o\times L_o^2$ perturbation of the SG wobbling kink via a B\"acklund transformation.
\end{prop}

\begin{rem}
Once again, because of the parity assumptions, we will be able to prove this lifting result without using any type of modulation on the data. Compare with Proposition \ref{THMQzeroA}, which satisfies similar properties, but around the kink solution. This time, we will prove lifting around the wobbling kink.
\end{rem}

\begin{proof}
We follow the proof of Proposition \ref{THMQzeroA} very closely, but this time we need different B\"acklund functionals, as well as new parity properties not stated in Lemma \ref{Parity} because of their lack of simplicity.

\medskip

Without loss of generality we assume that $\beta$ is positive. For the sake of simplicity from now on we shall denote by $\widetilde{W}_\beta$ the function 
\be\label{tildeW}
\widetilde W_\beta :=W_\beta-\pi,
\ee
\noindent
which is odd in space. Now, let $(y,v)\in H^1_e\times L^2_e$ be small enough given perturbations and let $t\in\R$ fixed. Consider the system of perturbed equations given by the B\"acklund functionals \eqref{b1_new}-\eqref{b2_new} 
\be\label{47_48}
\begin{aligned}
\mathcal{F}_1&:=\widetilde W_{\beta,x} + \widetilde u_x-B_{\beta,t} -v- \cos\left( \frac{\widetilde W_\beta + \widetilde u +B_\beta +y}{2}\right) -\cos\left( \frac{\widetilde W_\beta + \widetilde u -B_\beta -y}{2}\right),  \\
\mathcal{F}_2&:=\widetilde W_{\beta,t} + \widetilde s -B_{\beta,x} -y_x - \cos\left( \frac{\widetilde W_\beta + \widetilde u +B_\beta +y}{2}\right)  +\cos\left( \frac{\widetilde W_\beta + \widetilde u -B_\beta -y}{2}\right),
\end{aligned}
\ee
where $\mathcal{F}_i=\mathcal{F}_i(y,v,\widetilde u,\widetilde s)$ for $i=1,2$. Notice that for any given triplet $(y,v,\widetilde u)\in H^1_e\times L^2_e\times H^1_o$, equation $\mathcal{F}_2\equiv 0$ is trivially solvable for $s(\cdot)$ and defines a function in $L^2_o$. On the other hand, 
\[
\mathcal{F}_1:H^1_e(\R)\times L^2_e(\R)\times H^1_o(\R)\times L^2_o(\R) \longrightarrow L^2_e(\R),
\]
defines a $\mathcal{C}^1$ functional in a neighborhood of zero and due to Lemma \ref{conection_WB} we have $\mathcal{F}_1(0,0,0,0)\equiv0$. Therefore, in order to conclude the proof it is enough to show that the G\^ateaux derivative of $\mathcal{F}_1$ defines a invertible bounded linear operator with continuous inverse. In fact, notice that linearizing directly on the definition of $\mathcal{F}_1$ above and by using basic trigonometric identities we are lead to solve
\begin{align}\label{u_ODE}
\widetilde u_x = -\sin\left( \frac{\widetilde W_\beta }{2}\right)\cos\left( \frac{B_\beta}{2}\right)\widetilde u +f, \, \hbox{ for some }\, f\in L^2_e.
\end{align}
Now, in order to solve equation \eqref{u_ODE}, we define $\mu_\beta(x)$ to be the solution of 
\[
\begin{aligned}
\mu_{\beta,x}-\sin\left(\dfrac{\widetilde{W}_\beta}{2}\right)\cos\left(\dfrac{B_\beta}{2}\right)\mu_\beta=&~{} 0, \\
\hbox{ that is } \ \mu_\beta(x)=&~{} \exp\left(\int_0^x \sin\left(\dfrac{\widetilde{W}_\beta}{2}\right)\cos\left(\dfrac{B_\beta}{2}\right)\right).
\end{aligned}
\]
At this stage it is important to point out that $\mu_\beta(x)$ is an even function. Moreover, notice that by using the definitions of 
$\widetilde{W}_\beta$ and $B_\beta$ in \eqref{tildeW}-\eqref{breather} we conclude that there exists $R>1$ sufficiently large such that \begin{align}
\hbox{for all } \, x>R,& \quad \sin\left( \frac{\widetilde W_\beta }{2}\right)\cos\left( \frac{B_\beta}{2}\right)\sim 1- e^{- \beta x}, \label{growth_1}
\\ \hbox{and for all } \, x<-R,& \quad \sin\left( \frac{\widetilde W_\beta }{2}\right)\cos\left( \frac{B_\beta}{2}\right)\sim -1+ e^{ \beta x}.\label{growth_2}
\end{align}
Therefore, we conclude that $\mu_\beta\to+\infty$ as $x\to\pm\infty$. On the other hand, due to the fact that both $\mu_\beta$ and $f$ are even functions, we conclude that there is only one odd function solving \eqref{u_ODE}, which is given by
\begin{align}\label{u_sol}
\widetilde u(x)=\dfrac{1}{\mu_\beta(x)}\int_{0}^x \mu_\beta(z) f(z)dz.
\end{align}
Finally, by using Young's inequality, the explicit form of $u$ and the exponential growth of $\mu_\beta$ given by \eqref{growth_1}-\eqref{growth_2} it is easy to check that
\[
\Vert \widetilde u\Vert_{L^2(\R)}\lesssim \Vert f\Vert_{L^2(\R)}.
\]
We refer to \cite{MP} Section $6$ for a complete proof of the latter inequality in a similar context. Notice that in order to conclude that $\widetilde u\in H^1_o$ it only remains to prove that $\widetilde u_x\in L^2$. Nevertheless, this is a direct consequence of the explicit form of $\widetilde u$ in \eqref{u_sol} and the previous analysis. Therefore, we conclude the proof by applying the Implicit Function Theorem.
\end{proof}

An even more striking property is that under no extra hypothesis we are able to prove a \emph{reciprocal} theorem.

\begin{prop}\label{THMB}
Every sufficiently small  $H^1_o\times L_o^2$  perturbation of the wobbling kink leads to a unique sufficiently small $H^1_e\times L_e^2$ perturbation of the SG breather.
\end{prop}

\begin{rem}
The fact that we do not need any extra hypothesis to prove this theorem is (again) a consequence of restricting ourselves to $H^1_o\times L^2_o$ perturbations. An analogous statement for perturbations in the whole space $H^1\times L^2$ would require, for instance, some orthogonality condition hypothesis over the perturbations.
\end{rem}

\begin{proof}

We shall closely follow the ideas of the proof of Proposition \ref{THMA}, with some key differences. 
Without loss of generality we assume that $\beta$ is positive. With $\widetilde{W}_\beta$ as in \eqref{tildeW}, 
let now $(u,s)\in H^1_o\times L^2_o$ be small enough given perturbations and let $t\in\R$ fixed. 
Consider once again the system of perturbed equations given by the B\"acklund functionals \eqref{47_48}: 
\[
\begin{aligned}
\mathcal{F}_1&:=\widetilde W_{\beta,x} +\widetilde  u_x-B_{\beta,t} -v- \cos\left( \frac{\widetilde W_\beta + \widetilde u +B_\beta +y}{2}\right) -\cos\left( \frac{\widetilde W_\beta + \widetilde u -B_\beta -y}{2}\right),  \\
\mathcal{F}_2&:=\widetilde W_{\beta,t} +\widetilde s -B_{\beta,x} -y_x - \cos\left( \frac{\widetilde W_\beta + \widetilde u +B_\beta +y}{2}\right)  +\cos\left( \frac{\widetilde W_\beta + \widetilde u -B_\beta -y}{2}\right),
\end{aligned}
\]
where $\mathcal{F}_i=\mathcal{F}_i(y,v,u,s)$ for $i=1,2$. Notice that for any given triplet $(y,u,s)\in H^1_e\times H^1_o\times L^2_o$, equation $\mathcal{F}_1\equiv 0$ is trivially solvable for $v(\cdot)$ and defines a function in $L^2_e$. On the other hand,
 \[
\mathcal{F}_2:H^1_e(\R)\times L^2_e(\R)\times H^1_o(\R)\times L^2_o(\R)\longrightarrow L^2_o(\R),
\]
defines a $\mathcal{C}^1$ functional in a neighborhood of zero and due to Lemma \ref{conection_WB} we have $\mathcal{F}_2(0,0,0,0)\equiv0$. Therefore, in order to conclude the proof it is enough to show that the G\^ateaux derivative of $\mathcal{F}_2$ defines a invertible bounded linear operator with continuous inverse. 
In fact, notice that linearizing directly on the definition of $\mathcal{F}_2$ above and by using basic trigonometric identities we are lead to solve 
\begin{align}\label{y_ODE}
y_x = \sin\left( \frac{\widetilde W_\beta }{2}\right)\cos\left( \frac{B_\beta}{2}\right)y +f, \, \hbox{ for some }\, f\in L^2_o.
\end{align}
Note that unlike \eqref{u_ODE} now we have a ``$+$'' sign in the right-hand side. As before, in order to solve equation \eqref{y_ODE}, we define $\mu_\beta(x)$ to be the solution of 
\[
\mu_{\beta,x}+\sin\left(\dfrac{\widetilde{W}_\beta}{2}\right)\cos\left(\dfrac{B_\beta}{2}\right)\mu_\beta=0,
\]
that is
\[
 \mu_\beta(x)=\exp\left(-\int_0^x \sin\left(\dfrac{\widetilde{W}_\beta}{2}\right)\cos\left(\dfrac{B_\beta}{2}\right)\right).
\]
Note as before, that $\mu_\beta(x)$ is an even function. Moreover, notice that by using the definitions of 
$\widetilde{W}_\beta$ and $B_\beta$ in \eqref{tildeW}-\eqref{breather} we conclude that there exists $R>1$ sufficiently large such that \begin{align}
\hbox{for all } \, x>R,& \qquad \sin\left( \frac{\widetilde W_\beta }{2}\right)\cos\left( \frac{B_\beta}{2}\right)\sim 1- e^{- \beta x}, \label{growth_1_ap}
\\ \hbox{and for all } \, x<-R,& \qquad \sin\left( \frac{\widetilde W_\beta }{2}\right)\cos\left( \frac{B_\beta}{2}\right)\sim -1+ e^{ \beta x}.\label{growth_2_ap}
\end{align}
Notice that, from the previous analysis, we deduce that in this case $\mu_\beta\to0$ exponentially fast as $x\to\pm\infty$. 
Moreover since $\mu_\beta$ and $f$ are even and odd functions respectively we conclude\[
\int_\R\mu_\beta(x) f(x)dx=0.
\]
Thus, solving \eqref{y_ODE} from $-\infty$ to $x$ we conclude that there is only one solution to \eqref{y_ODE} which is given by
\begin{align}\label{y_sol}
y=\dfrac{1}{\mu_\beta(x)}\int_{-\infty}^x \mu_\beta(z) f(z)dz.
\end{align}
Finally, we claim that due to the explicit form of $y(\cdot)$ we have
\[
\Vert u\Vert_{L^2(\R)}\lesssim \Vert f\Vert_{L^2(\R)}.
\]
In order to prove this we shall follow the ideas of \cite{MP}. In fact, first of all notice that by using \eqref{growth_2_ap} we deduce that for all $s\leq x\ll-1$ we have \begin{align*}
\left\vert\dfrac{\mu_\beta(s)}{\mu_\beta(x)}\right\vert\leq C\left\vert\dfrac{\cosh (s)}{\cosh(x)}\right\vert\leq C e^{s-x},
\end{align*}
for some constant $C$ only depending on $\beta$. Therefore, by using formula \eqref{y_sol} we conclude that for $x\ll-1$ we have\[
\vert y(x)\vert\leq Ce^{-x}\star \Big(f(\cdot)\id_{(-\infty,x]}(\cdot)\Big),
\]
where $\star$ stands for the convolution in the space variable. Since $y(\cdot)$ is an even function the same bound holds for $x\gg 1$. 
Therefore, by using Young's inequality we conclude that 
\[
\Vert u\Vert_{L^2(\R)}\lesssim \Vert f\Vert_{L^2(\R)}.
\]
Finally, notice that it only remains to prove that $y_x\in L^2$. Nevertheless, this is a direct consequence of the explicit form of $y(\cdot)$ in \eqref{y_sol} and the previous analysis. Therefore, we conclude the proof by applying the Implicit Function Theorem.
\end{proof}

%

As a corollary of Propositions \ref{THMA}-\ref{THMB} and the orbital stability of the SG breather for $H^1\times L^2$ perturbations (see \cite{MP}, Theorem $1.1$) we obtain the orbital stability of the SG wobbling kink (Theorem \ref{WK_orbital}). See Fig. \ref{Fig:2} for a graphic explanation of all previous results in this section.

\subsection{Proof of Theorem \ref{WK_orbital}} Now we state a quantitative version of Theorem \ref{WK_orbital}. 

\begin{thm}[Orbital stability of Wobbling kinks under odd perturbations]\label{WK_orbital0}
There exists $\eta_0>0$ such that the following holds. Let $(W_\beta,\partial_t W_\beta)$ be the wobbling kink written in \eqref{wobbling}.  Consider initial data of the form $(W_\beta +u_0, \partial_tW_{\beta} +s_0)$, with $(u_0,s_0) \in H^1_o\times L^2_o$ satisfying
\be\label{smallness_data}
\|(u_0,s_0)\|_{H^1\times L^2} <\eta<\eta_0.
\ee
Then, there exists $C_0>0$ and $x_1:\R\to \R$, $x_1=x_1(t)$ of class $C^1$ such that the solution $(\phi,\phi_t)(t,x)$ to SG satisfies
\[
\sup_{t\in\R} \| (\phi,\phi_t) (t) - (W_\beta (t+x_1(t),\cdot),(\partial_tW_{\beta}) (t+x_1(t),\cdot)\|_{H^1\times L^2} <C_0\eta.
\]
\end{thm}

\begin{rem}
Recall that no shift on the $x$ variable is allowed in the wobbling kink since the data is odd. This implies, following the lines just below \eqref{wobblingK}, that the solution is odd for all time.
\end{rem}

\begin{proof}
We follow the ideas in \cite{MP}. Assume \eqref{smallness_data}. Proposition \ref{THMB} allows us to construct via BT a unique small, (even, even) perturbation $(y_0,v_0)$ of the breather solution $(B_\beta, B_{\beta,t})$ from \eqref{breather}. 
Therefore, from \cite{MP} we know that the breather is stable, up to some shifts $x_1(t)$ and $x_2(t)$. By parity, only $x_1(t)$ is not necessarily zero. Evolving in time the perturbation of the SG breather solution, and using Proposition \ref{THMA} with a suitably chosen wobbling kink (it must have the same shift $x_1(t)$, see Remark \ref{shift_remark} for details), we conclude the orbital stability.  
\end{proof}

\begin{figure}[h!]
\begin{center}
\begin{tikzpicture}[scale=0.9]
\filldraw[thick, color=black!50] (-2,0.05)--(2,0.05) -- (2,-0.05) --(-2,-0.05) -- (-2,0.05);
\filldraw[thick, color=lightgray!90] (0.05,2)--(-0.05,2) -- (-0.05,-2) --(0.05,-2) -- (0.05,2);
\draw[->] (-2,0) -- (3.5,0) node[below] {\small $(e,e)\in\mathcal E_0$};
\draw[->] (0,-2) -- (0,3) node[right] {\small $(o,o) \in\mathcal O_0$};
\node at (0,0){$\bullet$};
\node at (4/3,0){$\bullet$};
\node at (4/3,-0.4){\small $(B_\beta,B_{\beta,t})$};
\node at (0.85,2.3){\small $\longrightarrow P=0$};
\node at (0,0){$\bullet$};
\node at (0,1){$\bullet$};
\node at (-0.8,1){$(y_0,v_0)$};
\node at (-0.6,-0.4){$(0,0)$};
\end{tikzpicture}
\quad
\begin{tikzpicture}[scale=0.9]
\node at (0,3){BT};
\node at (0,2.5){$\longrightarrow$};
\node at (0,0){$(t=0)$};
\end{tikzpicture}
\quad
\begin{tikzpicture}[scale=0.9]
\filldraw[thick, color=lightgray!90] (-1,-1.05)--(1,0.95) -- (1,1.05) --(-1,-0.95) -- (-1,-1.05);
\filldraw[thick, color=black!50] (0.05,2)--(-0.05,2) -- (-0.05,-2) --(0.05,-2) -- (0.05,2);
\draw[thick,dashed] (-1.5,-1.5)--(2,2);
\draw[->] (-2,0) -- (4,0) node[below] {\small $(Q+e,e)\in\mathcal E_Q$};
\draw[->] (0,-2) -- (0,3) node[right] {\small $(Q+o,o) \in\mathcal O_Q$};
\node at (0,0){$\bullet$};
\node at (0.8,0.8){$\bullet$};
\node at (2.8,0.8){\small $(Q+\widetilde u_0,\widetilde s_0) \in \mathcal{OE}_Q$};
\node at (2.8,2){\small $\longrightarrow P=0$};
\node at (0,0){$\bullet$};
\node at (0,1){$\bullet$};
\node at (-1.2,1){\small $(W_\beta,W_{\beta,t})$};
\node at (0.6,-0.4){$(Q,0)$};
\end{tikzpicture}
\end{center}
\caption{A schematic representation of the action of the BT at time $t=0$ on the initial-data manifolds $\mathcal E_0$ and $\mathcal O_0$ in \eqref{EO}, as well as why Theorem \ref{WK_orbital} holds. Here $(o,o)$ and $(e,e)$ mean odd-odd and even-even data in $H^1\times L^2$. The horizontal submanifold on the left (containing the breather $B_\beta$ in \eqref{breather} and its time derivative) is sent via BT towards a vertical submanifold in $\mathcal O_Q$ (see \eqref{O_Q} for definitions) containing the wobbling kink $W_\beta$ from \eqref{wobbling}, and its time derivative. See Propositions \ref{THMQzeroA}, \ref{THMQzeroB} and Lemma \ref{conection_WB} for the rigorous proofs. On the other hand, the vertical submanifold on the left for which there is AS (Theorem \ref{thmkmm2}) is sent in Theorem \ref{MT00}, via BT, towards an ``oblique'' submanifold $\mathcal M_{\eta,0}$ preserving the zero momentum condition. On the right, only the vertical manifold is preserved by the flow, and the image via BT of $(y_0,v_0)$ is $(Q+\widetilde u_0,\widetilde s_0)$, with zero momentum (see Section \ref{7} for more details).}\label{Fig:2}
\end{figure}
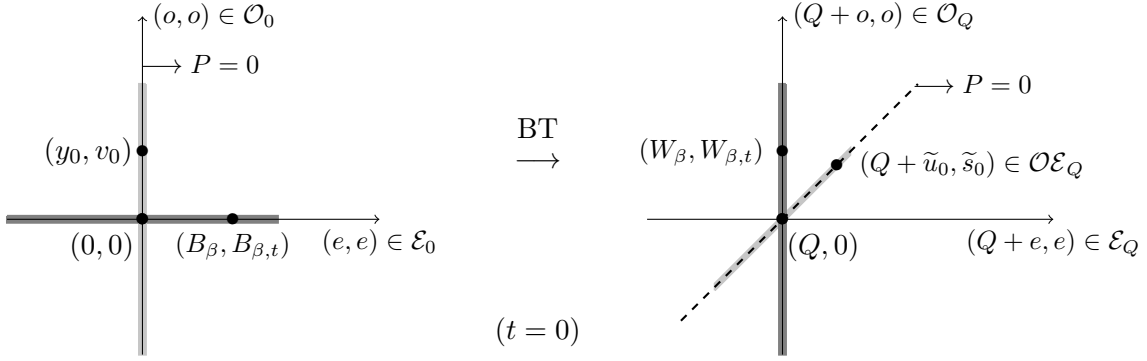

We finish this section with a simple lemma (see \cite{MVD} for instance) stating that wobbling kinks cannot be orbitally stable for general data, in the sense that general perturbations may not lead to the evolution of a kink plus a perturbation which \emph{is periodic in time}. In that sense, \emph{the wobbling kink \eqref{wobbling} ceases to exist}. However, there is a family of 3-soliton solutions which represents the interaction between a static kink and a moving breather. The stability of this family, as already stated in the introduction, is an open problem. 

\begin{lem}\label{3soliton}
There exists a family of SG 3-solitons with frequency $\beta\in (-1,1)$ and speed $v\in(-1,1)$, for which the wobbling kink \emph{wobbling} is the case of zero momentum, i.e. $v=0$. This family is explicitly given by
\begin{align}\label{QB}
W_{\beta,v}(t,x):=Q(x)-4\arctan\left(\dfrac{\beta}{i\alpha}\tan( \Theta-\overline{\Theta})\right),
\end{align}
where $\overline{\Theta}$ denotes the complex conjugate of $\Theta$, which is given by
\[
\Theta:=\arctan\left(\dfrac{\beta a_v+i\alpha a_v+1}{\beta a_v+i\alpha a_v-1}\tan\left(\arctan e^x-\arctan e^{\gamma[\beta(x-vt)-i\alpha (t-vx)]}\right)\right),
\]
and the parameters $a_v$, $\al$ and $\gamma$ are given by
\[
a_v:=\left(\dfrac{1+v}{1-v}\right)^{1/2}, \quad \al=\sqrt{1-\beta^2}, \quad \hbox{and}\quad \gamma=\dfrac{1}{\sqrt{1-v^2}}.
\]
(compare with \eqref{a(beta)}).
\end{lem}

Before finishing this section, some important remarks are in order.

\begin{rem}
Some snapshots of this 3-soliton family \eqref{QB} are presented in Figure \ref{Fig:6}.
\end{rem}

\begin{rem}\label{convergencia_a_WK}
The family $(W_{\beta,v},\partial_t W_{\beta,v})(t,x)$ in \eqref{QB} converges naturally to the wobbling kink \eqref{wobbling} when $v\to 0$. See Appendix \ref{B} for details.
\end{rem}

\begin{rem}
The family $(W_{\beta,v},\partial_t W_{\beta,v})(t=0)$ in \eqref{QB} can also be regarded as an essentially (odd, even) perturbation of the kink $(Q,0)$. However, as we shall see in Section \ref{7}, it does not belong to the class of initial data under which Theorem \ref{MT00} holds, due to the fact that it has nonzero momentum. Note also that these initial data leads to a perturbation on the position of the kink.
\end{rem}

\begin{figure}[h!]
   \centering
   \includegraphics[scale=0.3]{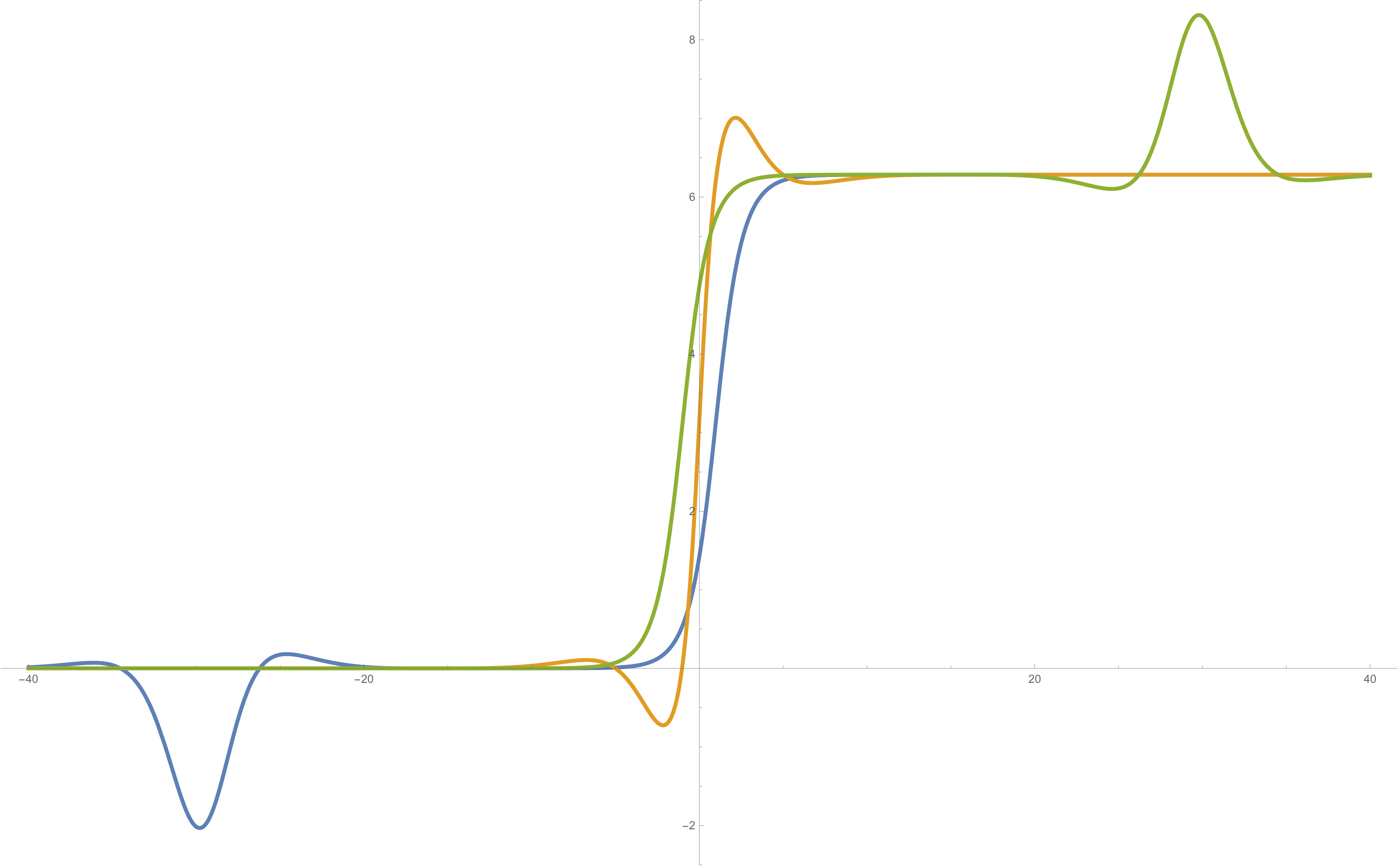} 
   \caption{The general 3-soliton or generalized wobbling kink \eqref{QB} with $\beta=0.5$ and $v=0.4$ at times $t=-55.3$ (blue), $t=0$ (orange), and $t=55.3$ (green). This solution represents a breather colliding with a static kink. Notice that, after the ``breather'' collides against the kink, the latter shifts. Moreover, notice that at time $t=0$ this corresponds to an odd perturbation of the kink, that is, $(W_{\beta,v}-Q)(t=0)$ is odd.}
   \label{Fig:6}
\end{figure}

\section{Linearized B\"acklund transformations and resonances}\label{5}

Having proved Theorem \ref{WK_orbital}, now we make an interesting digression from the proof of the remaining Theorem \ref{MT00}, which will start in the next Section. 

\medskip

An essential point where SG \eqref{sg1} and $\phi^4$ \eqref{phi4} meet is at the level of linearized transformations, even when only one of them is integrable (this is one of the reasons why the AS for the $\phi^4$ kink $H$ is harder).  Section \ref{2} first showed such an analogy at the level of linearized operators around kinks, as well as resonances. 

\medskip

In this section we present a new point of contact between both theories, maybe not recognized before in full detail. This connection is of independent interest, in view of recent advances in AS problem via dual methods \cite{KMM3}.

\subsection{The SG case}  Let us consider the {\bf linearized B\"acklund transformations} (LBT) around the SG kink solution (see \eqref{b1}-\eqref{b2} and \eqref{b1_new}-\eqref{b2_new})
\be\label{LBT}
\begin{cases}
\partial_x \phi -\partial_t \varphi =-\sin\left( \frac{\widetilde Q}{2}\right) \phi\\
\partial_t \phi -\partial_x \varphi= -\sin\left( \frac{\widetilde Q}{2}\right) \varphi.
\end{cases}
\ee
Here, $\phi$ and $\varphi$ are $C^2$ functions depending on $(t,x)$. Some interesting properties of \eqref{LBT} are stated in the following result (maybe well-known in the literature), which are just consequence of \eqref{b1_new}-\eqref{b2_new}.

\begin{lem}[LBT in the SG case]\label{LBT_SG}
One has that
\ben
\item If $(\phi,\varphi)$ solves \eqref{LBT}, then they satisfy 
\be\label{KKK}
\phi_{tt} +\mathcal L_{Q} \phi =0 \qquad \hbox{and} \qquad \varphi_{tt} -\varphi_{xx} +\varphi =0,
\ee
for $\mathcal L_Q$ given in \eqref{L_Q}, respectively. The converse is not necessarily true.
\smallskip
\item (Translations of kernels). $(\phi,\varphi)=(Q',0)$ solves \eqref{LBT}.
\smallskip
\item Let
\be\label{L_M}
L:= \frac14\partial_\beta W_\beta \Big|_{\beta=0} = \tanh x \cos t, \qquad M:=  \frac14\partial_\beta B_\beta \Big|_{\beta=0} = \sin t.
\ee
be the corresponding resonances of SG around the kink and zero, generated by the wobbling kink and breather respectively. Then $(\phi,\varphi)=(L,M)$ satisfies \eqref{LBT}.
\een
\end{lem}

\begin{rem}
Similar conclusions as in item (3) above are obtained if the periodic functions in time are correctly changed: in \eqref{L_M} $(\tilde L, \tilde M) = \Big(-\tanh x \sin t, \cos t\Big)$ is also solution of \eqref{LBT}. This is consequence of the fact that time derivatives of solutions also solve \eqref{LBT} (and \eqref{KKK}).
\end{rem}

\begin{proof}[Proof of Lemma \ref{LBT_SG}]
We start by proving the first point. In fact, by differentiating both equations in \eqref{LBT} with respect to space and time respectively we obtain
\[
\partial_{tt}\phi-\partial_{tx}\varphi=-\sin\left(\tfrac{\widetilde{Q}}{2}\right)\partial_t\varphi \quad \hbox{and} \quad 
\partial_{xx}\phi-\partial_{xt}\varphi=-\sech^2(x)\phi-\sin\left(\tfrac{\widetilde{Q}}{2}\right)\partial_x\phi.
\]
Therefore, using that $\sin\left(\tfrac{\widetilde{Q}}{2}\right)=\tanh(x)$ we obtain 
\[
\partial_{tt}\phi-\partial_{xx}\phi=\sech^2(x)\phi-\sin\left(\tfrac{\widetilde{Q}}{2}\right)(\partial_t\varphi-\partial_x\phi)=-(1-2\sech^2(x))\phi.
\]
In the same way, by differentiating both equations in \eqref{LBT} in the opposite order, that is, with respect to time and space respectively, we conclude\[
\partial_{tt}\varphi-\partial_{xx}\varphi=-\sech^2(x)\varphi+\sin\left(\tfrac{\widetilde{Q}}{2}\right)\left(\partial_t\phi-\partial_x\varphi\right)=-\varphi.
\]
Now, recalling that the static kink satisfies \eqref{eqn:BT_Q} with $a=1$ we immediately obtain 
\be\label{calculo interesante}
Q''=2\left(\sin\left(\dfrac{Q}{2}\right)\right)'=2\cos\left(\dfrac{Q}{2}\right)\sin\left(\dfrac{Q}{2}\right)=\sin\left(\dfrac{\widetilde{Q}}{2}\right)Q'.
\ee
Finally, differentiating \eqref{L_M} we obtain \[
L_x= \sech^2 x\cos t,\quad L_t=-\tanh x \sin t, \quad M_x\equiv 0 \quad \hbox{and}\quad M_t=\cos t.
\]
Replacing these formulas into \eqref{LBT} we obtain
\[
L_x-M_t=(\sech^2x-1)\cos t=-\tanh^2(x)\cos t=-\sin\left(\tfrac{\widetilde{Q}}{2}\right)L,
\]
and
\[
L_t-M_x=-\tanh x \sin t=-\sin\left(\tfrac{\widetilde{Q}}{2}\right)M.
\]
The proof is complete.
\end{proof}

\subsection{The $\phi^4$ case as extension of SG} What is the corresponding LBT for $\phi^4$? Although $\phi^4$ is not integrable, and apparently has no BT, it has essentially two suitable LBT around their soliton states. Indeed, for $H$ as in \eqref{H},
\be\label{LBT_p4}
\begin{cases}
\partial_x \phi -\partial_t \varphi =-\sqrt{2} H \phi\\
\partial_t \phi -\partial_x \varphi= - \sqrt{2} H \varphi,
\end{cases}
\ee
is a LBT for $\phi^4$. Recall that $H'=(1-H^2)/\sqrt{2}$. The resonance in \eqref{Reso_phi4} enters in \eqref{LBT_p4} as follows:

\begin{lem}[LBT in the $\phi^4$ case]\label{LBT_p4_1}
Let \eqref{LBT_p4} be the LBT of $\phi^4$. Then one has the following properties.
\ben
\item If $(\phi,\varphi)$ solves \eqref{LBT_p4}, then 
\[
\phi_{tt} +\mathcal L_{H} \phi =0 \quad \hbox{ and } \quad \varphi_{tt} +\widetilde{\mathcal L}_{H} \varphi =0,
\]
for $\mathcal L_H$ given in \eqref{L_H} and 
\be\label{tildeL}
\widetilde{\mathcal L}_{H}:= -\partial_x^2 +1 +H^2=-\partial_x^2 +2 -(1-H^2).
\ee
The converse is not necessarily true. Note that $\widetilde{\mathcal L}_{H} \geq 0$ by definition.
\smallskip
\item $(\phi,\varphi)=(H',0)$ solves \eqref{LBT_p4}.
\smallskip
\item {\color{black}$\sigma (\widetilde{\mathcal L}_{H})=\{\frac32\} \cup [2,\infty)$. $\la=0$ is not an eigenvalue, $\la=\frac32$ is the first eigenvalue associated to the eigenfunction $Y_0:=-\frac1{\sqrt{3}}\sech \left(\frac{x}{\sqrt{2}}\right)$, and $H $ is odd resonance at $\la=2$.}
\smallskip
\item (Connection between internal modes). Recall $Y_1$ from \eqref{Y_1}. Then, $(\phi, \varphi)$ given by 
\be\label{L_M_3}
(\phi, \varphi) = \Big(Y_1(x) \sin (t\sqrt{3/2}), Y_0(x) \cos (t\sqrt{3/2}) \Big)
\ee
solves \eqref{LBT_p4}.
\item Let
\be\label{L_M_4}
L_4:=-\left(1-\frac32\sech^2\left(\frac{x}{\sqrt{2}}\right) \right) \sin(\sqrt{2}t), \qquad M_4:=  \tanh \left(\frac{x}{\sqrt{2}}\right) \cos(\sqrt{2}t).
\ee
be the corresponding resonances of $\phi^4$ around the kink and ``around the internal mode'', respectively. Then $(\phi,\varphi)=(L_4,M_4)$ satisfies \eqref{LBT_p4}.
\een
\end{lem}

\begin{rem}
Similar conclusions in items (4) and (5) above are obtained if the periodic functions in time are correctly changed: in \eqref{L_M_3} $(\phi, \varphi) = \Big(Y_1(x) \cos (t\sqrt{3/2}), -Y_0(x) \sin (t\sqrt{3/2}) \Big)$ is also solution of \eqref{LBT_p4}, and instead of \eqref{L_M_4},
\[
\widetilde L_4:=-\left(1-\frac32\sech^2\left(\frac{x}{\sqrt{2}}\right) \right) \cos(\sqrt{2}t), \qquad \widetilde M_4:=  -\tanh \left(\frac{x}{\sqrt{2}}\right) \sin(\sqrt{2}t).
\]
are also solutions of \eqref{LBT_p4}. This is consequence of the fact that time derivatives of solutions also solve \eqref{LBT_p4}.
\end{rem}

\begin{rem}
Note that $L_4$ already appeared in this paper in \eqref{Reso_phi4}. Also, in item 4, $M_4$ is called ``resonance around the internal mode'' because it is exactly a resonance of $\widetilde{\mathcal L}_H$ in \eqref{tildeL}, which can be regarded as the linear operator for which the internal mode $Y_1$ in \eqref{Y_1} is its generator, in the sense of \cite{KMM3}.
\end{rem}
 
\begin{rem}
Note that $\mathcal L_{H}$ and $\widetilde{\mathcal L}_{H}$ correspond, in the terminology of Schr\"odinger operators, to a linear operator and its dual, respectively, see e.g. \cite{KMM3} and references therein.
\end{rem}

\begin{proof}[Proof of Lemma \ref{LBT_p4_1}]
The proof follows from straightforward computations. Nevertheless, for the sake of completeness we shall show them below. In fact, by 
differentiating both equations in \eqref{LBT_p4} with respect to space and time respectively we obtain \[
\partial_{tt}\phi-\partial_{tx}\varphi=-\sqrt{2}H\varphi_t \quad \hbox{and} \quad 
\partial_{xx}\phi-\partial_{tx}\varphi=-\sqrt{2}H'\phi-\sqrt{2}H\partial_x\phi.
\]
Thus, by replacing one equation into the other we conclude
\[
\partial_{tt}\phi-\partial_{xx}\phi=(1-H^2)\phi-\sqrt{2}H(\partial_t\varphi-\partial_x\phi)=(1-3H^2)\phi.
\]
In the same way, deriving both equations in \eqref{LBT_p4} with respect to time and space respectively we conclude
\[
\partial_{tt}\varphi-\partial_{xx}\varphi=-\sqrt{2}H'\varphi+\sqrt{2}H(\partial_t\phi-\partial_x\varphi)=-(1+H^2)\varphi.
\]
This proves the first point.

\medskip

Now, recalling that $H$ satisfies the equation $H'=(1-H^2)/\sqrt{2}$, we immediately conclude
\[
H''=\dfrac{1}{\sqrt{2}}(1-H^2)'=-\sqrt{2}HH',
\]
and hence $(H',0)$ solves \eqref{LBT_p4}. 

\medskip

The third point is consequence of standard Sturm-Liouville theory, and the fact that $Y_0$ has a sign and it is even.

\medskip

Now we intend to prove point $(4)$. In fact, let us start by some computations. By differentiating directly in the definition of $(\phi,\varphi)$ we obtain \[
\partial_x\phi= -\dfrac{1}{\sqrt{2}}\left(1-2\sech^2\left(\dfrac{x}{\sqrt{2}}\right)\right)\sech\left(\dfrac{x}{\sqrt{2}}\right)\tanh^2\left(\dfrac{x}{\sqrt{2}}\right)\sin\left(\dfrac{\sqrt{3}t}{\sqrt{2}}\right),
\]
and
\[
\partial_x\varphi=\dfrac{1}{\sqrt{6}}\sech\left(\dfrac{x}{\sqrt{2}}\right)\tanh\left(\dfrac{x}{\sqrt{2}}\right)\cos\left(\dfrac{\sqrt{3}t}{\sqrt{2}}\right).
\]
Thus, by replacing these formulas into the system \eqref{LBT_p4} we obtain 
\begin{align*}
\partial_x\phi-\partial_t\varphi& =-\sqrt{2}\left(1-\sech^2\left(\dfrac{x}{\sqrt{2}}\right)\right)\sech\left(\dfrac{x}{\sqrt{2}}\right)\sin\left(\dfrac{\sqrt{3}t}{\sqrt{2}}\right)
\\ & = -\sqrt{2}H\sech\left(\dfrac{x}{\sqrt{2}}\right)\tanh\left(\dfrac{x}{\sqrt{2}}\right)\sin\left(\dfrac{\sqrt{3}t}{\sqrt{2}}\right)=-\sqrt{2}H\phi,
\end{align*}
and
\[
\partial_t\phi-\partial_x\varphi=\left(\sqrt{\dfrac{3}{2}}-\dfrac{1}{\sqrt{6}}\right)\sech\left( \dfrac{x}{\sqrt{2}}\right)\tanh\left( \dfrac{x}{\sqrt{2}}\right)\cos\left(\dfrac{\sqrt{3}t}{\sqrt{2}}\right)=\sqrt{2}H\varphi,
\]
what finish the proof of the fourth statement. Finally, by differentiating \eqref{L_M_4} we obtain 
\[
L_{4,x}=-\dfrac{3}{\sqrt{2}}\sech^2\left(\dfrac{x}{\sqrt{2}}\right)\tanh\left(\dfrac{x}{\sqrt{2}}\right)\sin(\sqrt{2}t)
\]
and
\[
 M_{4,x}:=\dfrac{1}{\sqrt{2}}\sech^2\left(\dfrac{x}{\sqrt{2}}\right)\cos(\sqrt{2}t).
\]
Therefore, by replacing these formulas into the left-hand side of the first equation in \eqref{LBT_p4} we obtain
\begin{align*}
L_{4,x}-M_{4,t}&=-\dfrac{3}{\sqrt{2}}\sech^2\left(\dfrac{x}{\sqrt{2}}\right)\tanh\left(\dfrac{x}{\sqrt{2}}\right)\sin(\sqrt{2}t)+\sqrt{2}\tanh\left(\dfrac{x}{\sqrt{2}}\right)\sin(\sqrt{2}t)
\\ & = \sqrt{2}\left(1-\dfrac{3}{2}\sech^2\left(\dfrac{x}{\sqrt{2}}\right)\right)\tanh\left(\dfrac{x}{\sqrt{2}}\right)\sin(\sqrt{2}t),
\end{align*}
what proves that $(L_4,M_4)$ satisfy the first equation in \eqref{LBT_p4}. On the other hand, by replacing these formulas into the left-hand side of the second equation in \eqref{LBT_p4} we obtain \begin{align*}
L_{4,t}-M_{4,x}&=-\sqrt{2}\left(1-\dfrac{3}{2}\sech^2\left(\dfrac{x}{\sqrt{2}}\right)\right)\cos(\sqrt{2}t)-\dfrac{1}{\sqrt{2}}\sech^2\left(\dfrac{x}{\sqrt{2}}\right)\cos(\sqrt{2}t)
\\ & =-\sqrt{2}\left(1-\sech^2\left(\dfrac{x}{\sqrt{2}}\right)\right)\cos(\sqrt{2}t)=-\sqrt{2}HM_4.
\end{align*}
The proof is complete. 
\end{proof}

It turns out that the resonance $M_4$ in \eqref{L_M_4} plays the role of $L$ in \eqref{L_M}, and it is also related via LBT to a resonance of the zero solution of linear Klein-Gordon, as in Lemma \ref{LBT_SG}, with a slight modification coming from the eigenvalue 3/2. What we prove now is in some sense similar to the factorization discussed in
\cite{KMM3} and references therein: we can also connect $M_4$ with the even resonance associated to the vacuum state (the equivalent to $M$ in \eqref{L_M}). For simplicity, we work with complex valued data. Let $\la_0:= i \sqrt{3/2},$ and consider the following LBT: 
\be\label{LBT_2}
\begin{cases}
\partial_x \widetilde\varphi -\partial_t\widetilde \psi =-\frac1{\sqrt{2}} H\widetilde \varphi  \mp \la_0 \widetilde \psi \\
\partial_t \widetilde\varphi-\partial_x \widetilde\psi= -\frac1{\sqrt{2}} H \widetilde\psi  \mp \la_0 \widetilde \varphi.
\end{cases}
\ee
Note that we have two LBT depending on the sign $\pm$ on the right, but both are essentially the same. The following second LBT result for $\phi^4$ follows: 
\begin{lem}[LBT in the $\phi^4$ case, second part]\label{LBT_p4_2}
One has that
\ben
\item If $(\widetilde\varphi,\widetilde\psi)$ solves \eqref{LBT_2}, then 
\be\label{Aux_00}
\widetilde \varphi_{tt}  + \widetilde{\mathcal L}_{H} \widetilde \varphi  =0  \quad \hbox{ and } \quad  \widetilde \psi_{tt}  -\widetilde \psi_{xx} +2\widetilde\psi =0.
\ee
Once again, the converse is not necessarily true. 
\item Let
\be\label{L_M_4_3}
M_4:=  \tanh \left(\frac{x}{\sqrt{2}}\right) e^{i\sqrt{2}t},\qquad N_4:= (-2i\mp \sqrt{2} \la_0)e^{i\sqrt{2}t}.
\ee
Then $(M_4,N_4)$ satisfies \eqref{LBT_2}.
\een
\end{lem}
%
%

\begin{proof}
Just differentiating both equations in \eqref{LBT_2} with respect to space and time respectively we obtain 
\[
\partial_{xx} \widetilde\varphi -\partial_{tx}\widetilde \psi =-\frac1{\sqrt{2}} H'\widetilde \varphi -\frac1{\sqrt{2}} H\widetilde \varphi_x  \mp \la_0 \widetilde \psi_x , \quad \hbox{and} \quad
\partial_{tt} \widetilde\varphi-\partial_{tx} \widetilde\psi= -\frac1{\sqrt{2}} H \widetilde\psi_t  \mp \la_0 \widetilde \varphi_t.
\]
Thus, by replacing one equation into the other we conclude
\[\begin{aligned}
&\partial_{tt}\widetilde\varphi-\partial_{xx}\widetilde\varphi
=-\frac{1}{\sqrt{2}}H(\widetilde\psi_t - \widetilde\varphi_x) + \frac{1}{\sqrt{2}}H'\widetilde\varphi\mp\lambda_0(\widetilde\varphi_t-\widetilde\psi_x)\\
&=-\frac{1}{2}H^2\widetilde\varphi \mp\frac{\lambda_0}{\sqrt{2}}H\widetilde\psi + \frac{1}{2}(1-H^2)\widetilde\varphi
\pm\frac{\lambda_0}{\sqrt{2}}H\widetilde\psi+\lambda_0^2\widetilde\varphi\\
&=-(1+H^2)\widetilde\varphi.
\end{aligned}
\]
\noindent
Namely, we get 
\[
 \widetilde \varphi_{tt}  + \widetilde{\mathcal L}_{H} \widetilde \varphi  =0.
\]
In the same way, deriving both equations in \eqref{LBT_2} with respect to time and space respectively we conclude
\[
\partial_{tt}\widetilde\psi-\partial_{xx}\widetilde\psi=-\frac{1}{\sqrt{2}}H'\widetilde\psi+\frac{1}{\sqrt{2}}H(\partial_t\widetilde\varphi-\partial_x\widetilde\psi)
\pm\lambda_0(\widetilde\psi_t-\widetilde\varphi_x)=-2\widetilde\psi.
\]
Namely, we get 
\[
 \widetilde \psi_{tt} -\partial_{xx}\widetilde\psi  + 2\widetilde\psi  =0.
\]
This ends the proof of the first point. 

\medskip

Finally, by differentiating \eqref{LBT_2} we obtain
\[\begin{aligned}
&M_{4x}=\frac{1}{\sqrt{2}}\sech^2\left(\frac{x}{\sqrt{2}}\right)e^{i\sqrt{2}t},\quad N_{4x}=0,\\
&M_{4t}=i\sqrt{2} \tanh \left(\frac{x}{\sqrt{2}}\right)e^{i\sqrt{2}t},\quad N_{4t}=i\sqrt{2}(-2i\mp\sqrt{2}\lambda_0)e^{i\sqrt{2}t}.
\end{aligned}
\]
Now, subtracting we get
\[\begin{aligned}
& M_{4x} - N_{4t} = \frac{1}{\sqrt{2}}\sech^2\left(\frac{x}{\sqrt{2}}\right)e^{i\sqrt{2}t} - i\sqrt{2}(2\pm\sqrt{2}\lambda_0)e^{i\sqrt{2}t}\\
& = -\frac{1}{\sqrt{2}} \tanh^2\left(\frac{x}{\sqrt{2}}\right)e^{i\sqrt{2}t} + \left( \frac{1}{\sqrt{2}} - 2\sqrt{2} \pm2i\lambda_0 \right) e^{i\sqrt{2}t}\\
& = -\frac{1}{\sqrt{2}}HM_4 + \left( \sqrt{2}(i\sqrt{\frac{3}{2}})^2 \pm2i\lambda_0 \right)e^{i\sqrt{2}t}
 = -\frac{1}{\sqrt{2}}HM_4 + \lambda_0 \left( \sqrt{2}\lambda_0 \pm2i \right)e^{i\sqrt{2}t}\\
& = -\frac{1}{\sqrt{2}}HM_4 \mp \lambda_0[-2i\mp\sqrt{2}\lambda_0]e^{i\sqrt{2}t}=-\frac{1}{\sqrt{2}}HM_4\mp\lambda_0N_4.\\
\end{aligned}
\]
Similarly, we also get
\[\begin{aligned}
 M_{4t} - N_{4x} &= i\sqrt{2} \tanh \left(\frac{x}{\sqrt{2}}\right)e^{i\sqrt{2}t}
 = -\frac{1}{\sqrt{2}}H[-2i\mp\sqrt{2}\lambda_0]e^{i\sqrt{2}t} \mp \lambda_0 \tanh \left(\frac{x}{\sqrt{2}}\right)e^{i\sqrt{2}t}\\
& = -\frac{1}{\sqrt{2}}HN_4 \mp \lambda_0M_4.
\end{aligned}
\]
This last fact ends the proof of the second point.
\end{proof}

\begin{rem}
In terms of the results in \cite{KMM3}, understanding the AS of the $\phi^4$ under general data is related to the understanding of the corresponding AS around zero of the equation
\[
\widetilde \psi_{tt}  -\widetilde \psi_{xx} +2\widetilde\psi  + N(\widetilde \psi)=0,
\]
for some $N(\widetilde \psi)$ determined after two consecutive reductions of $\phi^4$ using Lemmas \ref{LBT_p4_1} and \ref{LBT_p4_2}. This equation has the simplified $N_4$ in \eqref{L_M_4_3} as resonance, which is far simpler that $L_4$ in \eqref{L_M_4}.  
\end{rem}

\medskip

\section{Asymptotic stability manifolds for the SG kink: Theorem \ref{MT00} revisited}\label{6}

Now we have all the ingredients to fully understand Theorem \ref{MT00}: Theorem \ref{WK_orbital0}, the whole Section \ref{5}, and particularly Figure \ref{Fig:2}.

\medskip

Recall the manifold $\mathcal {OE}_Q$ in \eqref{O_Q}. In this section, our goal is precisely to construct a smooth manifold of initial data of the form 
\[
\vec\phi(t=0)=(Q,0) + \hbox{(odd, even)} \in \mathcal {OE}_Q,
\]
perturbations of $(Q,0)$, for which a final state is attained. Unlike the (odd, odd) configuration $\mathcal O_Q$ discussed in \eqref{O_Q}, in our present case the initial parities shall not be preserved by the flow. Recall the definitions of $H^m_e(\R)$ and $H^m_o(\R)$ given in \eqref{Espacios_paridad}. Theorem \ref{MT00} will follow from the following \emph{static} asymptotic stability manifold for the SG kink:
%
%
%
%
%

\begin{thm}[Zero momentum manifold for asymptotic stability in SG]\label{MT0} There exists $\eta_0>0$ such that, for all $0<\eta<\eta_0$, the following holds. 
There exists a smooth manifold $\mathcal M_{\eta,0}$ (of infinite codimension) of initial data $(\phi_0,\phi_1)$ of the form 
\be\label{initial_data}
 (\phi_0,\phi_1)(x) = (Q(x) + \widetilde u_0(x), \widetilde s_0(x)), \qquad \Vert (\widetilde u_0,\widetilde s_0) \Vert_{H^1\times L^2}< \eta,
\ee
 where $\widetilde u_0\in H^1_o(\R),$ $\widetilde s_0\in L^2_e(\R)$, and satisfying the following properties. Let $(\phi,\phi_t)(t)$ be the global solution of \eqref{sg1} with initial data $ (\phi_0,\phi_1)$. Then,
 
\ben
\item
 $(\phi,\phi_t)(t)$ has zero momentum: $P[(\phi,\phi_t)]=0$.
 \medskip
 \item There exists a smooth $\rho(t)\in\R$ satisfying
 \be\label{eta2}
 \sup_{t\in\R}| \rho'(t)| \lesssim \eta^2,
 \ee 
such that, for any sufficiently large bounded interval $I\subset\R$, the following alternative holds:
\ben 
\item[(a)] there exists a sequence $t_n\to \pm\infty$ such that $|\rho(t_n) |\to +\infty$ and
\begin{align}\label{final_data0_0}
\lim_{n\to\pm\infty}\Vert (\phi,\phi_t)(t_n)-(Q,0)(\cdot-\rho(t_n))\Vert_{(H^1\times L^2)(I)}=0;
\end{align} 
\item[(b)] $\rho(t)$ stays bounded for all $t\in\R$ and
\begin{align}\label{final_data0_1}
\lim_{t\to\pm\infty}\Vert (\phi,\phi_t)(t)-(Q,0)(\cdot-\rho(t))\Vert_{(H^1\times L^2)(I)}=0.
\end{align} 
Moreover, $\rho(t)\to \bar\rho \in\R$ in this case.
 \een
 \medskip
\item The manifold  $\mathcal M_{\eta,0}$ defining \eqref{initial_data} is characterized as the image, under a B\"acklund transformation and the Implicit Function Theorem, of initial perturbations $(y_0,v_0)\in H_o^1\times L^2_e$ of the zero solution which satisfy $v_0=0$ and are connected to \eqref{initial_data} in such a way that they preserve their total zero momentum. 
\een 
\end{thm}

Some remarks are certainly necessary.

\begin{rem}
Note that in Theorem \ref{MT0} we do not specify the space where $(\phi,\phi_t)$ are posed, this because $(Q,0)(t)$ in \eqref{Kink} does not belong to $H^1\times L^2$. However, it is possible to show local and global well-posedness (LWP), such that $H^1\times L^2$ perturbations are naturally allowed \cite{DG}. 
\end{rem}

\begin{rem}[Explicit examples]
Note that the exact solution to SG (see \eqref{Kink_lorentz} for the notation)
\be\label{exact_solution}
(\phi,\partial_t \phi)(t,x) =(Q(t,x;\beta,0), \partial_t Q(t,x;\bt,0)),
\ee
has nonzero momentum, provided $\beta\neq 0$. These data are not included in Theorem \ref{MT0} precisely because of this nonzero momentum property. Indeed, with the terminology proposed above, for any $\beta$ small enough, one has the initial data $ (\phi_0,\phi_1)(x) = (Q(x) + \widetilde u_0(x), \widetilde s_0(x))$, with $\widetilde u_0(x):=Q(0,x;\bt,0)-Q(0,x;0,0) $ and $\widetilde s_0(x):=\partial_t Q(t,x;\bt,0)\Big|_{t=0} $ (see \eqref{Kink_lorentz} for definitions). Moreover, the evolution in this case is given by the exact solution \eqref{exact_solution}. However, using a Lorentz transformation \eqref{Lorentz}, it is possible to reduce the problem of data as in \eqref{exact_solution} to data of zero momentum, by changing the initial time (and also the data) considered at the new initial time.
\end{rem}

\begin{rem}[About the manifold $\mathcal M_{\eta,0}$]
Note that, unlike in the (odd, odd) case, data of the form \eqref{initial_data} is
not preserved by the SG equation. The clearest example is probably stated in the previous remark. The fact that $M_{\eta,0}$ has \emph{infinite codimension} should not be a surprise: by Theorem \ref{WK_orbital} a big part of the (kink + odd, odd) manifold does not satisfy the asymptotic stability property, or in other works, asymptotic stability manifolds are maybe far from being finite codimensional. Finally, recall that small shifts of kinks as initial data are not contained in the manifold $\mathcal M_{\eta,0}$, because they break the parity assumptions. 
\end{rem}

\begin{rem}
It is worth noticing that, for any non-zero $\beta\in (-1,1)$ and any non-zero $v\in (-1,1)$, the wobbling kink $W_{\beta,v}$ in \eqref{QB} does not belong to $\mathcal{M}_{\eta,0}$. This is a consequence of the fact that for such an election of $\beta$ and $v$, the wobbling kink $W_{\beta,v}$ has non zero momentum.
\end{rem}

\begin{rem}
Theorem \ref{MT0} is in some sense sharp: if now $\widetilde s_0$ is general (e.g. odd), the convergence does not hold. 
Also, it seems to be the first asymptotic stability result  in one dimension valid for perturbations of kinks \emph{in the energy space} that lead to modifications in the shifts (that is to say, the parity symmetry is not preserved by the flow). We remark that Kopylova and Komech \cite{KK2} also considered general perturbations in weighted Sobolev spaces of kink solutions for field theories with sufficiently flat nonlinearities, which do not contain SG nor $\phi^4$.
\end{rem}

\begin{rem}[About the shift $\rho(t)$]
It was not possible for us to show that, with data only in the energy space, $\rho(t)$ always converge to a final state. Indeed, because of the proof that we invoke, this fact is deeply related to the AS of the zero solution along moving in time space-intervals, a property that it is not known for odd data (note that in general this property fails to be true: any small moving breather contradicts that property). See Remarks \ref{convergencia_1} and \ref{convergencia_2} for full details. However, even in the case where $\rho(t)$ diverges along a subsequence, there is a sort of AS around any compact spatial interval. We conjecture that under some additional condition on the initial data, $\rho(t)$ is always convergent to a final state. See \cite{KK2} for similar results. 
\end{rem}

\begin{rem}[On the literature]
The use of suitable choices of manifolds of initial data is not a new tool in analyzing stability issues in nonlinear models. Remember for instance the 
use of a $C^1$ center-stable manifold around the soliton solution of the NLKG equation \cite{KNS}. Besides
that, basic definitions of stable, unstable and center-stable manifolds can be found in \cite{BJ} and \cite{NS1}. In critical gKdV equations, smooth manifolds around the unstable soliton are constructed in \cite{MaMeRa,MaMeNaRa}. See also \cite{KMM3} for a very recent construction of the stable manifold leading to decay in monic subcritical NLKG equations.
\end{rem}

The rest of this paper (Sections \ref{7}-\ref{11}) is devoted to the proof of Theorem \ref{MT0}, which is essentially divided in five parts:
\ben
\item Section \ref{7}: construction of the manifold of initial data $\mathcal M_{\eta,0}$.
\item Section \ref{8}: modulation of the data around the kink.
\item Section \ref{9}:  using BT, lifting of the data from (odd, odd) perturbations around the zero solution.
\item Section \ref{10}: improved estimates on the shift parameters.
\item Section \ref{11}: end of proof, essentially proving \eqref{final_data0_0} and \eqref{final_data0_1}.
\een
 In next section, we will construct the manifold of initial data $\mathcal M_{\eta,0}$, proving part (3) of Theorem \ref{MT0}.

\begin{rem}[About the initial data]
In what follows, and in order to avoid misunderstandings in the notation, we shall denote (see \eqref{Q}-\eqref{Qt})
\be\label{Q0}
(Q^0, Q_t^0):=(Q,Q_t)(x;\beta=0,x_0=0), \qquad  (\widetilde Q^0, \widetilde Q_t^0):=(\widetilde Q,\widetilde Q_t)(x;\beta=0,x_0=0).
\ee
These functions are nothing but the background initial data ``$(Q,0)$'' stated in Theorem \ref{MT0}, written this time in terms of kink profiles.
\end{rem}

\medskip

\section{Proof of Theorem \ref{MT0}: construction of the manifold of initial data}\label{7}

In this Section we will construct the initial data \eqref{initial_data}. In order to prove this, we will solve the BT functionals \eqref{b1}-\eqref{b2} (more precisely, \eqref{f1_new}-\eqref{f2_new}) in the opposite sense to the one performed in \cite{MP}; that is to say, given any initial data near zero of (odd, odd) type, we will show the existence of (odd, even) type perturbation data around the kink.

\medskip

The idea is to use the Implicit Function Theorem, choosing $(a,\phi_0,\phi_1)$ around $(1,0,0)$ in \eqref{b1}-\eqref{b2}  and uniquely solving for $(\varphi_0,\varphi_1)=\Phi(a,\phi_0,\phi_1)$ around the kink $(Q,0)$ ($\Phi$ represents the implicit function). Here the properties of the BT play a key role, in the sense that the only possibility for a solution to the previous question, because of strong parity constraints in BT, is to choose the initial perturbation $\phi_1$ on the velocity at time zero exactly equals to zero (more precisely, we need $\phi_1$ odd as above explained, but parity restrictions in BT only permit $\phi_1$ even). Here is when the manifold $\mathcal M_\eta$ in \eqref{Manifold} will rise up ($\eta$ is an artificial smallness parameter, needed to run the Implicit Function Theorem), because that restriction to zero speed reduces the open character of the Implicit Function sets and solutions to the direct image of a graph of the form $\Phi(a,\phi_0,0)$. Finally, the zero momentum manifold $\mathcal M_{\eta,0}$ appearing in Theorem \ref{MT0} is just the image obtained by the set $\Phi(1,\phi_0,0)$, for which the momentum of both data is zero.

\medskip

Recall $\widetilde{\mathcal F}_1$ and $\widetilde{\mathcal F}_2$ in \eqref{f1_new}-\eqref{f2_new}.  In what follows, we will prove that given $(y_0,v_0,\delta)\in H^1_o\times L^2_e\times \R$ sufficiently small, it is possible to uniquely solve the nonlinear system
\be\label{F_new}
\widetilde{\mathcal F}_1(\widetilde u_0,\widetilde s_0,y_0,v_0,\delta) =0,\quad  \widetilde{\mathcal F}_2(\widetilde u_0,\widetilde s_0,y_0,v_0,\delta)=0,
\ee
for $(\widetilde u_0,\widetilde s_0) \in H^1_o\times L^2_e\times \R$ sufficiently small.

\begin{lem}[Construction of initial data]\label{Implicit}
Assume $x_0=0$ as in \eqref{Q0}. There exists $\eta_0>0$ such that, for all $0<\eta<\eta_0$, the following is satisfied.

\ben
\item Given any $(y_0,v_0,\delta)\in H^1_o\times L^2_e\times \R$ such that $\|(y_0,v_0)\|_{H^1\times L^2} +|\delta| <\eta$, there are unique $(\widetilde u_0,\widetilde s_0):=\Phi(y_0,v_0,\delta) \in H^1_o\times L^2_e$ small enough, and such that  \eqref{F_new} are satisfied. 
\smallskip
\item
Moreover, the implicit mapping $\Phi$, defined from $(y_0,v_0,\delta)\in H^1_o\times L^2_e\times \R$ such that $\|(y_0,v_0)\|_{H^1\times L^2} +|\delta|<\eta$ into $H^1_o\times L^2_e $,  is a $C^1$ diffeomorphism in its domain of definition.
\een
\end{lem}

\begin{rem}
Lemma \ref{Implicit} can be understood, in terms of the works \cite{AM1} and \cite{MP}, as a sort of \emph{lifting} of the initial data from the zero background. In those papers, such a property holds only if suitable orthogonality 
conditions are imposed. Otherwise, the derivative mapping $D\widetilde{\mathcal F}$ does not define an homeomorphism. The novelty here is that, whenever we restrict ourselves to the subclass $H^1_o\times L^2_e\times \R$ (namely, we impose a fixed parity), these orthogonality conditions are not needed anymore.

\end{rem}

\begin{proof}[Proof of Lemma \ref{Implicit}]
The proof follows the ideas in \cite{MP} (see also \cite{HW} for the first approach in the SG case), with the main difference being which function will be found in terms of the others. From \eqref{F_new}, \eqref{Q0} and \eqref{f1_new}-\eqref{f2_new}, we are lead to solve the equations 
\begin{align}
 \widetilde{Q}^0_x+\widetilde u_{0,x}-v_0&= \dfrac{1}{1+ \delta}\cos\left(\dfrac{\widetilde{Q}^0+\widetilde u_0+y_0}{2}\right)+ (1+\delta) \cos\left(\dfrac{\widetilde{Q}^0+\widetilde u_0-y_0}{2}\right),\label{perm_k_bajada1}
    \\ \widetilde{Q}^0_t+ \widetilde s_0-y_{0,x}&=\dfrac{1}{1+ \delta}\cos\left(\dfrac{\widetilde{Q}^0+\widetilde u_0+y_0}{2}\right)-(1+\delta)\cos\left(\dfrac{\widetilde{Q}^0+\widetilde u_0-y_0}{2}\right). \label{perm_k_bajada2}
\end{align} 
(Recall that $a=1$ because $\beta=0.$) A simple checking shows that $\widetilde{Q}^0_t=0$, see \eqref{Qt}. Note also that knowing $(y_0,v_0,\delta)$
and $\widetilde u_0$, $\widetilde s_0$ is easily found using the second equation above. 
Therefore, we are only lead to show the existence of uniqueness of $\widetilde u_0$.

\medskip

Thanks to Lemma \ref{Parity}, we are lead to consider the invertibility of the linear operator around $(\widetilde u_0,y_0)=(0,0)$ of $\widetilde{\mathcal F}_1$. Therefore, given $\delta\in\R$ small, we must solve
\be\label{EDO1}
\widetilde u_{0,x}  +\frac12 \left( \frac{1}{(1+ \delta)} + (1+\delta)\right) \sin\left(\frac{\widetilde{Q}^0}{2}\right)\widetilde u_{0}  = f, \qquad \widetilde u_{0}\in H^1_o,
\ee
for any $f \in L^2_e$.  

\medskip

The term $ \sin\left(\frac{\widetilde{Q}^0}{2}\right)$ is odd and easily\footnote{Indeed, 
\[
 \sin\left(\frac{\widetilde{Q}^0}{2}\right) =\sin \left(2 \arctan e^x -\frac\pi2\right) =-\cos \left(2 \arctan e^x \right) =\tanh x.
\]
} computed: $ \sin\left(\frac{\widetilde{Q^0}}{2}\right) =\tanh x$. Consequently, \eqref{EDO1} becomes
\[
\left( \widetilde u_{0}\cosh^{\nu_0} x \right)_x = f \cosh^{\nu_0} x,\qquad \nu_0 := \frac12 \left( \dfrac{1}{(1+ \delta)} + (1+\delta)\right)\geq1.
\]
Hence, since $ \widetilde u_{0} $ must be odd,
\[
\begin{aligned}
 \widetilde u_{0} (x) =&~{}  \widetilde u_{0} (x=0) \cosh^{-\nu_0} x  + \cosh^{-\nu_0} x \int_0^x  f(s) \cosh^{\nu_0} s \, ds \\
 =&~{}  \cosh^{-\nu_0} x \int_0^x  f(s) \cosh^{\nu_0} s \, ds.
\end{aligned} 
\]
Note that, given $f\in L^2_e$,  $\widetilde u_{0}$ is clearly odd. 
The proof that it belongs to  $H^1$ and defines an homeomorphism is direct and can be found in \cite{MP}.
\end{proof}

\medskip

We finally describe the smooth manifold $\mathcal M_\eta$ of initial data $(Q^0,0)+  (\widetilde u_0, \widetilde s_0) \in H^1_o\times L^2_e$ for which Theorem \ref{MT0} is valid. 

\medskip

First of all, recall the parameter $\eta\in (0,\eta_0)$ and the implicit function $\Phi$ obtained in Lemma \ref{Implicit}. We will define
\be\label{Manifold}
\mathcal M_\eta := \Big\{ (Q^0,0) + (\widetilde u_0, \widetilde s_0) ~ :  ~  (\widetilde u_0, \widetilde s_0):=\Phi(y_0, 0,\delta), \quad y_0 \in H^1_o, \quad \|y_0\|_{H^1} +|\delta|<2\eta  \Big\}.
\ee
By definition $\mathcal M_\eta$ is a smooth manifold; it also satisfies some nice properties related to the uniqueness associated to the Implicit Function Theorem.

\begin{lem}[Basic properties of $\mathcal M_\eta$]
Consider the manifold $\mathcal M_\eta$ introduced in \eqref{Manifold}, and recall $a=a(\beta)$ defined in \eqref{a(beta)}. 
Then, the following properties are satisfied:
\ben
\item[(a)] For any $0<\eta<\eta_0$, one has $(Q^0,0)\in \mathcal M_\eta$.
\smallskip
\item[(b)] $\mathcal M_\eta-(Q^0,0) \in H^1_o\times L^2_e$.
\smallskip
\item[(c)] For any $\beta\in (-1,1)$, $\beta\neq 0$ sufficiently small (depending on $\eta_0$ in Lemma \ref{Implicit}), let $(Q,Q_t)(x;\beta,x_0)$ be the kink profile in \eqref{Q}. Then, 
\[
(\widetilde u_0,\widetilde s_0)(x) :=(Q(x;\beta,0)-Q(x;0,0),Q_t(x;\beta,0))
\]
belongs to $\mathcal M_\eta-(Q^0,0)$, with $y_0=0$ and $\delta =a(\beta)-a(0)=a(\beta)-1$. In other words, for $\beta$ sufficiently small,
\[
(Q(\cdot;\beta,0)-Q(\cdot;0,0),Q_t(\cdot;\beta,0)) = \Phi (0,0,a(\beta)-1) \in H^1_o \times L^2_e. 
\]
\een
\end{lem}
\begin{proof}
The proof of (a) follows from the fact that $ \Phi(0, 0)=(0,0,0)$. The proof of (b) is direct from Lemma \ref{Implicit}. The proof of (c) is also direct from Lemma \ref{prkk1} and the uniqueness for the values of $\Phi$ given by the Implicit Function Theorem.
\end{proof}

\medskip

It turns out that a very important quantity to consider when dealing with the asymptotic stability of 
kinks, is its momentum (see \eqref{Momentum}), which is a key tool to find suitable rigidity and smoothness 
properties of the limit profile. For instance, in \cite{KMM} all solutions considered had zero momentum. 
In Theorem \ref{MT0}, solutions do have nonzero momentum, but the kink itself may not have zero momentum, which makes its description harder than usual.

\begin{lem}[Momentum of the initial data]\label{Mom_Backlund}
Let $\vec \phi=\vec \phi(t,x)$ be a solution issue of the initial data $(\phi_0,\phi_1) \in \mathcal M_\eta$ in \eqref{Manifold}. Then one has
\be\label{new_P}
P[\vec\phi] =  2\left(\frac1{1+\delta} - (1+\delta) \right).
\ee
In particular, the sign of the momentum depends on the sign of $\delta$, and the zero momentum submanifold is generated by the  perturbation data $(\widetilde u_0, \widetilde s_0):=\Phi(y_0, 0,0)$.
\end{lem}

\begin{proof}
The proof of this result is consequence of Lemma 2.4 in \cite{MP}, which states that 
\[
P[\vec\phi] = P[y_0,0] +\frac{1}{a+\delta} (\ell_+^+-\ell_-^+)(t) -(a+\delta)(\ell_+^- -\ell_-^-)(t),
\] 
and the fact that $P[y_0,0] =0$ and
\[
\begin{aligned}
\ell_+^\pm (t) := &~{} \lim_{x\to + \infty }\left( 1- \cos\left(\dfrac{Q +\widetilde u_0\pm y_0}{2}\right)\right) = 2, \\
 \ell_-^\pm (t) :=&~{}  \lim_{x\to - \infty } \left( 1-\cos\left(\dfrac{Q+\widetilde u_0\pm y_0}{2}\right) \right)= 0.
 \end{aligned}
\]
(Note that both limits make sense because $\widetilde u_0$ and $y_0$ are in $H^1$ in one dimension.)
\end{proof}
Since it will be important in the proof of Theorem \ref{MT0}, we will record the {\bf zero momentum submanifold} $\mathcal M_{\eta,0} \subseteq \mathcal M_\eta  $ as
\be\label{Manifold0}
\mathcal M_{\eta,0} := \Big\{ (Q^0,0) + (\widetilde u_0, \widetilde s_0) ~ :  ~  (\widetilde u_0, \widetilde s_0):=\Phi(y_0, 0,0), \quad y_0 \in H^1_o, \quad \|y_0\|_{H^1} <\eta  \Big\}.
\ee
This manifold is characterized by taking $\delta=0$ in $\mathcal M_{\eta}$, see \eqref{new_P}. As we will see below, working in this manifold will simplify and clarify the dynamics of the kink. However, we believe that some interesting properties are also possible to show in the general manifold $\mathcal M_{\eta}$.

\medskip

We conclude this section with the following result, which ends the proof of Theorem \ref{MT0}, parts (1) and (3).

\begin{cor}
For data in $(\phi,\phi_t) \in \mathcal M_{\eta,0}$, one has zero momentum: $P[(\phi,\phi_t)]=0$.
\end{cor}

\medskip

\section{Modulation of the data}\label{8}

In this section we will choose suitable modulation parameters to ensure a uniquely defined dynamics for the perturbation terms in \eqref{final_data0_0}.

\subsection{Choice of final speed} In this subsection, we shall describe the final speed obtained by a perturbed kink. There are at least two ways to understand this final speed, but both are equivalent, as we will see below.

\medskip

The first definition of final speed is motivated by the momentum of the kink profile, assuming that around it there is nothing close at infinity in time (which means that there is asymptotic stability). 

\begin{defn}[Final speed via conservation of momentum]\label{beta1}
Consider $0<\eta<\eta_0$ as given in Lemma \ref{Implicit}. Let $y_0\in H^1_o$, and $(\widetilde u_0,\widetilde s_0,\delta)=\Phi(y_0,0,\delta)$ such that $(Q^0,0)+ (\widetilde u_0,\widetilde s_0)$ is in $\mathcal M_\eta$ in \eqref{Manifold}. We define the final speed $\beta_1\in \R$ as the unique solution to the equation
\be\label{momentum_beta1}
P[(Q,Q_t)(\cdot; \beta_1,0)] = \frac12\int \widetilde s_0 (Q^0+\widetilde u_0)_x .
\ee
where $P$ is the momentum \eqref{Momentum}, and $(Q, Q_t)=(Q,Q_t)(x;\beta,x_0)$ be a kink profile of parameters $\beta\in (-1,1)$ and $x_0\in\R$.
\end{defn}

\begin{rem}
Under the smallness assumption $0<\eta<\eta_0$, equation \eqref{momentum_beta1} has always a unique solution for $\beta_1$. Indeed, from \eqref{Q}-\eqref{Qt} one has
\be\label{igualdad_1}
\begin{aligned}
P[(Q,Q_t)(\cdot; \beta_1,0)] = &~{} \frac12 \int (Q_xQ_t)(x; \beta_1,0)dx  \\
=&~{} -\frac12 \beta_1\gamma(\beta_1) \int Q'^2 = -4 \beta_1 (1-\beta_1^2)^{-1/2},
\end{aligned}
\ee
so that \eqref{momentum_beta1} has a unique solution thanks to the Inverse Function Theorem.
\end{rem}

\begin{rem}
As we have already mentioned, the motivation for this choice is simple and comes from the conservation of momentum. If asymptotic stability holds around the kink solution, it must imply that around the kink there is no additional momentum.  
\end{rem}

The second option to define a final speed is given by the B\"acklund transformation. 
Indeed, from \eqref{perm_k_bajada1}-\eqref{perm_k_bajada2}, there is a natural way to define a final
speed in the presence of asymptotic stability, respecting the absence of additional energy/momentum at infinity. 
For this definition, recall Lemma \ref{prkk1} on the B\"acklund transformations for general kink profiles.

\begin{defn}[Final speed via B\"acklund transformations]\label{beta2}
Consider $0<\eta<\eta_0$ as given in Lemma \ref{Implicit}. Let $y_0\in H^1_o$, and $(\widetilde u_0,\widetilde s_0,\delta)=\Phi(y_0,0,\delta)$ such that $(Q^0,0)+ (\widetilde u_0,\widetilde s_0)$ is in $\mathcal M_\eta$ in \eqref{Manifold}. Then, we define the final speed $\beta_2\in (-1,1)$ as the unique solution for the speed of a kink + no-dispersion in the B\"acklund equations with parameter $1+\delta:$
\begin{align}
 \widetilde{Q}_x&= \dfrac{1}{1+ \delta}\cos\left(\frac{\widetilde{Q}}{2}\right)+ (1+\delta) \cos\left(\frac{\widetilde{Q}}{2}\right),\label{perm_k_bajada1_a}
    \\ \widetilde{Q}_t &=\dfrac{1}{1+ \delta}\cos\left(\frac{\widetilde{Q}}{2}\right)-(1+\delta)\cos\left(\frac{\widetilde{Q}}{2}\right), \label{perm_k_bajada2_a}
\end{align} 
in the sense that $a(\beta_2)=1+\delta$ (see \eqref{a(beta)}), and  $(\widetilde Q, \widetilde Q_t)=(\widetilde Q, \widetilde Q_t)(x;\beta_2,0)$ be a kink profile of parameters $\beta_2\in (-1,1)$ and shift $x_0=0.$
\end{defn}


The following result shows that both definitions of final speed are equivalent. 

\begin{lem}\label{betas}
For a given solution $\vec \phi =(\phi,\phi_t)$ of SG with initial data of the form $(Q^0,0)+ (\widetilde u_0,\widetilde s_0) \in \mathcal M_\eta$ (see \eqref{Manifold}), one has  $\beta_1=\beta_2=:\beta$. 
\end{lem}

\begin{proof}
First of all, we recall that from \eqref{momentum_beta1} we have $P[\vec \phi] = P[Q^0 + \widetilde u_0, \widetilde s_0] =: P[(Q,Q_t)(\cdot; \beta_1,0)].$ Now notice that, on the one-hand, since $(\widetilde u_0, \widetilde s_0)$ are constructed using \eqref{perm_k_bajada1} and \eqref{perm_k_bajada2}, $\delta$ given and $v_0=0$, from Lemma 
\ref{Mom_Backlund} we have
\[
P[Q^0 + \widetilde u_0, \widetilde s_0]  =  2\left(\frac1{1+\delta} - (1+\delta) \right) = P[(Q,Q_t)(\cdot; \beta_2,0)] .
\]
On the other hand, by Definition \ref{beta1} we have $P[(Q,Q_t)(\cdot; \beta_1,0)]  = -4 \beta_1 (1-\beta_1^2)^{-1/2}$. Hence, recalling that $a(\beta_2)=1+\delta$ and by using \eqref{a(beta)} we obtain 
\[
  2\left(\frac1{1+\delta} - (1+\delta) \right) = -\frac{4\beta_2}{(1-\beta^2_2)^{1/2}}, \ \, \hbox{ and therefore } \ \, \frac{\beta_1}{(1-\beta^2_1)^{1/2}}=\frac{\beta_2}{(1-\beta^2_2)^{1/2}}.
\]
Finally, noticing that $f(x)=\tfrac{x}{(1-x^2)^{1/2}}$ is injective for $x\in(0,1)$, we conclude the desired result.   
\end{proof}

Consequently, in what follows, we will work with kink profiles of the form 
\[
(Q,Q_t)(\cdot; \beta,x_0), 
\]
with $\beta$ given by any of the equivalent definitions \ref{beta1} or \ref{beta2}, and $x_0\in\R$, possibly depending on time, to be chosen later.

\subsection{Rewriting of the initial data}

We need an additional makeover to fully introduce modulations, related to the initial data for general nonzero momentum. Recall that from \eqref{Manifold} we have the initial data of the form
\be\label{Initial_data_1}
(Q^0,0) + (\widetilde u_0, \widetilde s_0) \quad \hbox{s.t.}  \quad (\widetilde u_0, \widetilde s_0):=\Phi(y_0, 0,\delta), \quad y_0 \in H^1_o, \quad \|y_0\|_{H^1} +|\delta|<\eta.
\ee
These initial data can be conveniently written as follows:
\be\label{Initial_data_2}
(Q^0,0) + (\widetilde u_0, \widetilde s_0)= (Q,Q_t)(\cdot, \beta,0) + (\hat u_0, \hat s_0),
\ee
where
\be\label{Initial_data_3}
(\hat u_0, \hat s_0):= ((Q^0,0) -(Q,Q_t))(\cdot, \beta,0)+ (\widetilde u_0, \widetilde s_0) \in H^1_o\times L^2_e.
\ee
Note that, written this way, $(\hat u_0, \hat s_0)$ is still (odd, even) and small enough. Moreover, the solutions to SG with initial data \eqref{Initial_data_1} and \eqref{Initial_data_2} are just the same.

\begin{rem}
The makeover \eqref{Initial_data_2} is made to precisely catch the final speed of the kink in the case where the data has general nonzero momentum. However, in Theorem \ref{MT0}, since the data has zero momentum,  \eqref{Initial_data_2} will not be necessary, in the sense that we can take $\beta=0$ (since $\delta=0$) on the right hand side and $(\widetilde u_0, \widetilde s_0)=(\hat u_0, \hat s_0)$.
\end{rem}

\subsection{Modulation}

Now we are ready to show a modulation result for the solution close to the kink profile. For a similar statement, see e.g. \cite{AM1,MP}. Since Theorem \ref{MT0} is equivalent for positive and negative times, we only consider the positive time case.

\medskip

Let $(\phi,\phi_t)$ be the solution to SG \eqref{sg1} with initial data \eqref{Initial_data_2}. The perturbation data $(\hat u_0, \hat s_0)$ is assumed small, depending on $\eta\in (0,\eta_0)$, parameter of the problem. Define, for $K^*>1$ large,
\be\label{T_mod}
\begin{aligned}
T^*:= &~{} \sup \Big\{ T>0 ~  : ~  \forall~ t\in[0,T], \, \exists~ \tilde \rho(t)\in\R ~ s.t. ~  \\
&~{} \qquad\qquad\qquad   \|(\phi,\phi_t)(t)-  (Q,Q_t)(x; \beta,\beta t+ \tilde \rho(t))\|_{H^1\times L^2 } <K^* \eta \Big\}.
\end{aligned}
\ee
Clearly $T^*>0$ because of the continuity of the SG flow, the assumption  \eqref{initial_data} on the initial data and the fact that $K^*>1$. Suppose, as in \cite{AM,MP}, that $T^*<+\infty$. 

\begin{lem}[Modulation]\label{Ortho}
By taking $\eta_0$ smaller if necessary, the following is satisfied. Let $\vec \phi_0:=(Q,Q_t)(\cdot, \beta,0) + (\hat u_0, \hat s_0)$ be given initial data as in \eqref{Initial_data_2}, and let $\vec \phi(t)$ be the corresponding solution to SG, with $\vec \phi(t=0)=\vec \phi_0$. Then, there is  $\rho(t)\in \R$ such that, 
\be\label{Ortho1}
(\widetilde u,\widetilde s)(t,x) := (\phi,\phi_t)(t,x)-  (Q,Q_t)(x; \beta,\beta t + \rho(t))
\ee
satisfies, for all $t\in[0,T^*]$,
\be\label{Ortho2}
\int (\widetilde u,\widetilde s)(t,x)\cdot (\widetilde Q_x,\widetilde Q_{t,x})(x;\beta,\beta t + \rho(t)) dx=0.
\ee
Moreover, under this condition we have the following dynamical equations for $(\widetilde u,\widetilde s)$ and $\rho'(t):$
\be\label{rho(0)}
 \rho(0)=0, \quad (\widetilde u,\widetilde s)(t=0)=(\hat u_0,\hat s_0),
\ee
and
\be\label{eqn_us}
\left\{\begin{array}{ll}
\widetilde u_t = \widetilde s+\rho' \widetilde Q_x,
\\ \widetilde s_t=\widetilde u_{xx} -\sin( Q+\widetilde u) +\sin Q + \rho' \widetilde Q_{t,x}.
\end{array}\right.
\ee
\end{lem}

\begin{proof}
Let $z_0\in\R$ be fixed. The idea of the proof is to use the Implicit Function Theorem to ensure the existence of such decomposition with $(\widetilde u,\widetilde s)(t,x)$ satisfying the orthogonality condition \eqref{Ortho2}. First of all, let us consider the neighborhood \[
\mathcal{U}(\nu):=\{(\phi,\psi)\in H^1_{loc}\times L^2_{loc}: \ \Vert (\phi,\psi)-(\widetilde Q,\widetilde Q_t)(x;\beta,z_0)\Vert_{H^1\times L^2}<\nu\}.
\]
Note that even when $(\phi,\psi)$ does not belong to $H^1\times L^2$, its difference with $(\widetilde Q,\widetilde Q_t)$ does. Now, define the functional $Y:\mathcal{U}(\eta)\times(-\eta,\eta)\to\R$, given by
\[
Y(\phi,\psi, \rho):=\int_\R\big((\phi,\psi)(x)-(\widetilde Q,\widetilde Q_t)(x;\beta,z_0+ \rho)\big)\cdot (\widetilde Q_x,\widetilde Q_{t,x})(x;\beta,z_0+ \rho) dx.
\]
It is clear that $Y$ is a $C^1$ functional. Moreover, we have \begin{align*}
\dfrac{\partial Y}{\partial\rho}(\phi,\psi, \rho)&=\int_\R\Big(\widetilde Q^2_x(x;\beta,z_0+ \rho)+\widetilde Q^2_{t,x}(x;\beta,z_0+ \rho)\Big) dx
\\ & \quad + \int_\R\big((\phi,\psi)(x)-(\widetilde Q,\widetilde Q_{t})(x;\beta,z_0+ \rho)\big)\cdot (\widetilde Q_{xx},\widetilde Q_{t,xx})(x;\beta,z_0+ \rho) dx
\end{align*}
Finally, note that at the point $(\widetilde Q,\widetilde Q_t)(x;\beta,z_0)$ we have 
$Y(\widetilde Q,\widetilde Q_t,0)=0$ 
and 
\begin{align}\label{invertible}
\dfrac{\partial Y}{\partial \rho}(\widetilde Q,\widetilde Q_t,0)=-\int_\R \Big(\widetilde Q_x^2(x;\beta,z_0)+\widetilde Q_{t,x}^2(x;\beta,z_0)\Big)dx\neq 0.
\end{align}
Thus, by the Implicit Function Theorem we deduce the existence of $\widetilde\eta>0$ small enough and a $\mathcal{C}^1$ function $\rho(\phi,\psi)$ 
defined on the neighborhood 
$\mathcal{U}(\widetilde \eta)\times (-\widetilde \eta,\widetilde \eta)$ 
to a neighborhood of zero such that 
\[
Y(\phi,\psi, \rho)=0 \quad \hbox{for all }\  (\phi,\psi)\in \mathcal{U}(\widetilde \eta)\times (-\widetilde \eta,\widetilde \eta).
\]
Note that $\widetilde \eta$ only depends on \eqref{invertible} and not on the point $z_0\in\R$. Hence, for every $(\phi,\psi)\in \mathcal{U}(\widetilde\eta)\times(-\widetilde\eta,\widetilde\eta)$ we define the shift $\rho_{z_0}(\phi,\psi):=z_0+\rho(\phi,\psi)$. Finally, using the definition of $T^*$ we can define the mapping $t\mapsto \widetilde\rho(t)$ on $[0,T^*]$ by setting $\widetilde\rho(t):=\rho_{\beta t}((\phi,\phi_t)(t))$. Thus we obtain \[
\int (\widetilde u,\widetilde s)(t,x)\cdot (\widetilde Q_x,\widetilde Q_{t,x})(x;\beta,\beta t+\rho(t))=0, \quad \forall t\in[0,T^*].
\]
Finally, \eqref{rho(0)} is direct from the chosen type of initial perturbative data (odd, even) and \eqref{Ortho2}, and \eqref{eqn_us} are direct.  
%
%
\end{proof}

\medskip

\section{Lifting of the data}\label{9}

Let $(\phi_0,\phi_1)\in \mathcal M_\eta$ be initial data as described in Lemma \ref{Implicit} and \eqref{Manifold}. Then, there exist unique  $(y_0, \delta)\in H^1_o\times \R $ with   $ \|y_0\|_{H^1} +|\delta|<\eta$ and such that $(\widetilde u_0, \widetilde s_0)=\Phi(y_0, 0,\delta)\in H^1_o\times L^2_e$.  Consider the SG equation \eqref{sg1} with initial data $(y_0,0)$. Since this data is (odd, odd), the evolution preserves this property. Namely, there exists a unique global solution $(y,v) \in C(\R ; \in H^1_o\times L^2_o)$ such that 
\[
\| (y,v)(t)\|_{H^1\times L^2} \lesssim \eta.
\]
Note additionally that, since $v_0=0$ in this case, one has that $(y,v)(t)$ has zero momentum. 

\medskip

We shall use the results in \cite{KMM2} (see also \cite{Coron} for earlier results), which claim that (odd, odd) small perturbations of the zero state in SG must converge to zero in compact intervals of space, as time tends to infinity. This is the key part of the paper, stated in Theorem \ref{thmkmm2}, in the sense that if we are not able to get  (odd, odd) data around zero, then we cannot use \cite{KMM2}. The construction of the manifold (odd, even) of initial data around $Q$ is precisely the way we have to ensure that the data around zero have the right parity conditions. 

\medskip

Consider the (odd, odd) solution $(y,v)(t)\in H^1\times L^2 $ of SG mentioned in Theorem \ref{thmkmm2}, 
and constructed using the initial data in $\mathcal M_\eta$. The purpose of this Section is to connect this 
solution with the one described in Lemma \ref{Ortho}, eqn. \eqref{Ortho1}. 

\medskip

The main problem associated to this connection (if possible), is to arrive to the correct 
solution of SG. Here, the uniqueness of the solution (given the same initial data) will be essential to conclude this property. The correct choice of data will be given by \eqref{Ortho1}.

\medskip

\subsection{Lifting via modified Implicit Function} Indeed, let $t\in [0,T^*]$ be fixed, and consider $(Q,Q_t)= (Q,Q_t)(x; \beta,\beta t + \rho(t))$ and its corresponding modified profile $(\widetilde Q, \widetilde Q_t):= (\widetilde Q, \widetilde Q_t)(x; \beta,\beta t + \rho(t))$ (see \eqref{tildeQ}).
Using \eqref{f1}-\eqref{f2}, and given $(y(t),v(t),\delta)\in H^1_o\times L^2_o\times \R$, we will look for a solution $(\hat u,\hat s)(t) \in H^1\times L^2$ of\footnote{Note that $(\hat u,\hat v)(t)$ have no longer a parity property, consequence of the shift $\rho(t)$ and (if nonzero), the speed parameter $\beta$.} 
\begin{align}
 \widetilde{Q}_x+\hat u_{x}-v &= \dfrac{1}{1+\delta}\cos\left(\dfrac{\widetilde{Q}+ \hat u+ y}{2}\right)+ (1+\delta) \cos\left(\dfrac{\widetilde{Q}+\hat u-y}{2}\right),\label{a111}
    \\ \widetilde{Q}_t+ \hat s - y_{x}&=\dfrac{1}{1+\delta}\cos\left(\dfrac{\widetilde{Q}+\hat u+y}{2}\right)-(1+\delta)\cos\left(\dfrac{\widetilde{Q}+\hat u-y}{2}\right),\label{a112}\\
 \text{with}\quad  & 
 \int (\hat u,\hat s)(t,x)\cdot (\widetilde Q_x,\widetilde Q_{t,x})(x;\beta,\beta t + \rho(t)) dx = 0.\label{a113}
\end{align} 
As in \cite{HW}, \cite{AM1} and \cite{MP}, the extra orthogonality condition \eqref{a113} is essential to uniquely solve this nonlinear system.

\medskip

Following the proof of Lemma \ref{Implicit}, solving for $\hat s$ is trivial once we solve for $\hat u$. Hence, we must solve
\be\label{EDO1_new}
\hat u_{x}  +\frac12 \left( \frac{1}{(1+ \delta)} + (1+\delta)\right) \sin\left(\frac{\widetilde{Q}}{2}\right)\hat u  = f, \qquad \hat u \in H^1,
\ee
for any $f \in L^2$.  We have $ \sin\left(\frac{\widetilde{Q}}{2}\right)= \tanh (\gamma (x-\beta t -\rho(t))),$ and \eqref{EDO1_new} becomes
\[
\left( \hat u\cosh^{\nu_0} (x-\beta t -\rho(t)) \right)_x = f \cosh^{\nu_0} (x-\beta t -\rho(t)),\quad \nu_0 = \frac1{2\gamma} \left( \dfrac{1}{(1+ \delta)} + (1+\delta)\right).
\]
Hence, 
\[
\begin{aligned}
 \hat u (x) =&~{}  \hat u (x=\beta t+ \rho(t)) \cosh^{-\nu_0} (x-\beta t -\rho(t)) \\
 &~{} + \cosh^{-\nu_0} (x-\beta t -\rho(t)) \int_{\beta t + \rho(t)}^{x}  f(s) \cosh^{\nu_0} (s-\beta t -\rho(t)) \, ds.
\end{aligned} 
\]
Solving for $\hat s$ is also direct. Consequently, since from \eqref{Q} we have 
$\widetilde Q_x(x;\beta,\beta t+ \rho(t)) = 2\ga\sech(\ga(x-\beta t -\rho(t)))$, we conclude that from the orthogonality condition in \eqref{a113} we have $\hat u (x=\beta t+ \rho(t))$ uniquely determined. The rest of the proof is direct (as in Lemma \ref{Implicit}), or as in \cite{MP}.

\subsection{Uniqueness of the lifted data}\label{uniqueness} Now we discuss the uniqueness of the solution found, that is, $(\widetilde{Q},\widetilde Q_t)+ (\hat u,\hat s)$. Since $(y,v)(t)$ is solution to SG \eqref{sg1}, by Lemma \ref{Solutions_Backlund} we have that $(Q, Q_t)(t)+ (\hat u,\hat s)(t)$ also solves SG. The initial data associated to this solution is
\[
(Q, Q_t)(x;\beta,0)+ (\hat u,\hat s)(0),
\]
since $\rho(0)=0$ (see \eqref{rho(0)}). Now, recall that $(y,v)(t=0)=(y_0,v_0)=(y_0,0)$, and by uniqueness associated to the Implicit Function (Lemma \ref{Implicit}), we know that $(\widetilde u_0,\widetilde s_0)=\Phi(y_0,0,\delta)$. Consequently,
\[
\begin{aligned}
(Q, Q_t)(x;\beta,0)+ (\hat u,\hat s)(0)= (Q^0,0) + (\widetilde u_0, \widetilde s_0),
\end{aligned}
\]
and therefore $(\hat u,\hat s)(0)=  ((Q^0,0) -(Q,Q_t)(\cdot, \beta,0))+ (\widetilde u_0, \widetilde s_0)=(\hat u_0, \hat s_0)$ (see \eqref{Initial_data_2} and \eqref{Initial_data_3}).

\medskip

The previous argument guarantees that the lifted data $(Q, Q_t)(t)+ (\hat u,\hat s)(t)$ is just $(Q, Q_t)(t)+ (\widetilde u,\widetilde s)(t)$ as in Lemma \ref{Ortho}. Moreover, we have also proved that the kink is orbitally stable (see earlier results by \cite{HPW,HW}).

\medskip

\section{Estimates on the shift}\label{10}

In this section we prove further estimates on the shift $\rho(t)$ for the case $\beta=0$ that will allow us to prove the remaining parts in Theorem  \ref{MT0}. Our aim is to prove Corollary \ref{quadratic_rho_prop_2}, which relates $\rho'(t)$ with exponentially weighted, quadratic integrals only depending on $(y,y_x,v)$. We start out with the following mixed estimate. 

\begin{lem}\label{quadratic_rho_prop}
Assume $\beta=0$. Under the hypothesis of Lemma \ref{Ortho}, for all $t\in\R$ the following bound holds: for any $\varepsilon>0$ small, 
\begin{align}\label{dot_rho_quadratic}
\vert \rho'(t)\vert\lesssim \int e^{-(1+\varepsilon)\vert x-\rho(t)\vert}(\widetilde u^2+\widetilde{u}_x^2)(t,x)dx+\int e^{-(1-\varepsilon)\vert x-\rho(t)\vert}(y^2+y_x^2)(t,x)dx,
\end{align}
with implicit constant independent of time and $\widetilde u$, $y$.
\end{lem}

\begin{rem}[About quadratic estimates and convergence]\label{convergencia_1}
Estimate \eqref{dot_rho_quadratic} reveals that, under the orthogonality condition \eqref{Ortho2} in the case $\beta=0$, the derivative of $\rho(t)$ is of quadratic order in  $\widetilde u$ and $y$, a fact that should imply that $\rho(t)$ may converge as $t\to +\infty$. Unfortunately, there is no simple relationship between the weight $e^{-(1-\varepsilon)\vert x-\rho(t)\vert}$ and the dynamics of $(y^2+y_x^2)(t,x)$. This lack of evident connection for data only in the energy space makes the proof of convergence for $\rho(t)$ harder than usual. In any case, we are able to show in this paper (Theorem \ref{MT0}) that either $\rho(t_n)$ diverges for some sequence $t_n\to +\infty$, or it converges to a final state $\bar \rho.$ In both cases, a portion of the radiation term $\widetilde u$ \emph{converges in the energy space}, on compact sets. 
\end{rem}

\begin{rem}[Conditional convergence]\label{convergencia_2}
A conditional result for convergence of $\rho(t)$ in every possible case is the following: if the data $(y_0,0)\in \mathcal M_{\eta,0} \subseteq H_o^1\times L^2_{e}$ is such that the solution of SG $(y,v)(t)\in H^1_o\times L^2_o$ with initial data at $t=0$ given by $(y_0,0)$ satisfies
\[
 \int_0^\infty \!\!\! \int e^{-(1-\varepsilon)\vert x-\rho(t)\vert}(y^2+y_x^2)(t,x)dxdt <+\infty,
\]
then $\rho(t)\to \bar\rho \in\R.$ This result will transpire from the proof of Theorem \ref{MT0}.
\end{rem}

\begin{proof}[Proof of Lemma \ref{quadratic_rho_prop}]
Recall that $Q_t=0$ in the $\beta=0$ case, see \eqref{Qt}.  In order to prove \eqref{dot_rho_quadratic} we multiply \eqref{eqn_us} by $Q'$ and integrate in space, from where we obtain
\be\label{primera_rho}
\int \widetilde u_tQ'=\int \widetilde sQ'+\rho'\int Q'^2.
\ee
Now notice that, due to the orthogonality condition and the uniformly smallness of $u(t,x)$ we have
\[
c\leq \left\vert \int Q'^2-\int\widetilde u_tQ'\right\vert \leq C,
\]
for some positive constants $c,C>0$. Therefore, using \eqref{backlund_s_simplified} in \eqref{primera_rho} we conclude 
\be\label{intermediario}
\vert \rho' (t)\vert\lesssim \left\vert\int y_xQ'-2\int Q'\left(\cos\left(\dfrac{\widetilde{Q}}{2}\right)\sin\left(\dfrac{\widetilde u}{2}\right) + \sin\left(\dfrac{\widetilde{Q}}{2}\right)\cos\left(\dfrac{\widetilde u}{2}\right) \right) \sin\left(\dfrac{y}{2}\right)\right\vert.
\ee
On the other hand, notice that by using $Q'=-2\cos(\widetilde{Q}/2)$ and basic trigonometric identities we can rewrite the first and last term on the right-hand side of the latter inequality as
\begin{align*}
\left\vert\int y_xQ'-2Q'\sin\left(\dfrac{\widetilde Q}{2}\right)\cos\left(\dfrac{\widetilde u}{2}\right)\sin\left(\dfrac{y}{2}\right)\right\vert=\left\vert \int y_xQ'+2\int \cos\left(\dfrac{\widetilde{u}}{2}\right)\sin\left(\dfrac{y}{2}\right)Q''\right\vert,
\end{align*}
which, by integration by parts, is equivalent to
\be\label{milagro}
\left\vert \int \left(1-\cos\left(\dfrac{\widetilde{u}}{2}\right)\cos\left(\dfrac{y}{2}\right)\right)Q'y_x+\int\sin\left(\dfrac{\widetilde{u}}{2}\right)\sin\left(\dfrac{y}{2}\right)Q'\widetilde{u}_x\right\vert.
\ee
Finally, we recall that by basic trigonometric bounds together with the uniform smallness of the functions involved we have
\[
0\leq 1-\cos\left(\dfrac{\widetilde{u}}{2}\right)\cos\left(\dfrac{y}{2}\right)\lesssim \widetilde{u}^2+y^2 \quad \hbox{and}\quad  \left| \sin\left(\dfrac{\widetilde{u}}{2}\right)\sin\left(\dfrac{y}{2}\right) \right| \lesssim \vert \widetilde{u} y\vert.
\]
Therefore, plugging together the last identities and by using $ab\leq 2a^2+2b^2$ we conclude
\be\label{SG_maldad}
\begin{aligned}
&\left\vert\int y_xQ'-2Q'\sin\left(\dfrac{\widetilde Q}{2}\right)\cos\left(\dfrac{\widetilde u}{2}\right)\sin\left(\dfrac{y}{2}\right)\right\vert
\\ & \qquad \lesssim \int e^{-(1-\varepsilon)\vert x-\rho(t)\vert}(y^2+y_x^2)+\int e^{-(1+\varepsilon)\vert x-\rho(t)\vert}(\widetilde{u}^2+\widetilde {u}_x^2),
\end{aligned}
\ee
for any $0<\varepsilon\ll1$ small. Notice that, in the same way as before, and by using the fact that $\cos \widetilde{Q}/2\lesssim \sech(x-\rho(t))$, we also conclude in \eqref{intermediario} that
\[
\left\vert \int \cos\left(\dfrac{\widetilde{Q}}{2}\right)\sin\left(\dfrac{\widetilde{u}}{2}\right)\sin\left(\dfrac{y}{2}\right)Q'\right\vert\lesssim \int e^{-2\vert x-\rho(t)\vert}(\widetilde u^2+y^2),
\]
which concludes the proof of the lemma.
\end{proof}

\begin{rem}
Note that the nice cancelation produced in \eqref{milagro} is essentially due to the fact that \eqref{calculo interesante} holds precisely under the orthogonality condition \eqref{Ortho2} in the case $\beta=0$. 
\end{rem}

Now our objective is to eliminate the term in $\widetilde u_x^2$ in \eqref{dot_rho_quadratic}.

\begin{lem}
Assume $\beta=0$. Under the assumptions of Lemma \ref{Ortho}, for any $0<\varepsilon\ll1$ small and all $t\in\R$ the following bound holds
\begin{align}\label{quadratic_u_yv}
\int e^{-(1+\varepsilon)\vert x-\rho(t)\vert}\widetilde{u}_x^2\lesssim \int e^{-(1+\varepsilon)\vert x-\rho(t)\vert}(\widetilde{u}^2+y^2+v^2).
\end{align}
\end{lem}

\begin{proof}
To show \eqref{quadratic_u_yv} it is enough to notice that by using identity \eqref{CpC} and plugging it into equation \eqref{a111_new} we deduce 
\[
\widetilde{u}_x=v-\sin\left(\dfrac{\widetilde{Q}}{2}\right)\sin\left(\dfrac{\widetilde{u}}{2}\right)\cos\left(\dfrac{y}{2}\right)-2\left(1-\cos\left(\dfrac{\widetilde{u}}{2}\right)\cos\left(\dfrac{y}{2}\right)\right)\cos\left(\dfrac{\widetilde{Q}}{2}\right).
\]
Therefore, bounding in the same way as in the proof of Lemma \ref{quadratic_rho_prop} we obtain \begin{align*}
\int e^{-(1+\varepsilon)\vert x-\rho(t)\vert}\widetilde{u}_x^2\lesssim\int e^{-(1+\varepsilon)\vert x-\rho(t)\vert} (\widetilde{u}^2+y^2+v^2),
\end{align*}
which concludes the proof.
\end{proof}

The final step to prove Corollary \ref{quadratic_rho_prop_2} is a control of $\widetilde u^2$ in terms of $(y^2+y_x^2+v^2)$, with a reasonable loss in the decaying exponentials involved in the weighted norms.

\begin{lem}\label{10p3}
Under the assumptions of Lemma \ref{Ortho}, for any $0<\varepsilon\ll1$ small and all $t\in\R$ the following bound holds
\be\label{10p3_1}
\int \widetilde{u}^2(t,x)\sech^{1+\varepsilon}(x-\rho(t))dx\lesssim\int (y^2+y_x^2+v^2)(t,x)\sech^{1-\varepsilon}(x-\rho(t))dx,
\ee
with implicit constant independent of $t$ and $\widetilde u, y, v$.
\end{lem}

\begin{proof}
We start by rewriting the B\"acklund equation for $\widetilde{u}_x(t,x)$ in a more convenient form. In particular, notice that \eqref{a111_new} is equivalent to 
\begin{align*}
\widetilde{u}_x+\sin\left(\dfrac{\widetilde{Q}}{2}\right)\widetilde{u}=\widetilde{F}(t,x),
\end{align*}
where $\widetilde{F}(t,x)$ is given by 
\begin{align*}
\widetilde{F}(t,x)&:=v-2\cos\left(\dfrac{\widetilde{Q}}{2}\right)\left(1-\cos\left(\dfrac{\widetilde{u}}{2}\right)\cos\left(\dfrac{y}{2}\right)\right)+2\left(\dfrac{\widetilde{u}}{2}-\sin\left(\dfrac{\widetilde{u}}{2}\right)\cos\left(\dfrac{y}{2}\right)\right)\sin\left(\dfrac{\widetilde Q}{2}\right)
\\ & =: v+\mathrm{I}+\mathrm{II}.
\end{align*}
Therefore, by solving the ODE and due to the fact that $\sin\big(\tfrac{\widetilde{Q}}{2}\big)=\tanh(x-\rho(t))$ we obtain
\be\label{tilde_u_eqn}
\widetilde{u}(t,x)=b(t)\sech(x-\rho(t))+\sech(x-\rho(t))\int_{\rho(t)}^x\cosh(z-\rho(t))\widetilde{F}(t,z)dz.
\ee
Notice that direct computations and Cauchy-Schwarz's inequality yield us to 
\begin{align}
& \int \widetilde{u}^2\sech^{1+\varepsilon}(x-\rho(t))dx \nonumber \\
&\lesssim b^2(t)+\int \sech^{3+\varepsilon}(x-\rho(t))\left(\int_{\rho(t)}^x\cosh(z-\rho(t))\widetilde{F}(t,z)dz\right)^2dx \label{444}
\\
& \lesssim b^2(t)+\int \sech^{3+\varepsilon}(x-\rho(t))\left(\int_{\rho(t)}^x\cosh^3(z-\rho(t))dz\right)\left(\int_{\rho(t)}^x\sech(z-\rho(t))\widetilde{F}^2(t,z)dz\right)dx
\nonumber \\
 & \lesssim b^2(t)+\int\sech(x-\rho(t))\widetilde{F}^2(t,x)dx.\nonumber 
\end{align}
Now we claim that there exists $0<\delta\ll1$ such that 
\be\label{555}
\int \sech(x-\rho(t))\widetilde{F}^2(t,x)\lesssim\int \sech^{1-\varepsilon}(x-\rho(t))(y^2+y_x^2+v^2)+\delta\int\sech^{1+\varepsilon}(x-\rho(t))\widetilde{u}^2.
\ee
In fact, by using Cauchy-Schwarz's inequality we obtain
\begin{align*}
\int \vert \mathrm{I}\vert^2\sech(x-\rho(t))&\lesssim \int \sech^4(x-\rho(t))(\widetilde{u}^4+y^4)
\\ & \lesssim \sup_{t\in\R}\Vert y(t)\Vert_{H^1}^2\int\sech(x-\rho(t))y^2
\\ & \quad + \sup_{t\in\R}\Vert \widetilde u(t)\Vert_{H^1}^2\int\sech^{1+\varepsilon}(x-\rho(t))\widetilde{u}^2.
\end{align*}
On the other hand, by standard trigonometric inequalities and the uniform smallness of the functions involved we have \begin{align*}
\left\vert\dfrac{\widetilde{u}}{2}-\sin\left(\dfrac{\widetilde{u}}{2}\right)\cos\left(\dfrac{y}{2}\right)\right\vert\lesssim \vert \widetilde{u}\vert y^2,
\end{align*}
and hence, by using again $2ab\leq a^2+b^2$ we conclude 
\begin{align*}
\int\vert \mathrm{II}\vert^2 \sech(x-\rho(t))&\lesssim \sup_{t\in\R}\Vert u(t)\Vert_{H^1}^2\sup_{t\in\R}\Vert y(t)\Vert_{H^1}^2\int\sech^{1+\varepsilon}(x-\rho(t))\widetilde{u}^2
\\ & \quad +\sup_{t\in\R}\Vert y(t)\Vert_{H^1}^4\int \sech^{1-\varepsilon}(x-\rho(t))y^2,
\end{align*}
what concludes the claim. Finally, we are lead to estimate $b(t)$. 

\medskip

Using \eqref{Ortho2} in the case $\beta=0$, and \eqref{Qx}-\eqref{Qt} (recall that $\ga=1$), we have in \eqref{tilde_u_eqn}
\[
|b(t)|^2 \lesssim  \left(\int \sech^2(x-\rho(t)) \left| \int_{\rho(t)}^x\cosh(z-\rho(t))\widetilde{F}(t,z)dz \right|dx \right)^2.
\]
By Cauchy-Schwarz,
 \be\label{SG_diablo_2}
|b(t)|^2 \lesssim  \int \sech^{4-}(x-\rho(t)) \left| \int_{\rho(t)}^x\cosh(z-\rho(t))\widetilde{F}(t,z)dz \right|^2 dx .
\ee
Proceeding as in \eqref{444}, we conclude 
 \[
|b(t)|^2 \lesssim \int\sech(x-\rho(t))\widetilde{F}^2(t,x)dx .
\]
Gathering this last estimate and \eqref{555}, we conclude.
\end{proof}

We finish this section gathering the main result concerning the quadratic behavior of $\rho'(t)$ in the case $\beta=0$.

\begin{cor}\label{quadratic_rho_prop_2}
Assume $\beta=0$. Under the assumptions of Lemma \ref{Ortho}, for all $t\in\R$ the following bound holds: for any $0<\varepsilon\ll1$ small and fixed, 
\begin{align}\label{dot_rho_quadratic_2}
\vert \rho'(t)\vert\lesssim \int e^{-(1-\varepsilon)\vert x-\rho(t)\vert}(v^2+y^2+y_x^2)(t,x)dx,
\end{align}
with implicit constant independent of time and $(y,v)$.
\end{cor}

\medskip

This result proves \eqref{eta2} in Theorem \ref{MT0}, part (2). It only remains to show part (2), equations \eqref{final_data0_0} and \eqref{final_data0_1}.

\medskip

\section{End of proof of Theorem \ref{MT0}}\label{11}

In this Section we prove Theorem \ref{MT0}, estimates \eqref{final_data0_0} and \eqref{final_data0_1}. Recall that parts (1) and (3) of this result were already proved in Section \ref{7} and \ref{8}. 

\medskip

Let us consider $\vec\phi_0:= (\phi_0,\phi_1)$ data belonging to the manifold $\mathcal M_\eta$ in \eqref{Manifold}. Let also $(\phi(t),\phi_t(t))$ be the unique solution of \eqref{sg1} with initial condition $(\phi,\phi_t)(0)=(\phi_0,\phi_1)$. We restrict ourselves to the zero momentum submanifold $\mathcal M_{\eta,0}$ introduced in \eqref{Manifold0}.  In particular, $\beta=0$ and $\ga=1$ in the kink profile (see \eqref{tildeQ}) 
\[
\widetilde Q(x;\beta,\beta t + \rho(t))=\widetilde Q(x;0, \rho(t)) =Q(x-\rho(t)) -\pi =: \widetilde Q(x-\rho(t)).
\]
Moreover, $\delta=0$ in \eqref{Manifold0}.

\medskip

Let us fix now a compact interval $I$. We divide the proof into several steps. The first two of them concern the limit of the $L^2$-norm of the functions $(\widetilde u,\widetilde s)$. We shall only consider the case  in which $t\to\infty$. Nevertheless, the same proof holds for the case when $t\to-\infty$ up to some obvious modifications.

\medskip

\textbf{Step 1.} First of all consider \eqref{a111}-\eqref{a112} with $\hat u=\widetilde u$ (see the discussion in Subsection \ref{uniqueness}):
\begin{align}
 \widetilde{Q}_x+\widetilde u_{x}- v &= \cos\left(\dfrac{\widetilde{Q}+\widetilde u+y}{2}\right)+  \cos\left(\dfrac{\widetilde{Q}+\widetilde u-y}{2}\right),\label{a111_new}
    \\ \widetilde{Q}_t+\widetilde s-y_{x}&=\cos\left(\dfrac{\widetilde{Q}+\widetilde u+y}{2}\right)-\cos\left(\dfrac{\widetilde{Q}+\widetilde u-y}{2}\right).\label{a112_new}
\end{align} 
Using that 
\[
\begin{aligned}
\cos(A+B+C) =&~{}  \cos A \cos B \cos C - \cos A \sin B \sin C \\
&~{} - \cos B \sin A \sin C - \cos C \sin A \sin B,
\end{aligned}
\]
\noindent 
we obtain
\be\label{CpC}
\begin{aligned}
& \cos\left(\dfrac{\widetilde{Q}+\widetilde u+y}{2}\right)+  \cos\left(\dfrac{\widetilde{Q}+\widetilde u-y}{2}\right) \\
&~{} \qquad = 2\left(\cos\left(\dfrac{\widetilde{Q}}{2}\right)\cos\left(\dfrac{\widetilde u}{2}\right) -\sin\left(\dfrac{\widetilde{Q}}{2}\right)\sin\left(\dfrac{\widetilde u}{2}\right) \right) \cos\left(\dfrac{y}{2}\right),
\end{aligned}
\ee
and
\be\label{CmC}
\begin{aligned}
& \cos\left(\dfrac{\widetilde{Q}+\widetilde u+y}{2}\right)-  \cos\left(\dfrac{\widetilde{Q}+\widetilde u-y}{2}\right)\\
&~{} \qquad = -2\left(\cos\left(\dfrac{\widetilde{Q}}{2}\right)\sin\left(\dfrac{\widetilde u}{2}\right) + \sin\left(\dfrac{\widetilde{Q}}{2}\right)\cos\left(\dfrac{\widetilde u}{2}\right) \right) \sin\left(\dfrac{y}{2}\right).
\end{aligned}
\ee
Both identities will be important in what follows. First of all, from \eqref{CmC} and the fact that $ \widetilde{Q}_t =0$ in \eqref{a112_new} (since $\beta=0$) one has
\begin{align}\label{backlund_s_simplified}
\widetilde s =y_x -2\left(\cos\left(\dfrac{\widetilde{Q}}{2}\right)\sin\left(\dfrac{\widetilde u}{2}\right) + \sin\left(\dfrac{\widetilde{Q}}{2}\right)\cos\left(\dfrac{\widetilde u}{2}\right) \right) \sin\left(\dfrac{y}{2}\right).
\end{align}
Recall the identities (see \cite{MP})
\be\label{sinQ_cosQ}
\sin\left(\frac{\widetilde Q}{2}\right) = \tanh ( x -\rho(t)), \qquad \cos\left(\frac{\widetilde Q}{2}\right) = \sech(x -\rho(t)),
\ee
and 
\be\label{easy}
|\sin x| \leq |x| \quad  \hbox{and} \quad  0\leq 1-\cos x \leq \min\left\{2,\frac12 x^2\right\}\quad  \hbox{(valid for all $x\in\R$).}
\ee
We have a.e. in $\R$
\be\label{QQQ}
|\widetilde s(t,x)| \lesssim |y_x(t,x)|  + |y(t,x)|,
\ee
so that, from \eqref{Integration0} and \eqref{AS_zero},
\begin{align}\label{epsilon_plus_limit}
\int\int e^{-c_1|x|}\widetilde s^2(t,x)dxdt <+\infty, \qquad \lim_{t\to \pm\infty} \|\widetilde s(t)\|_{L^2(I)} =0,
\end{align}
for any compact interval $I$. This proves the second component part of Theorem \ref{MT0} in \eqref{final_data0_0} and \eqref{final_data0_1} 
(note that no particular sequence of times $t_n\to +\infty$ is needed in this case).

\medskip

\textbf{Step 2.} In the case $\beta=0$, we have from \eqref{Qt} and  \eqref{a113}
\[
\int \widetilde u Q'(x-\rho(t))=0.
\]
Taking derivative and using \eqref{eqn_us},
\[
\rho' \left( \int Q'^2 + \int \widetilde u Q'' (x-\rho(t))\right) =-\int \widetilde s Q'(x-\rho(t)).
\]
On the other hand, from the conservation of (zero) momentum, we have
\[
0= \int (Q'(x-\rho(t)) +\widetilde u_x) \widetilde s \implies  -\int  \widetilde s Q'(x-\rho(t)) = \int \widetilde u_x \widetilde s.
\]
We conclude that $\rho'$ is of quadratic order, and satisfies
\[
\abs{\rho'(t)} \lesssim \left|\int \widetilde u_x \widetilde s\right| \lesssim \eta^2.
\]
This also is another proof of \eqref{eta2}.

\medskip

\textbf{Step 3.} In order to estimate the $\dot H^1$-norm of $\widetilde u$ we use \eqref{CpC} to re-write \eqref{a111_new} as
\be\label{EDO_aux}
\begin{aligned}
& ~{} \widetilde u_{x}(t,x)  + 2 \sin\left(\dfrac{\widetilde{Q}}{2}\right) \sin \left( \frac{\widetilde u(t,x)}{2}\right) \\
&~{} \qquad +2 \cos\left(\dfrac{\widetilde{Q}}{2}\right) \left( 1-\cos \left( \frac{\widetilde u(t,x)}{2}\right) \right) \cos\left( \frac{y}{2}\right)= F(t,x),
\end{aligned}
\ee
where $F$ is given by
\begin{align*}
F:= &~{} -  \widetilde{Q}_x + v +2 \cos\left(\dfrac{\widetilde{Q}}{2}\right) \cos\left(\dfrac{y}{2}\right) - 2  \sin\left(\dfrac{\widetilde{Q}}{2}\right) \left( \cos\left(\dfrac{y}{2}\right) - 1\right)\sin\left(\dfrac{\widetilde u}{2}\right) \\
=&~{} v + 2 \cos\left(\dfrac{\widetilde{Q}}{2}\right)  \left( \cos\left(\dfrac{y}{2}\right) -1\right) - 2  \sin\left(\dfrac{\widetilde{Q}}{2}\right) \left( \cos\left(\dfrac{y}{2}\right) - 1\right)\sin\left(\dfrac{\widetilde u}{2}\right) .
\end{align*}
In the last line, we have used \eqref{eqn:BT_Q} which implies that $\widetilde Q_x = 2\cos\left(\frac{\widetilde Q}{2}\right)$. Note that since $\widetilde u\in L^\infty (\R; H^1(\R))$, we have $\widetilde u$ bounded independently of time. Therefore, $F$ is bounded in $L^\infty(\R)$, uniformly in time. 

\medskip

Now we prove that the $L^2_x(I)$-norm of $F(t)$ goes to zero as $t\to+\infty$. In fact, 
\be\label{F_defi}
\begin{aligned}
\left\| F  \right\|_{L^2(I)}  &~{} \lesssim \|v\|_{L^2(I)} + \left\| \cos\left(\dfrac{\widetilde{Q}}{2}\right)\right\|_{L^2(I)} \left\|   \cos\left(\dfrac{y}{2}\right) -1  \right\|_{L^\infty(I)}\\
&~{} \quad+ \left\| \sin\left(\dfrac{\widetilde{Q}}{2}\right)\right\|_{L^\infty(I)} \left\| \cos\left(\dfrac{y}{2}\right) -1\right\|_{L^\infty(I)} \left\| \sin\left(\dfrac{\widetilde u}{2}\right)\right\|_{L^2(I)}.
\end{aligned}
\ee
On the other hand, the identities \eqref{sinQ_cosQ} imply that
\[
\left\| F  \right\|_{L^2(I)} \lesssim \|v\|_{L^2(I)} + \left\|   \cos\left(\dfrac{y}{2}\right) -1  \right\|_{L^\infty(I)}.
\]
Hence, using Theorem \ref{thmkmm2}, the inequalities \eqref{easy}, and the continuous embedding of $H^1(\R)$ into $L^\infty(\R)$ we obtain 
\[
\lim_{t\to+\infty}\left\Vert \cos\left(\dfrac{y}{2}\right)-1 \right\Vert_{L^\infty(I)}=0.
\]
Combining the previous limit and applying Theorem \ref{thmkmm2} again, we conclude 
\begin{align}\label{limit_I}
\lim_{t\to+\infty}\Vert F(t)\Vert_{L^2_x(I)}=0.
\end{align}
Actually, we can prove even more. From \eqref{Integration0}, \eqref{sinQ_cosQ} and the second line in \eqref{F_defi}, 
\[
\begin{aligned}
\int e^{-c_0|x|} F^2 (t,x)dx \lesssim &~{}  \int e^{-c_0|x|}v^2(t,x)dx \\
&~{} +  \int e^{-c_0|x|}   \sech^2(x -\rho(t))y^4(t,x) dx  +\int e^{-c_0|x|} \widetilde u^2 y^4(t,x) dx\\
\lesssim &~{}\int e^{-c_1|x|}v^2(t,x)dx +  \eta^2 \int e^{-c_1|x|} y^2(t,x)dx.
\end{aligned}
\]
Consequently, we obtain the stronger property
\be\label{Integrability_1}
\int \int e^{-c_0|x|} F^2 (t,x)dx dt <+\infty.
\ee
Assume now that $\|\widetilde u(t)\|_{L^2(I)}$ tends to zero as $t\to +\infty$. Using \eqref{EDO_aux} and \eqref{limit_I},
\[
\| \widetilde u_{x}(t) \|_{L^2(I)}  \lesssim  \| \widetilde u(t) \|_{L^2(I)} + \| F(t) \|_{L^2(I)} \to 0 
\]
as $t\to +\infty.$ This last result shows that we only need to prove $L^2_{loc}$ decay on $\widetilde u(t)$.

\medskip

\textbf{Step 4. Proof of AS.} Here we shall consider two cases.

\medskip
{\bf Subcase 1.} $|\rho(t_n)| \to +\infty$ for some $(t_n)$ tending to $+\infty$. With no loss of generality, we assume $\rho(t_n) \to +\infty$. Fix $x\in I$. Now the ODE \eqref{EDO_aux} reads
\[
\widetilde u_{x}(t,x)  - 2  \sin \left( \frac{\widetilde u(t,x)}{2}\right)  = \widetilde F (t,x)   ,
\]
where
\[
\begin{aligned}
\widetilde F(t,x):= &~{} F(t,x) -   2\left(1+ \sin\left(\dfrac{\widetilde{Q}}{2}\right)\right) \sin \left( \frac{\widetilde u(t,x)}{2}\right) \\
&~{} - 2 \cos\left(\dfrac{\widetilde{Q}}{2}\right) \left( 1-\cos \left( \frac{\widetilde u(t,x)}{2}\right) \right) \cos\left( \frac{y}{2}\right).
\end{aligned}
\]
Note that due to the fact that $\widetilde{u}\in L^\infty(\R,H^1(\R))$ and the explicit form of $\widetilde{Q}$ we deduce \[
\sup_{t\in\R}\|\widetilde F(t)\|_{L^2(\R)} <+\infty.
\]
Now, we conveniently rewrite \eqref{EDO_aux} as
\be\label{EDO_new}
\widetilde u_{x}(t,x)  + V(t,x) \widetilde u(t,x) =\widetilde F(t,x), 
\ee
where $V(t,x)$ is given by
\[
 V(t,x):=\begin{cases}- \frac2{\widetilde u(t,x)} \sin \left( \frac{\widetilde u(t,x)}{2}\right), & \widetilde u(t,x)\neq 0, \\
-1, & \widetilde u(t,x)=0.\end{cases}
\]
Clearly $V$ defined a bounded function in $(t,x)$. 
Thus, solving this ODE we obtain the explicit solution
\begin{align}\label{sol_EDO_u}
\widetilde u(t,x)   =   -\int_{x}^\infty e^{  \int_x^s V(t,\sigma)d\sigma } \widetilde{F}(t,s)ds.
\end{align}
Using that $\sup_{x\in\R }| \widetilde u(t,x)| \lesssim \eta \ll1$, uniformly for $t\in\R$, there exist $\nu=\nu(\eta)\ll1$ such that 
\[
-1-\nu\leq \sup_{x\in\R}V(t,x)\leq -1+\nu,
\]
uniformly for $t\in\R$. Replacing this in \eqref{sol_EDO_u} we obtain 
\[
\vert \widetilde u(t,x)\vert\leq\int_x^\infty e^{-c_0(s-x)}\vert \widetilde{F}(t,s)\vert ds=\int_0^\infty e^{-c_0w} \vert \widetilde{F}(t,w+x)\vert dw, \quad c_0:=(1-\nu).
\]
On the other hand, since all the functions involved belong to $L^\infty(\R;H^1(\R))$ we have 
\[
\Vert \widetilde{F}(t)\Vert_{L^\infty_x(\R)}\int_0^\infty e^{-c_0w}dw<\infty,
\]
uniformly for $t\in\R$. Finally, taking $L^2_x(I)$-norm and using that \eqref{limit_I} holds for any bounded interval, we get that
\begin{align*}
\lim_{n\to+\infty}\Vert \widetilde u(t_n)\Vert_{L^2_x(I)}&\leq \lim_{n\to\infty}\int_0^\infty e^{-c_0w}\Vert \widetilde{F}(t_n)\Vert_{L^2_x(I+w)}dw
\\ &\leq \int_0^\infty e^{-c_0w}\lim_{n\to\infty}\Vert \widetilde{F}(t_n)\Vert_{L^2_x(I+w)}dw =0,
\end{align*}
where we have used Minkowski's integral inequality and Dominated Convergence's Theorem, which concludes this step. Therefore, we conclude
\be\label{beee}
\lim_{n\to+\infty}\Vert(\widetilde{u}(t_n),\widetilde{s}(t_n))\Vert_{H^1_x(I)\times L^2_x(I)}=0.
\ee
This proves  \eqref{final_data0_0}.

\begin{rem}
Note that the in the case $\rho(t_n) \to +\infty$ only $\widetilde{u}(t_n)$ converges along a subsequence. The remaining part $\widetilde{s}(t)$ converges, locally on compact sets of space, along any subsequence of time, no matter if $\rho(t)$ stays bounded or not (see \eqref{epsilon_plus_limit}).
\end{rem}

{\bf Subcase 2. $\rho(t)$ is bounded.} In this case we have $e^{-(1-\varepsilon)\vert x-\rho(t)\vert} \leq C e^{-(1-\varepsilon)\vert x\vert}$ for all time, with $C>0$ depending on $\max_\R\rho(t)<\infty$. From \eqref{dot_rho_quadratic_2},
\be\label{cota_unif}
\vert \rho'(t)\vert\lesssim \int e^{-(1-\varepsilon)\vert x\vert}(v^2+y^2+y_x^2)(t,x)dx.
\ee
Now let us recall that by \eqref{Integration0} we have
\[ 
\int_0^\infty\!\!\! \int e^{-(1-\varepsilon)\vert x\vert}(v^2+y^2+y_x^2)(t,x)dxdt <+\infty,
\]
and hence, \eqref{cota_unif} leads to 
\[
|\rho(s_n) -\rho(s_m)| \to 0, \quad n,m\to +\infty,
\] 
for every sequence $(s_n)$ such that $s_n\to +\infty.$ Fix one such sequence $(s_n)$; we have $\rho(s_n)\to \bar\rho$ (this $\bar \rho$ still depends on the chosen sequence $(s_n)$, but we will prove that the whole $\rho(t)$ converges to this limit). Now, once again from \eqref{cota_unif},
\[
|\rho(s_n) -\rho(t)| \lesssim  \int_{t}^{s_n}\!\!\! \int e^{-(1-\varepsilon)\vert x\vert}(v^2+y^2+y_x^2)(t,x)dxdt<+\infty.
\]
Sending $n\to \infty$, and then $t\to +\infty$, we conclude $\rho(t)\to \bar \rho.$ This proves that, whenever $\rho(t)$ stays bounded, it must converge to a final position $\bar \rho.$

\medskip

Finally, we prove \eqref{final_data0_1}. Note that from \eqref{eqn_us}, \eqref{dot_rho_quadratic_2} and after integration by parts we have
\begin{align*}
\frac12\frac{d}{dt} \int  (\widetilde s^2 + \widetilde u^2 + \widetilde u_x^2)(t,x)\sech^{1+\varepsilon}(x)dx&=\int \big(\widetilde{s}_t\widetilde{s}+\widetilde{u}_t\widetilde{u}+\widetilde{u}_{tx}\widetilde{u}_x)\sech^{1+\varepsilon} (x)dx
\\ & \lesssim \int(\sin\widetilde{Q}-\sin(\widetilde{Q}+u))\widetilde{s}\sech^{1+\varepsilon} (x)dx
\\ & \quad +\int (\widetilde{u}\widetilde{s}+\vert \widetilde{u}_x\widetilde{s}\vert+\dot{\rho}\widetilde{u}_x\widetilde{Q}'')\sech^{1+\varepsilon} (x)dx
\end{align*}
Thus, by using again the standard inequality $2ab\leq a^2+b^2$ we deduce \begin{align*}
\frac12\frac{d}{dt} \int  (\widetilde s^2 + \widetilde u^2 + \widetilde u_x^2)(t,x)\sech^{1+\varepsilon}(x)dx& \lesssim\int (\widetilde{u}^2+\widetilde{u}_x^2+\widetilde{s}^2)\sech^{1+\varepsilon} (x)dx
\end{align*}
Using \eqref{QQQ}, \eqref{quadratic_u_yv} and  \eqref{10p3_1}, we simply obtain 
\[
\frac12\frac{d}{dt} \int (\sech^{1+\varepsilon} x )(\widetilde s^2 + \widetilde u^2 + \widetilde u_x^2)(t,x)dx \lesssim  \int e^{-(1-\varepsilon)|x|} (y^2+y_x^2+v^2)(t,x)dx.
\]
Consequently, from \eqref{Integration0} we conclude \eqref{beee} (and therefore, \eqref{final_data0_1}) exactly as it was done in \cite{KMM}.

\medskip

\appendix

\section{Proof of Lemma \ref{THMQzeroA}}\label{BT_kink_parity_appendix}

We follow the ideas in \cite{MP}, but with some important modifications. Let us consider the B\"acklund functionals introduced in \eqref{f1_new} and \eqref{f2_new} with $a=1$ and $\delta=0$ (i.e. $\beta=0$ from \eqref{a(beta)}):
\be\label{setting4p1}
\begin{aligned}
     0=&~{} \widetilde{\mathcal F}_1(\widetilde u,\widetilde s,y,v,0) =  \widetilde Q_x + \widetilde u_{x} -v - \cos\left(\dfrac{\widetilde Q +\widetilde u + y}{2}\right)-\cos\left(\dfrac{\widetilde Q +\widetilde u-y}{2}\right), 
    \\ 
    0=&~{} \widetilde{\mathcal F}_2(\widetilde u,\widetilde s,y,v,0) = \widetilde s -y_{x} - \cos\left(\dfrac{\widetilde Q +\widetilde u + y}{2}\right) +\cos\left(\dfrac{\widetilde Q +\widetilde u-y}{2}\right).
\end{aligned}
\ee
Recall that $ \widetilde Q_t =0$ in the case $\beta=0$ (see \eqref{Qt}). These functionals are well-defined, see Lemma \ref{Parity}, item (d).

\medskip

Let $(y,v)\in H^1_e\times L^2_e$ be small enough given perturbations  (maybe depending on time, but of size uniformly bounded for $t\in\R$). Notice that for any given triplet $(y,v, \widetilde u)\in H^1_e\times L^2_e\times H^1_o$, equation $L_o^2\ni \mathcal{F}_2\equiv 0$ is trivially solvable for $\widetilde s$ and defines a function in $L^2_o$. On the other hand, with a slight abuse of notation, 
\[
\widetilde{\mathcal{F}}_1: H^1_o(\R)\times L^2_o(\R) \times H^1_e(\R)\times L^2_e(\R) \to L^2_e(\R), \quad \widetilde{\mathcal F}_1=\widetilde{\mathcal F}_1(\widetilde u,\widetilde s,y,v),
\]
defines a $\mathcal{C}^1$ functional in a neighborhood of zero and due to Lemma \ref{prkk1} we have $\widetilde{\mathcal{F}}_1(0,0,0,0)\equiv0$. Therefore, in order to conclude the proof it is enough to show that the G\^ateaux derivative of $\mathcal{F}_1$ defines an invertible bounded linear operator with continuous inverse. In fact, notice that linearizing directly on the definition of $\mathcal{F}_1$ above and by using basic trigonometric identities we are lead to solve
\begin{align}\label{u_ODE_Q}
\widetilde u_x =- \sin\left( \frac{\widetilde Q}{2}\right)\widetilde u +f, \, \hbox{ for some }\, f\in L^2_e.
\end{align}
Here, $\sin\left(\frac{\widetilde Q}{2}\right)=\tanh x.$ Now, in order to solve equation \eqref{u_ODE}, we define $\mu_\beta(x)$ to be the solution of \[
\mu_{\beta,x}-\sin\left(\dfrac{\widetilde Q}{2}\right)\mu_\beta=0, \, \hbox{ that is } \quad \mu_\beta(x)=\cosh x.
\]
At this stage it is important to point out that $\mu_\beta(x)$ is an even function. On the other hand, due to the fact that both $\mu_\beta$ and $f$ are even functions, we conclude that there is only one odd function solving \eqref{u_ODE_Q}, which is given by
\begin{align}\label{u_sol_Q}
\widetilde u(x)=\dfrac{1}{\mu_\beta(x)}\int_{0}^x \mu_\beta(z) f(z)dz.
\end{align}
Finally, by using Young's inequality, the explicit form of $\widetilde u$ and the exponential growth of $\mu_\beta$ it is a straightforward checking that
\[
\Vert \widetilde u\Vert_{L^2(\R)}\lesssim \Vert f\Vert_{L^2(\R)}.
\]
We refer to \cite{MP} Section $6$ for a complete proof of the latter inequality in a similar context. Notice that in order to conclude that $\widetilde u\in H^1_o$ it only remains to prove that $\widetilde u_x\in L^2$. Nevertheless, this is a direct consequence of the explicit form of $u$ in \eqref{u_sol_Q} and the previous analysis. Therefore, we conclude the proof by applying the Implicit Function Theorem. \qed

\section{Proof of Lemma \ref{THMQzeroB}}\label{BT_kink_parity_2_appendix}

We follow the guidelines of the proof of Proposition \ref{THMQzeroA}, with minor but essential differences. Once again, we put in the framework of Lemma \ref{Parity}, item (d).

\medskip

Recall the setting of BT in \eqref{setting4p1}. Now we will consider $(\widetilde u,\widetilde s)\in H^1_o\times L^2_o$ be small enough given perturbations. Notice that for any given $(\widetilde u,y)\in  H^1_o \times H^1_e$, equation $\widetilde{\mathcal{F}}_1\equiv 0$ is trivially solvable for $v(\cdot)$ and defines a function in $L^2_e$. On the other hand, 
\[
\widetilde{\mathcal{F}}_2:H^1_e(\R)\times L^2_e(\R)\times H^1_o(\R)\times L^2_o(\R)\to L^2_o(\R),  \quad \widetilde{\mathcal F}_2=\widetilde{\mathcal F}_2(\widetilde u,\widetilde s,y,v),
\]
defines a $\mathcal{C}^1$ functional in a neighborhood of zero and due to Lemma \ref{prkk1} we have $\widetilde{\mathcal{F}}_2(0,0,0,0)\equiv0$. Therefore, linearizing directly on the definition of $\widetilde{\mathcal{F}}_2$ above and by using basic trigonometric identities we are lead to solve 
\begin{align}\label{y_ODE_Q}
y_x = \sin\left( \frac{\widetilde Q}{2}\right)y +f, \, \hbox{ for some }\, f\in L^2_o.
\end{align}
Note that unlike \eqref{u_ODE_Q} now we have a ``$-$'' sign in the right-hand side. As before, in order to solve equation \eqref{y_ODE_Q}, we define $\mu_\beta(x)$ to be the solution of 
\[
\mu_{\beta,x} + \sin\left(\dfrac{\widetilde Q}{2}\right)\mu_\beta=0, \qquad \hbox{ that is } \qquad \mu_\beta(x)=\sech x.
\]
Notice that since $\mu_\beta$ and $f$ are even and odd functions respectively we conclude
\[
\int_\R\mu_\beta(x) f(x)dx=0.
\]
Therefore, solving \eqref{y_ODE_Q} from $-\infty$ to $x$ we conclude that there is only one solution to \eqref{y_ODE_Q} which is given by
\begin{align}\label{y_sol_Q}
y(x)=\dfrac{1}{\mu_\beta(x)}\int_{-\infty}^x \mu_\beta(z) f(z)dz.
\end{align}
Finally, by using Young's inequality, the explicit form of $y$ and the exponential decay of $\mu_\beta$ it is a straightforward checking that
\[
\Vert y\Vert_{L^2(\R)}\lesssim \Vert f\Vert_{L^2(\R)} \quad \hbox{and}\quad \Vert y_x\Vert_{L^2(\R)}\lesssim \Vert f\Vert_{L^2(\R)}.
\]
We refer to \cite{MP} Section $6$ for a complete proof of the latter inequality in a similar context. Therefore, we conclude the proof by applying the Implicit Function Theorem.\qed

\section{Proof of Lemma \ref{conection_WB}}\label{A}

Recall that the wobbling kink is given by \eqref{wobblingK}
\be\label{wobb}
W_\bt(t,x):= 4\arctan e^x  +  4\arctan f,\quad f=\frac{g}{h},
\ee
where
\be\label{gh}
\begin{aligned}
 &g:=\bt(\sinh(x)\cos(\al t) - \sinh(\bt x)),\\
 &h:=\cosh(x)\cosh(\bt x) - \bt\sinh(x)\sinh(\bt x) - \bt\cos(\al t).
 \end{aligned}
\ee
Consequently,
\begin{align*}
W_{\beta,t}=\dfrac{4f_t}{1 + f^2}, \qquad W_{\beta,x}=2\sech(x) + \dfrac{4f_x}{1 + f^2}.
\end{align*}
Moreover, directly from \eqref{breather}
\begin{align*}
B_{\beta,t}= \dfrac{4\alpha^2\beta\cos(\alpha t)\cosh(\beta x)}{\alpha^2\cosh^2	(\beta x)+\beta^2\sin^2(\alpha t)},
\qquad B_{\beta,x}=\dfrac{-4\alpha\beta^2\sin(\alpha t)\sinh(\beta x)}{\alpha^2\cosh^2(\beta x)+\beta^2\sin^2(\alpha t)},
\end{align*}
and
\begin{align*}
\sin\left(\dfrac{B_\beta}{2}\right)=\dfrac{2\alpha\beta \sin(\alpha t)\cosh(\beta x)}{\alpha^2\cosh^2(\beta x)+\beta^2\sin^2(\alpha t)}, 
\qquad \cos\left(\dfrac{B_\beta}{2}\right)=\dfrac{\alpha^2\cosh^2(\beta x)-\beta^2\sin^2(\alpha t)}{\alpha^2\cosh^2(\beta x)+\beta^2\sin^2(\alpha t)}.
\end{align*}
On the other hand,
\[
\begin{aligned}
\sin\left(\dfrac{W_\beta}{2}\right)=&~{} \frac{1-f^2}{1+f^2}\sech(x) - \frac{2f\tanh(x)}{1+f^2}, \\
 \cos\left(\dfrac{W_\beta}{2}\right)= &~{}  -\frac{1-f^2}{1+f^2}\tanh(x) - \frac{2f\sech(x)}{1+f^2}.
\end{aligned}
\]
Then we recast \eqref{b1_new}-\eqref{b2_new} as follows,
\begin{align*}
&(\alpha^2\cosh^2(\beta x)+\beta^2\sin^2(\alpha t))\Big(2(1+f^2)\sech(x) + 4f_x\Big) - 4\alpha^2\bt\cos(\alpha t)\cosh(\beta x)(1+f^2)\\ 
&\qquad = 2(\alpha^2\cosh^2(\beta x) - \beta^2\sin^2(\alpha t)) \Big((1-f^2)\sech(x) - 2f\tanh(x)\Big),
\end{align*}
and
\begin{align*}
&4f_t(\alpha^2\cosh^2(\beta x)+\beta^2\sin^2(\alpha t)) + 4\al\bt^2\sin(\alpha t)\sinh(\beta x)(1+f^2)\\
&\qquad =- 4\alpha\bt\sin(\alpha t)\cosh(\beta x) \Big((1-f^2)\tanh(x) + 2f\sech(x)\Big),
\end{align*}
or in terms of $g,h,$ we get
\be\label{sist2}
\begin{aligned}
&(\alpha^2\cosh^2(\beta x)+\beta^2\sin^2(\alpha t))\Big(2(h^2+g^2)\sech(x) + 4(g_xh - gh_x)\Big)\\
&\qquad\qquad\qquad\qquad\qquad\qquad- 4\alpha^2\bt\cos(\alpha t)\cosh(\beta x)(h^2+g^2)\\ 
&= 2(\alpha^2\cosh^2(\beta x) - \beta^2\sin^2(\alpha t)) \Big((h^2-g^2)\sech(x) - 2gh\tanh(x)\Big),\\
\end{aligned}
\ee
\be\label{sist3}
\begin{aligned}
&4(g_th - gh_t)(\alpha^2\cosh^2(\beta x)+\beta^2\sin^2(\alpha t)) + 4\al\bt^2\sin(\alpha t)\sinh(\beta x)(h^2+g^2)\\
&= -4\alpha\bt\sin(\alpha t)\cosh(\beta x) \Big((h^2-g^2)\tanh(x) + 2gh\sech(x)\Big).
\end{aligned}
\ee
Now, having in mind that from \eqref{gh},
\begin{align*}
&g_x=\bt(\cosh(x)\cosh(\bt x) - \bt\cosh(\bt x)),\\ 
&g_t=-\al\bt \sinh(x)\sin(\al t),\\
&h_x=\al^2\sinh(x)\cosh(\bt x),\\
&h_t=\al\bt\sin(\al t),
\end{align*}
substituting in \eqref{sist2}-\eqref{sist3} and after easy manipulations, we conclude and the proof is complete. \qed

\medskip

\section{Proof of Remark \ref{convergencia_a_WK}}\label{B}

First of all, notice that by standard trigonometric identities we have
\[
\tan(\Theta-\overline{\Theta})=\dfrac{\Upsilon-\overline{\Upsilon}}{1-\Upsilon\overline{\Upsilon}} \quad \hbox{ where } \, \quad \Upsilon:=\left(\dfrac{\beta a_v+i\alpha a_v+1}{\beta a_v+i\alpha a_v-1}\right)\dfrac{e^x-e^{\gamma[\beta(x-vt)-i\alpha(t-vx)]}}{1+e^{x+\gamma[\beta(x-vt)-i\alpha(t-vx)]}}.
\]
Thus, after some easy manipulations we conclude that $\tan(\Theta-\overline{\Theta})=\tfrac{A_1}{A_2}$, where
\begin{align*}
A_1&=i(a_v^2-1)\cosh(x)\sin(\gamma\alpha(t-vx))
\\ & \quad -2ia_v\alpha\cos\big(\gamma\alpha(t-vx)\big)\sinh(x)-2ia_v\alpha\sinh\big(\gamma\beta(tv-x)\big)
\end{align*}
and
\begin{align*}
A_2&=-2a_v\beta \cos\big(\gamma\alpha(t-vx)\big)
\\ & \quad +\cosh(\gamma v\beta  t)\Big((1+a_v^2)\cosh(x)\cosh(\gamma\beta x)-2a_v\beta\sinh(x)\sinh(\gamma\beta x)\Big)
\\ & \quad +\sinh(\gamma v\beta t)\Big(2a_v\beta\sinh(x)\cosh(\gamma\beta x)-(1+a_v^2)\cosh(x)\sinh(\gamma\beta x)\Big).
\end{align*}
Notice that, if $v=0$, then $W_{\beta,v}\equiv W_\beta$, where $W_\beta$ is given by \eqref{wobblingK}.

\end{document}